\documentclass[11pt]{amsart}
\usepackage{graphicx, amssymb, amsmath,amsthm}

\newcommand{\mb}{\mathbb}
\newcommand{\mc}{\mathcal}

\def \a{\alpha} \def \b{\beta} \def \g{\gamma} \def \d{\delta}
\def \t{\theta}   \def \e{\epsilon}
\def \s{\sigma} \def \l{\lambda}  
 \def \D{\Delta}  
\def \k{\kappa}  \def \G{\Gamma}
\def \di{\mathrm{dist}} \def \spa{\mathrm{span}}
 \def\an{\measuredangle}
\def \L{\Lambda}

\newtheorem{theorem}{Theorem}[section]

\newtheorem{remark}{Remark}[section]

\newtheorem{lemma}[theorem]{Lemma}
\newtheorem{prop}[theorem]{Proposition}

\newtheorem{cor}[theorem]{Corollary}
\newtheorem{definition}[theorem]{Definition} 

\numberwithin{equation}{section}

\begin{document}

 \keywords{SRB measure, infinite dimensional dynamical systems, partial hyperbolicity, absolute continuity, stable foliation, observable sets.  \\  This work is partially supported by grants from CNSF and NSF}


 \author{Zeng Lian} \address[Zeng Lian] {Department of Mathematical Sciences\\ Loughborough University\\
    Loughborough, Leicestershire, LE11 3TU, UK} \email[Z.~Lian]{Z.Lian@lboro.ac.uk}

\author{Peidong Liu} \address[Peidong Liu] {School of Mathematical Sciences\\ Peking University\\
    Beijing, 100871, P. R. China} \email[P.~Liu]{lpd@pku.edu.cn}

\author{Kening Lu} \address[Kening Lu] {  Department of Mathematics\\
 Brigham Young University\\
 Provo, Utah 84602, USA}
\email[k.~Lu]{klu@math.byu.edu}

\title[SRB Measures for  A Class of Partially
 Hyperbolic Attractors in Hilbert spaces]{SRB Measures for  A Class of Partially
 Hyperbolic Attractors in Hilbert spaces}

\pagestyle{plain}

\begin{abstract} { In this paper, we study the existence of SRB measures and their properties for infinite dimensional dynamical systems in a Hilbert space. We show several results including (i) if the system has a partially hyperbolic attractor with nontrivial finite dimensional unstable directions, then it has at least one SRB measure; (ii) if
 the attractor is uniformly hyperbolic and the system is topological mixing and the splitting is H\"older continuous, then there exists a unique SRB measure which is mixing; (iii)  if
 the attractor is uniformly hyperbolic and the system is non-wondering and and the splitting is H\"older continuous,  then there exists at most finitely many SRB measures;
 (iv) for a given hyperbolic measure, there exist at most countably many ergodic components whose  basin contains an observable set.
} \end{abstract}

\maketitle

\section{Introduction}
In this  paper, we study the existence of SRB measures and their properties for infinite dimensional dynamical systems that are usually generated by dissipative partial differential equations. This research is a continuation of the study initiated by Ruelle \cite{R} and Ma$\tilde n\acute e$ \cite{M} in 80's, when Oseledets' Multiplicative Ergodic Theorem and Pesin's stable and unstable manifold theorems were established. The problems we study here were posed by Eckmann and Ruelle \cite{ER} in 1985.  It is also one of topics on dynamical systems proposed by Young \cite{Young} in a recent survey article.

\vskip0.05in
SRB measures are the invariant measures introduced by Sinai, Ruelle and Bowen in the 70s to describe dynamical systems in terms of the average or statistical properties of their``typical'' orbits. The``typical'' here means for a set of points of full measure. SRB measures can also be ``observed'' in the sense that the set of points, whose orbits have a common asymptotic distribution or their distributions converge to the measure, has positive volume ( Lebesgue measure).

\vskip0.05in
Let $f$ be a diffeomorphism on a compact manifold $M$ and $\mu$ be an $f$-invariant ergodic probability measure. By Birkhoff's ergodic theorem, for each continuous  function $\phi: M \to \mathbb{R}$, the following limit holds on a full $\mu$-measure subset of  $M$,
\begin{equation}\label{limit}
\lim_{n\to\infty} \frac{1}{n}\sum_{i=0}^{n-1} \phi\big(f^i(x)\big)=\int_M\phi \;d\mu.
\end{equation} However, this full $\mu$-measure set may not contain any subset with positive Lebesque measure. When there exists a positive Lebesque measure set $B\subset M$ such that the above limit (\ref{limit}) holds for each  $x\in B$ and any continuous  function $\phi: M\to\mathbb{R}$, the invariant measure $\mu$ is called a Sinai-Ruelle-Bowen (SRB) measure.
Such a measure was first studied by Sinai for Anosov diffeomorphisms \cite{Sinai}, and later by Ruelle
and Bowen \cite{Ruelle, Bowen, BR} for Axiom A diffeomorphisms and flows.

\vskip0.05in
It is possible that an SRB measure is supported on a periodic sink. To exclude such cases, one often imposes the condition on the existence of  positive Lyapunov exponents of $f$ with respect to $\mu$.

\vskip0.05in
An alternative definition of SRB measures was given by Eckmann and Ruelle \cite{ER},
also see Young \cite{Young2}.
Consider a $C^2$ diffeomorphism $f$ on a finite dimensional manifold $M$.  A subset $\Lambda \subset M$ is called an {\it attractor} of $f$ with basin $U$,  $U$ being an open neighborhood of $\Lambda$, if  $\Lambda$ is compact, $f \Lambda =\Lambda$, and
 \[\cap_{n\ge0}f^n(U)=\Lambda.
 \] Let $\mu$ be an $f$-invariant probability measure supported on $\Lambda$. Then, by Oseledets' multiplicative ergodic theorem, we have that for $\mu$ almost every $x\in \Lambda$ there are Lyapunov exponents $\lambda_1(x)> \lambda_2(x)> \cdots >\lambda_{p(x)}(x)$ and an invariant splitting of tangent space
 \[
 T_x M=\oplus E_i(x),
 \] where $E_i$ is the Oseledets space associated with the Lyapunov exponent $\lambda_i$.

\vskip0.05in
 Suppose that $f$ has a positive Lyapunov exponent. Then, from Pesin's invariant manifold theorem,   $f$ has an unstable manifold $W^u(x)$ defined for $\mu$ almost all $x\in\Lambda$. The invariant measure $\mu$ is called an {\bf SRB measure} if $\mu$ has absolutely continuous conditional measures on unstable manifolds with respect to  the Lebesgue measures on these submanifolds induced by the inherited Riemannian structure. In this case, one also likes to know whether it is a physical measure, i.e., there is a subset
 $B\subset U$ with positive Lebesgue measure of the whole manifold such that
\begin{equation}\label{limit2}
\lim_{n\to\infty} \frac{1}{n}\sum_{i=0}^{n-1} \phi\big(f^i(x)\big)=\int_M\phi \;d\mu,\quad\text{for } x\in B
\end{equation} for any continuous function $\phi: M \to \mathbb{R}$. This is true when $f$ has both positive Lyapunov exponents and negative Lyapunov exponents, but no zero Lyapunov exponents, following from the absolute continuity of the stable foliation,  see \cite{P}, \cite{KS}, \cite{Pu}. Note that the property (\ref{limit2}) does not follow from Birkhoff's ergodic theorem since $\mu$ is supported on $\Lambda$ and its support may have zero Lebesgue measure.

\vskip0.05in
In the recent decades, much progress has been made on the existence of SRB measures for finite dimensional dynamical systems beyond uniform hyperbolicity. Here we mention some of the works in the directions of non-uniformly hyperbolic systems and partially hyperbolic systems.

\vskip0.05in
For non-uniformly hyperbolic systems,  the results on SRB measures were obtained by Pesin\cite{P} for diffeomorphisms preserving smooth measures, Jakobson\cite{J} for the quadratic family $f_a(x) = 1-ax^2$,  Benedicks and Young \cite{BY} for H\'enon attractors, and Wang and Young for  rank-one attractors with applications to  slow-fast differential equations \cite{GY} and two dimensional differential equation with dissipative homoclinic loops \cite{WO}.

\vskip0.05in
For partially hyperbolic systems, Pesin and Sinai \cite{PS} considered the diffeomorphism $f$ with an attractor having a dominated splitting into  strong-unstable subbundle (uniformly expanding) and  a center-stable subbundle. They proved that if the center direction is nonuniformly contracting, then $f$ has an SRB measure.
Bonatti and Viana \cite{BV} studied the partially hyperbolic diffeomorphism $f$  having an attractor with an invariant splitting
of the tangent bundle into a strong-unstable subbundle (uniformly expanding) and a center subbundle, dominated by the unstable bundle, and proved that if the central direction  is mostly contracting, then $f$ has SRB measures. In \cite{ABV}, Alves, Bonatti, and Viana considered the diffeomorphism $f$ having a compact invariant subset $K$ which has a dominated splitting into a strong-stable subbundle (uniformly contracting) and a center-unstable subbundle. They proved that if the diffeormophism $f$  has positive central exponents on a subset of $K$ with positive Lebesgue measure, then $f$ has SRB measures. Cowieson and Young \cite{CY} studied the diffeomorphism $f$ having an attractor with a dominated splitting  into a strong-stable subbudle (uniformly contracting), a center subbundle, and a strong unstable subbundle(uniformly expanding) and proved that if the center bundle is one dimensional, then $f$ has SRB measures. In \cite{LiuL}, Liu and Lu considered the diffeomorphism $f$ with an attractor having an invariant splitting into a strong unstable subbundle (uniformly expanding) and a center-stable subbundle and proved that if the center direction is not expanding, then $f$ has an SRB measure.

\vskip0.05in
A natural question is what one can say about SRB measures for infinite dimensional dynamical systems. Due to the absence of a notion of Lebesgue measure in infinite dimensional spaces, there is no direct generalization of SRB measures to infinite dimensional systems. However, the definition of SRB measures given by Eckmann and Ruelle \cite{ER} can be extended to infinite dimensional systems when the relevant unstable manifolds are defined and finite dimensional. If there is such a measure, what is the analog of the observability result? Other challenges are: (i) the infinite dimensional dynamical systems generated, for example, by parabolic PDEs are not invertible and (ii) the phase space is not locally compact.

\vskip0.05in
The only result that we are aware of for infinite dimensional systems is the recent work by Lu, Wang, and Young \cite{LWY} for parabolic PDEs undergoing a Hopf bifurcation driven by a periodic forcing.  In the absence of forcing, a limit cycle emerges from generic supercritical Hopf bifurcations. The periodic forcing turns this limit cycle into an attractor lying on a two-dimensional center manifold $W^c$. They showed that under certain conditions this attractor has an SRB measure  and a positive Lyapunov exponent almost everywhere. They also proved the absolute continuity of a codimension-$2$ strong stable foliation $W^{ss}$. As a result, for Leb-a.e. initial points from a surface transversal to the leaves of the $W^{ss}$-foliation, roughly parallel to $W^c$, the long term statistic behaviors of the solutions  are described by the SRB measure.

\vskip0.05in

In this  paper, we study the existence of SRB measures and their properties for infinite dimensional dynamical systems. To be more precise, we consider a $C^2$ map $f$ on a separable Hilbert space $\mb H$. Suppose that $f$ has an attractor $\Lambda$ with basin  $U$ which is an open neighborhood of $\Lambda$ in  $\mb H$. We assume that  $f$ and its Fr\'{e}chet derivative $Df_x$ are injective for all $x\in \Lambda$ and the following holds
\begin{equation}\label{compactness}
\kappa(x):=\limsup_{n\to\infty}\frac1n\log\|Df^n(x)\|_{\kappa}<0,
 \end{equation} where $\|\cdot\|_{\k}$ is the Kuratowski measure of noncompactness of an operator. The typical such maps $f$ are the solution operators of dissipative parabolic PDEs or damped wave equations on bounded domains.

\vskip0.05in
We call  $f|_{\Lambda}$  to be  {\bf partially hyperbolic} if the following holds:
for every $x \in \Lambda$ there is a splitting
\[
H=E_x^u \oplus E_x^{cs}
\]
which depends continuously on $x \in \Lambda$  with $\dim E_x^u >0$ and satisfies that for every $x \in \Lambda$
\[
Df_{x} E_x^u=E_{fx}^u, \ \ \ Df_{x} E_x^{cs} \subset E_{fx}^{cs}
\]
and
\begin{equation}\label{E:PartialHyperbolic}
\left\{\begin{array}{ll}
|Df_{x} \xi| \geq e^{\lambda_0} |\xi|, \ \ \ &\forall\ \xi \in  E_x^u, \\
|Df_{x}\eta| \leq  |\eta|, \ \ \ &\forall\ \eta \in  E_x^{cs},
\end{array}
\right.
\end{equation}
where $\lambda_0>0$ is a constant.

\vskip0.05in
We say $f|_{\L}$ is {\bf uniformly hyperbolic} if $f|_{\L}$ is partially hyperbolic and it satisfies that
\begin{equation}\label{E:PartialHyperbolic}
\left\{\begin{array}{ll}
|Df_{x} \xi| \geq e^{\lambda_0} |\xi|, \ \ \ &\forall\ \xi \in  E_x^u, \\
|Df_{x} \eta| \leq  e^{-\l_0}|\eta|, \ \ \ &\forall\ \eta \in  E_x^{cs},
\end{array}
\right.
\end{equation}
for some $\l_0>0$.

\vskip0.05in
Our main results can be summarized as follows. The precise statement will be given in next section.

\vskip0.1in
\noindent{\bf Theorem}. {\it Let $f$ be a $C^2$ map on a separable Hilbert space $\mb H $. Suppose that $f$ has an attractor $\Lambda$  with basin $U$. We assume that  $f$ and its Fr\'{e}chet derivative $Df_x$ are injective for all $x\in \Lambda$ and \ref{compactness} holds. Then, we have the following
\begin{itemize}
\item[(i)] {\bf Existence:} If $f|_\L$ is partially hyperbolic, then there exists at least one SRB measure.
\item[(ii)] {\bf Countability:}
    There are at most countably many  {\it ergodic hyperbolic SRB measures} whose supports are contained in $\Lambda$.
\item[(iii)] {\bf Uniqueness and Mixing:}
If $f|_\L$ is uniformly hyperbolic, topologically mixing and the splitting is H\"older continuous on $x\in \Lambda$, then there exists a unique SRB measure which is mixing.
\item[(iv)] {\bf Finiteness:} If $f|_{\L}$ is uniformly hyperbolic, non-wandering, and the splitting is H\"older continuous on $x\in \Lambda$, then there are finitely many ergodic SRB measures.
\item[(v)] {\bf Observability:}
 For a given hyperbolic SRB measure $\mu$, up to a $\mu$-null set, there exist at most countably many ergodic attractors, $(K_i,\nu_i)$, whose basin $U'_i$ contains an observable set and $\nu_i$ is the normalization of $\mu|_{K_i}$.
\end{itemize}
}

\vskip0.1in

\noindent{\bf Remark:} { When $f$ is a diffeomorphism on a compact manifold, property (i)
was proved in \cite{LiuL}, property (ii) and (v) was proved in \cite{Pu}, and property (iii) and (iv) was proved in \cite{Bowen74}. A sufficient condition to insure that the hyperbolic splitting is H\"older continuous on $x\in \Lambda$ in infinite dimensional space is that $(f|_\L)^{-1}$ is Lipschitz continuous. Note that the solution operators generated by dissipative wave equations satisfy this Lipschitz continuity.
Property (v), the observability, follows from the absolute continuity of stable foliations. Once one has an SRB measure $\mu$, using a recent result by Li and Shu \cite{LSh}, Pesin's entropy formula holds for this infinite dimensional system
\[h_{\mu}(f)=\int \sum m_i(x)\l_i^+(x) du (x), \]
where $h_\mu(f)$ is the metric entropy of $(f, \mu)$ and $m_i(x)$ is the multiplicity of the positive Lyapunov exponent $\l_i(x)$. Thus, with a positive entropy, Lian and Young's result implies that the hyperbolic system has a horseshoe.
The existence of SRB measures also yields that the partially hyperbolic attractor is chaotic and contains a full weak horseshoe following from the recent result by Huang and Lu \cite{HuangLu} since the entropy is positive.

By a {\it full horseshoe of two symbols} we mean that there exist  subsets $U_1,U_2$ of Hilbert space $\mb H $  such that the following
properties hold
\begin{enumerate}
\item $U_1$ and $U_2$ are non-empty bounded closed subsets of $\mb H $ and
$d(U_1,U_2)>0$.

\item  there exists constant $b>0$ and $J\subset  \mathbb{N}_0$ such that the limit
\[\lim_{m\rightarrow +\infty}\frac{1}{m}|J\cap \{0,1,2,\cdots, m-1\}|\] exists and is larger than or equal to $b$
(positive density), and for any $s\in
\{1,2\}^{J}$, there exists $x_s\in \Lambda$ with
$f^j(x_s)\in U_{s(j)}\cap \Lambda$ for any
$j\in J$.
\end{enumerate}

Finally, the unstable dimensions for dissipative parabolic equations and damped wave equations on a bounded domain are finite.

\vskip0.05in
 We organize this paper as follows. In Section \ref{S:Setting}, we first give basic concepts and notations. Then, we state our main results. The proof of the existence of SRB measure is given in Section \ref{S:ExistSRB}. In Section \ref{S:ErgodicAttractors}, we formulate and prove the absolute continuity of stable foliations, which is the main technical result used in the proof of Theorem \ref{T:SRBCont} and \ref{T:ErgodicAttractor}. In Section \ref{S:UniFiniSRB}, we prove the uniqueness of SRB measure when the system is topologically mixing and the finiteness of SRB measures when the system is non-wandering. Lyapunov charts, symbolic dynamics and thermodynamical formalism, equilibrium states, and Ruelle inequality and entropy formula are given in the appendices.

\section*{Acknowledgment} We thank Professor Lai-Sang Young for her valuable comments and suggestions on this manuscript.

 \section{Settings and Main Results}  \label{S:Setting}
 Let $(\mb H,<\cdot,\cdot>)$ be a separable Hilbert space, $f:\mb H \to \mb H$ be a $C^2$ map. And let $Df_x$ be the Fr\'{e}chet derivative of $f$ at point $x\in \mb H$.
The conditions below are assumed throughout:
 \begin{itemize}
 \item[C1)] $f$ is injective;
 \item[C2)] There exist an $f$-invariant compact set $\Lambda$, on which $Df_x$ is (i) injective,\\
  and (ii) for all $x\in \Lambda$ $$\kappa(x):=\limsup_{n\to\infty}\frac1n\log\|Df^n(x)\|_{\kappa}<0,$$ where $\|\cdot\|_{\k}$ is the Kuratowski measure of noncompactness of an operator;
 \item[C3)] $\Lambda$ is an attractor with basin $U$ i.e. $U$ is an open neighborhood of $\Lambda$ and $$\cap_{n\ge0}f^n(U)=\Lambda.$$
 \end{itemize}

 Recall that, for a linear operator $T$, $\|T\|_\k$ is defined to be the infimum of the set of numbers $r>0$ where $T(B)$, $B$ being the unit ball, can be covered by a finite number of balls of radius $r$. Since $\|T_2\circ T_1\|_\k\le\|T_2\|_\k\|T_1\|_\k$ and $\|Df\|_{\k}\le \|Df\|$ is uniformly bounded on $\L$, the limit in the definition of $\k(x)$ exists for any $x\in \L$ and is a measurable function. And by definition, $\|T\|_{\k}=-\infty$ if $T$ is compact. The typical such maps $f$ are the solution operators of dissipative parabolic PDEs or damped wave equations on bounded domains.

 \begin{remark}\label{R:WeakerCondition}
 Note that some well-known results follow from the assumptions above immediately:
 \begin{itemize}
 \item[(i)] By the compactness of $\L$, there is always an $f$-invariant probability measure supported on $\L$,  which we denote by $\mu$.
  \item[(ii)] By applying the Multiplicative Ergodic Theorem, there is a full measure set $\L'\subset \L$, on which the notion of Lyapunov exponents can be introduced, we refer the reader to \cite{R} and \cite{LL} for details and will state a simplified version in a later section for completeness.
\item[(iii)] Both conditions C2) (ii) imply that, for all $x\in\L'$, there are at most finitely many non-negative Lyapunov exponents.
\item[(iv)] Attractors are important because they capture the asymptotic behavior of large
sets of orbits. In general, $\Lambda$ itself tends to be relatively small (compact and of finite
Hausdorff dimension) while its attraction basin, which by definition contains an open set, is
quite visible in the phase space. Notice that our attractors are not necessarily
global attractors in the sense of \cite{Hale} and \cite{Te}.
 \end{itemize}
 \end{remark}

 With the Lyapunov exponents, we define an $f$-invariant measure $\mu$ to be {\bf hyperbolic} if for $\mu$-a.e. $x\in \L$, there is no zero Lyapunov exponent and there is at least one positive Lyapunov exponent.
 For given $x\in \Lambda$, we define the stable and unstable set of $x$ as the following:
 \begin{align*}
 W^s(x)&=\{y\in \mb H | |f^n(y)-f^n(x)|\to 0  \text{ exponentially  as } n\to +\infty\},\\
  W^u(x)&=\{y\in \mb H | f^{-n}(y) \text{ exists }, \forall n\in \mb N,  |f^{-n}(y)-f^{-n}(x)|\to 0  \text{ exponentially  as } n\to +\infty\}.
 \end{align*}
 By C3), we have $ W^u(x) \subset \Lambda$ for all $x\in\Lambda$.

 In general, $ W^s(x)$ may not be a manifold, and $W^u(x)$ may be an immersed manifold rather than an embedded one. To avoid these disadvantages, one can study the local invariant manifolds defined below:
 \begin{align*}
 & W^{s}_{r(x)}(x)=\{y\in B(x,r(x))|\  |f^n(y)-f^n(x)|e^{-n\l(x)}\le C(x), n\ge 0\},\\
 & W^{u}_{r(x)}(x)=\{y\in B(x,r(x))|\  f^{-n}(y) \text{ exists for all } n\in \mb N, |f^{-n}(y)-f^{-n}(x)|e^{-n\l(x)}\le C(x), n\ge 0\},
 \end{align*}
   where $B(x,r)$ is the $r$-ball centered at $x$, $r$ and $C$ are positive tempered functions and $\l$ is an $f$-invariant function. In particular, given a hyperbolic measure $\mu$, there are measurable tempered functions $r,C: \L\to \mb R^+$ such that for $\mu$-a.e. $x\in \L$, $ W^{u}_{r(x)}$ and $ W^{s}_{r(x)}$ are embedded discs with well controlled distortions. Moreover, $ W^{u}_{r(x)}$ has finite dimension, and $ W^{s}_{r(x)}$ has finite co-dimension. These results are well established in finite dimension case, for which we refer to \cite{P} and \cite{Pu} for details. The results used in this paper are mainly due to \cite{LY} and \cite{LL}. \\

   Under the current setting, one quick observation is that, for $\mu$-a.e. $x$
   $$ W^u(x)=\bigcup_{n=0}^{+\infty}f^n_{f^{-n}(x)}( W^{u}_{r(f^{-n}(x))}(f^{-n}(x))),$$
    which implies that $ W^u(x)$ is an immersed manifold, since $ W^u_{r(x)}$ is finitely dimensional and $f$ is injective and differentiable.

    By the definition of $ W^u(x)$, it is obvious that $\{ W^u(x)\}_{x\in \L}$ form a foliation,  i.e. $ W^u(x)\cap  W^u(y)\neq \emptyset$ if and only if $ W^u(x)= W^u(y)$. So, up to a $\mu$-null set, $\bigcup_{x\in \L} W^u(x)$ form a partition of $\Lambda$. Unfortunately, this partition may be not measurable. We need to introduce the following concepts:

 \begin{definition}\label{D:Subordinate}
 Let $\mu$ be an $f$-invariant  Borel probability measure on $\Lambda$. A measurable partition ${\mathcal P}$
of $\Lambda$ is said to be {\it subordinate to the unstable manifolds} with respect to $\mu$ if, for $\mu$-a.e. $x \in \Lambda$, one has that ${\mathcal P}(x) \subset W^u(x)$ (here ${\mathcal P}(x)$ denotes the element of ${\mathcal P}$ which contains $x$) and it contains an open neighborhood of $x$ in $ W^u(x)$ (endowed with the submanifold topology).
\end{definition}

\begin{definition}\label{D:SRB}
 An $f$-invariant  Borel  probability measure $\mu$ on $\Lambda$ is called an {\bf SRB measure} if for every measurable partition ${\mathcal P}$
of $\Lambda$  subordinate to the unstable manifolds with respect to $\mu$ one has
$$
\mu_x^{\mathcal P} \ll {\rm Leb}_x
$$
for $\mu$-a.e.  $x \in \Lambda$, where $\mu_x^{\mathcal P}$ denotes the conditional measure of $\mu$ on ${\mathcal P}(x)$ and Leb$_x$ denotes the Lebesgue measure on $W^u(x)$ induced by the inner product of $H$.
\end{definition}

    \begin{theorem}\label{T:SRBExist}
    If $f|_\L$ is partially hyperbolic, then there exists at least one SRB measure of $f$ with support in $\L$.
    \end{theorem}

    \begin{theorem}\label{T:SRBCont}
    There are at most countably many  {\it ergodic hyperbolic SRB measures} of $f$ with support in $\L$.
    \end{theorem}

\begin{theorem}\label{T:SRBUnique}
If $f|_\L$ is uniformly hyperbolic, topologically mixing and the splitting is H\"older continuous on $x\in \Lambda$, then there exists a unique SRB measure of $f$ with support in $\L$, which is mixing.
\end{theorem}

If we assume $f|_{\L}$ is non-wondering which is weaker than the  topologically mixing, there are only finitely many  ergodic SRB measures of $f$ with support in $\L$. Recall that

\begin{definition}\label{D:NonWandering}
A point $x\in\L$ is $f|_{\L}$- {\em non-wandering} if for any open neighborhood $U$ of $x$
$$(U\cap \L)\cap \cup_{n>0}f^{n}(U\cap \L)\neq \emptyset.$$
And $f|_{\L}$ is {\em non-wandering} if every $x\in \L$ is $f|_{\L}$- {\em non-wandering}.
\end{definition}

We have the following finiteness on number of  SRB measures.
\begin{theorem}\label{T:SRBFinite}
If $f|_{\L}$ is uniformly hyperbolic, non-wandering, and the splitting is H\"older continuous on $x\in \Lambda$, then there are finitely many ergodic SRB measures of $f$ with support in $\L$.
\end{theorem}

In finite dimensional spaces, the SRB measure is a physical measure, that is, it  can  be ※observed§
in the sense that the set of points, whose orbits have a common asymptotic distribution or their distributions converge to the measure, has positive volume ( Lebesgue measure). Positive Lebesgue measure sets are used to characterize observable  events.

Due to the absence of a notion of Lebesgue measure in infinite dimensional spaces, we take a notion which was first used by Lu, Wang, and Young \cite{LWY}. The challenges are: (i) the infinite dimensional dynamical systems generated, for example, by parabolic PDEs are not invertible and (ii) the phase space is not locally compact.We introduce the following concept of observable set in infinite dimensional space.

 \begin{definition}\label{D:Observable}
 We call a subset $K$ of a separable Banach space $X$ being $observable$ if the following are satisfied:
\begin{itemize}
\item[1)] $K$ is a Borel set;
\item[2)] There exist an $x_0\in X$, a splitting $X=F\oplus E$ with $\ dim E<\infty$, and an open set $U\subset F$, and $\d>0$ such that for any $x\in U$ and $E'$ with $dim\ E'= dim\ E$ and $dist(E',E)<\d$, $m_{Exp_{x_0}(\{x\}+E')}(Exp_{x_0}(\{x\} +
    E')\cap K)>0$.
\end{itemize}
Here $Exp_{x_0}(\cdot)=x_0+\cdot$ is an affine map and $m_{Exp_{x_0}(\{x\}+E')}$ is a Lebesgue measure on finitely dimensional hyperplane $Exp_{x_0}(\{x\}+E')$, and
$$dist(E',E)=\max\left\{\sup_{v\in E'\cap S}\inf_{w\in E\cap S}|v-w|,\sup_{v\in E\cap S}\inf_{w\in E'\cap S}|v-w|\right\}$$ is the Kato's gap metric of subspaces, where $S$ is the unit sphere in $X$.
\end{definition}
\begin{definition}\label{D:ErgodicAttractor}
An ergodic attractor for $f$ is an $f$-invariant Borel set $K\subset H$ together with an $f$-invariant probability measure $\nu$ on $K$ such that for some observable set $K'\subset H$ the following holds
\begin{itemize}
\item[i)] For any $x\in K'$, $dist(f^n(x),K)\to 0$ as $n\to \infty$, i.e. $K'\subset Basin(K)$;
\item[ii)] $\nu$ is $f$-ergodic;
\item[iii)] For any continuous function $\phi:H\to \mb R$ and for every point $x\in K'$,
$$\lim_{n\to\infty}\frac1n\sum_{i=0}^{n-1}\phi(f^i(x))=\int_{K}\phi(x)\nu(dx)$$.
\end{itemize}
\end{definition}

\begin{theorem}\label{T:ErgodicAttractor}
 For a given hyperbolic SRB measure $\mu$ with support in $\L$, up to a $\mu$-null set, there exist at most countably many ergodic attractors, $(K_i,\nu_i)$, whose basin $K'_i$ contains an observable set and $\nu_i$ is the normalization of $\mu|_{K_i}$.
\end{theorem}

A consequence of Theorem \ref{T:ErgodicAttractor} is the following:
Let $h_{\nu_i}(f)$ be the metric entropy of $f$ with respect to $\nu_i$.
Then, by applying results from \cite{LSh} and \cite{LY}, one obtains the following corollary:
\begin{cor}\label{C:EntropyAndHorseshoe}
For each ergodic attractor $(K_i,\nu_i)$, letting $(\lambda_k, m_k)$ be its Lyapunov spectrum, one has
$$h_{\nu_i}(f)=\sum m_k\l_k^+ >0, $$
and for any $\e>0$, there exists a horseshoe $\tilde K_i\subset K_i$ such that
$$h_{top}(f|_{\tilde K_i})>h_{\nu_i}(f)-\e,$$
where $h_{top}(f)$ is the topological entropy of map $f$.
\end{cor}
 And the next result follows from \cite{HuangLu}.
\begin{cor}\label{T:WeakHorseshoe}
If $f|_\L$ is partially hyperbolic, then $f$ is chaotic on $\L$ and has a full horseshoe of two symbols.
\end{cor}

\noindent{\bf Remark:} With a positive entropy, Lian and Young's result implies that the hyperbolic system has a horseshoe.
The existence of SRB measures also yields that the partially hyperbolic attractor is chaotic and contains a full weak horseshoe following from the recent result by Huang and Lu \cite{HuangLu} since the entropy is positive.

By a {\it full horseshoe of two symbols} we mean that there exist  subsets $U_1,U_2$ of Hilbert space $\mb H $  such that the following
properties hold
\begin{enumerate}
\item $U_1$ and $U_2$ are non-empty bounded closed subsets of $\mb H $ and
$d(U_1,U_2)>0$.

\item  there exists constant $b>0$ and $J\subset  \mathbb{N}_0$ such that the limit
\[\lim_{m\rightarrow +\infty}\frac{1}{m}|J\cap \{0,1,2,\cdots, m-1\}|\] exists and is larger than or equal to $b$
(positive density), and for any $s\in
\{1,2\}^{J}$, there exists $x_s\in \Lambda$ with
$f^j(x_s)\in U_{s(j)}\cap \Lambda$ for any
$j\in J$.
\end{enumerate}

Finally, the unstable dimensions for dissipative parabolic equations and damped wave equations on a bounded domain are finite.

\section{Proof of Theorem \ref{T:SRBExist}.} \label{S:ExistSRB}

In this section, we assume that $f|_\L$ is uniformly partially hyperbolic and prove Theorem  \ref{T:SRBExist}, i.e., the existence of SRB measure.
\subsection{ Unstable manifold theorem for partially hyperbolic system}\label{S:SUManifoldUPH}
In this subsection, we state and prove a version of local unstable manifolds theorem for partially hyperbolic system,  part of which can be derived from the result of Section 9 in \cite{LL}.

\begin{lemma}\label{L:UnstableManifold}
  There exists a continuous family of $C^2$ embedded $k$-dimensional discs $\{W_\delta^u(x)\}_{x \in \Lambda}$ such that the following holds for each $x \in \Lambda$:
\begin{itemize}
\item[(1)] $W_\delta^u(x)=\exp_x\left({\rm Graph} (h_x)\right)$ where
$$
h_x:  E_x^u (\delta)  \to E_x^{cs}
$$
is a $C^{2}$ map with $h_x(0)=0$, $Dh_x(0)=0$, $\|Dh_x\| \leq \frac{1}{3}$, $\|D^2h_x\|$ being uniformly bounded in $x$ and $E_x^u (\delta)=   \{ \xi \in E_x^u: |\xi|<\delta\}$;

\item[(2)] $f W_\delta^u(x) \supset W_\delta^u(f(x))$ and $W^u(x)=\bigcup_{n \geq 1} f^n W_\delta^u(x_{-n})$
where $x_{-n}$ is the unique point in $\Lambda$ such that $f^nx_{-n}=x$;

\item[(3)] $d^u(y_{-n}, z_{-n}) \leq \gamma_0 e^{-n(\lambda_0-\varepsilon_0)} d^u (y, z)$ for any $y,\ z \in  W_\delta^u(x)$, where $d^u$ denotes the distance along the unstable discs, $y_{-k}$ is the unique point in $\Lambda$ such that $f^ky_{-k}=y$, $z_{-k}$ is defined similarly and $\gamma_0>0$, $0<\varepsilon_0 <<\lambda_0$ are some constants;

\item[(4)] there is $0<\rho <\delta$ such that, if  $W_\rho^u(x):=\exp_x\left({\rm Graph} (h_x|_{E_x^u (\rho)})\right)$ intersects $W_\rho^u(\bar{x})$ for $\bar{x} \in \Lambda$, then
$W_\rho^u(x) \subset W_\delta^u(\bar{x})$.
\end{itemize}
\end{lemma}
\begin{proof}
By the compactness of $\L$ and $C^2$-ness of $f$,  there exist $r_0>0$ and $M>0$ such that
\begin{equation}\label{E:C2BoundLambda}
\max\left\{\sup_{x\in B(\L,r_0)}\left\{\|Df_x\|\right\},\sup_{x\in B(\L,r_0)}\left\{\|D^2f_x\|\right\}\right\}\le M,
\end{equation}
where $B(\L,r_0)=\{x\in \mb H|\ \di(x,\L)<r_0\}$ is the $r_0$-neighbourhood of $\L$.  Recall the definition of local unstable manifold (with a modification purely for sake of convenience)
\begin{align*}
 & W^{u}_{r(x)}(x)=\{y\in B^u(x,r(x))\oplus B^{cs}(x,r(x))|\  f^{-n}(y) \text{ exists for all } n\in \mb N,\\
 &\quad\quad\quad\quad \quad |f^{-n}(y)-f^{-n}(x)|\le C(x)e^{-n\l(x)}, \text{ as } n\to +\infty\},
 \end{align*} where $\l\in(0,\l_0)$, $r$ and $C$ are tempered functions. By the Theorem 9.14 and its proof in \cite{LL}, for the system defined in this section and $\l=\l_0-\e_0$ for some $0<\e_0<<\l_0$, there exist constants $\d'\in(0,\frac12r_0)$ and $C$ which are only depending on $r_0,M,\l_0, \l, M'(:=\sup_{x\in\L}\left\{\max\{\pi^u_x,\pi^{cs}_x\}\right\})$  such that $\{W^u_{\d'}(x)\}_{x\in \L}$ are $C^2$-embedded $k$-dimensional discs satisfying property (1) in the lemma. Then, with the uniform $C^2$ bound of $\|D^2h\|$, we have that for any given small $\e>0$, there exists $\d>0$ such that for all $x\in\L$ $$\|Dh_x|_{E^u_x(\d)}\|<\e.$$
By (\ref{E:C2BoundLambda}) and $f|_\L$ being partially hyperbolic, we have that for small enough $\e$ and $\d$,
$$\pi^{u}_{fx}\circ\exp_{fx}^{-1}\circ f\circ \exp_{x}\circ (I|_{E^{u}_{x}}+h_{x})\approx Df_{x}|_{E^{u}_{x}}\text{ thus is expanding},$$
therefore  property (2) in this lemma is satisfied and $\g_{0}$ can be taken arbitrarily close to $1$ by shrinking $\d$. It is easy to see that property (3) follows  properties (1) and (2)  immediately. For $\e<<1$, it is easy to check that if $\rho<\frac14\d$, then property (4) is satisfied.

It remains to show that this family is continuous. The proof will follow the strategy of graph transform. For sake of simplicity, we will identify $\exp_{x}(\cdot)$ with $x+\cdot$ in the rest of the proof. We view $W^{u}_{\d}(\cdot)$ as subsets of $\mb H$ and show that for any $\e'>0$, there exists $\d'>0$ such that for any $x,y\in \L$ with $|x-y|<\d'$ $d_{H}(W^{u}_{\d}(x),W^{u}_{\d}(y))<\e'$ where $d_{H}$ is the Hausdorff distance. To prove this, we may need to shrink $\d$ properly to satisfy
\begin{equation}\label{E:SmallDelta}
\frac83M'M\d+2M\d+1<e^{\frac12\l_{0}}.
\end{equation}
For any $x\in\L$, define
$$G_{x}(\D x)=f(x+\D x)-f(x)-Df_{x}(\D x).$$
Then,  $f(y)=f(x)+Df_{x}(y-x)+G_{x}(y-x)$, $G_{x}(0)=0$, $(DG_{x})_0=0$, and if $|y-x|<2\d$ then $\|(DG_{x})_{(y-x)}\|\le 2M\d$.

By the compactness of $\L$ and continuity of $f$, $(f|_{\L})^{-1}$ is continuous, thus is uniformly continuous, and so does $(f^{n}|_{\L})^{-1}$ for $n\ge 1$. For any small $\e'>0$, since $E^{u}$ and $E^{cs}$ are uniformly continuous on $\L$, there exists $\d_{1}>0$ such that if $x,y\in\L$ with$|x-y|<\d_{1}$ then
\begin{equation}\label{E:DhEst1}
d_{H}\left(W^{u}_{\d}(y), W^{u}_{\d}(y)\cap\left(B^u(x,\d)\oplus B^{cs}(x,\d)\right)\right)<\frac12\e'.
\end{equation}
Let $N>0$ be the smallest positive integer such that
$$2\g_{0}\d e^{N(-\l_{0}+\e_{0})}<\frac 3{16M'}e^{-\frac12N\l_{0}}\e'.$$
By property (2) and the uniform continuity of $(f^{n}|_{\L})^{-1}$, $E^{u}$ and $E^{cs}$, there exists $\d_{2}>0$ such that if $x,y\in\L$ with$|x-y|<\d_{2}$ then for any $n\in [1,N]$
\begin{equation}\label{E:DhEst2'}
f^{-n} \left(W^{u}_{\d}(y)\cap\left(B^u(x,\d)\oplus B^{cs}(x,\d)\right)\right)\subset \left(B^u(f^{-n}x,\d)\oplus B^{cs}(f^{-n}x,\d)\right).
\end{equation}
By the choice of $N$ and property (3), we have that for any $z\in W^{u}_{\d}(y)$
\begin{equation}\label{E:DhEst2}
|f^{-N}z-f^{-N}y|<\frac 3{16M'}e^{-\frac12N\l_{0}}\e'.
\end{equation}
Since $(f^{N}|_{\L})^{-1}$ is uniformly continuous, there exists $\d_{3}>0$ such that if $x,y\in\L$ with$|x-y|<\d_{3}$ then
\begin{equation}\label{E:DhEst3}
|f^{-N}x-f^{-N}y|<\frac 3{16M'}e^{-\frac12N\l_{0}}\e'.
\end{equation}
We will show that $\d':=\min\{\d_{1},\d_{2},\d_{3}\}$ is the desired number. For any $x,y\in \L$ with $|x-y|<\d'$, and for any $z\in W^{u}_{\d}(y)\cap\left(B^u(x,\d)\oplus B^{cs}(x,\d)\right)$, by (\ref{E:DhEst2'}), we have that
$$f^{-n}z\in  \left(B^u(f^{-n}x,\d)\oplus B^{cs}(f^{-n}x,\d)\right)\text{ for all integers }n\in[1,N].$$
Denote that $\xi_{n}=\pi^{cs}_{f^{-n}x}(f^{-n}z-f^{-n}x)$ and $\eta_{n}=\pi^{u}_{f^{-n}x}(f^{-n}z-f^{-n}x)$ for integers $n\in [0,N]$. By the standard argument of graph transform, we obtain the following inductive formula: for $n\in[1,N]$
\begin{align*}
|h_{f^{-n+1}x}(\eta_{n-1})-\xi_{n-1}|&\le \left(\frac83M'M\d+2M\d+1\right)|h_{f^{-n}x}(\eta_{n})-\xi_{n}|\\
&\le e^{\frac12\l_{0}}|h_{f^{-n}x}(\eta_{n})-\xi_{n}|\quad (\text{by }(\ref{E:SmallDelta})).
\end{align*}
By using property (1) and inequalities (\ref{E:DhEst2}), (\ref{E:DhEst3}), we have that
$$|h_{f^{-N}x}(\eta_{N})-\xi_{N}|\le \frac43 M'|f^{-N}z-f^{-N}x|\le \frac12e^{-\frac12N\l_{0}}\e'.$$
Therefore, by applying the inductive formula $N$ times, we obtain that
$$|z-(x+h_{x}(\eta_{0}))|\le \frac12\e',$$
which together with (\ref{E:DhEst1}) implies that $d_{H}(W^{u}_{\d}(x),W^{u}_{\d}(y))<\e'$ . The proof is complete.
\end{proof}

\begin{remark}\label{R:ContinuousInC2}
Since the local unstable manifolds in Lemma \ref{L:UnstableManifold} are all stay in $\L$ and the tangential space of $W^u_\d(x)$ on $y\in W^u_\d(x)$ is $E^u_y$, the family of the local unstable manifolds is continuous in $C^1$ topology which is sufficient for the rest of this paper. Actually, the continuity can be improved to $C^{2}$ topology. Focusing on the main purpose of this paper, we do not give the proof of such continuity result here.
\end{remark}

\medskip

\subsection{Proof of Theorem \ref{T:SRBExist}.}\label{SS:SRBExist}
 We construct a SRB measure $\mu$ by taking a weak* limit of the average of pushed forward Lebesgue measure on a local unstable manifold. \\

  Let $J^u(x)= |\det (Df_x|_{E^u_x})|$ for $x \in \Lambda$, and we will show that this function is Lipchitz continuous on local unstable manifold with a uniform Lipchitz constant.

\begin{lemma}\label{L:DeterminateLip}
There is a constant $C>0$ such that for any $x \in \Lambda$ and $y,\ z \in W_\delta^u(x)$ one has
$$
|J^u(y)-J^u(z)| \leq C d^u(y, z)
$$
where $d^u(\ ,\ )$ is the distance along $W_\delta^u(x)$.
\end{lemma}

\begin{proof}
The proof is based on the Lipchitz continuity of $Df$ and of the subspaces $E^u$ restricted on unstable manifold. Note that, since $\Lambda$ is compact, $f$ is $C^2$, and the splitting $E^u\oplus E^{cs}$ is continuous, $\|Df\|,\|\pi^u\|,\|\pi^{cs}\|$ are uniformly bounded on $\Lambda$. For the sake of convenience, we assume $0<\d\le 1$. By Applying Lemma \ref{L:UnstableManifold}, there exist $M\ge 1$ such that
$$\max\left\{\sup_{x\in\Lambda}\{\|Df_x\|\}, \sup_{x\in\L}\{\|D^{2}f_{x}\|\},\sup_{x\in\Lambda}\{\|\pi^u_x\|,\|\pi^{cs}_x\|\}, \sup_{x\in\Lambda}\{Lip\ Dh_x\} \right\}\le M.$$

By Lemma \ref{L:UnstableManifold}, it is easy to see that for any $z,y\in W^u_\d(x)$, if $d^u(z,y)<\frac{3\d}{4M^3}$, then $z\in W^u_\d(y)$ and $fz\in W^u_\d(fy)$. So it is sufficient to prove that there exists a constant $C>0$ which does not depend on $x$ such that for any $y\in W^u_{\frac{3\d}{8M^4}}(x)$ (thus $fy\in W^u_\d(fx)$),
$$|J^u(y)-J^u(x)|\le Cd^u(y,x).$$

First, by Lemma \ref{L:UnstableManifold}, there exist $\xi_y\in E^u_x(\frac{3\d}{8M^3})$ and $\xi_{fy}\in E^u_{fx}(\frac\d{2M})$ such that $y=x+\xi_y+h_x(\xi_y)$, $E^u_y=graph((Dh_x)_{\xi_y})$, and $fy=fx+\xi_{fy}+h_{fx}(\xi_{fy})$, $E^u_{fy}=graph((Dh_{fx})_{\xi_{fy}})$. Then,  we define linear mappings $L_x,L_y:E^u_x\to E^u_{fx}$ by
$$L_x=Df_x|_{E^u_x},\ L_y=\pi^u_{fx}Df_y|_{E^u_y}(I+(Dh_x)_{\xi_y}) \text{ with } \|L_y\|,\|L_x\|\le \frac43M^2.$$
Noting that $\|L_x-L_y\|\le (\frac{16}9+M)M^2|\xi_y|$, then by the properties of determinate, we have that
\begin{equation}\label{E:DeterminateLip1}
|\det(L_x)-\det(L_y)|\le C_1|\xi_y|,
\end{equation}
where $C_1$ is a polynomial of $M$ and $\dim E^u_x$.

Also note that for any $\xi_y\in E^u_x(\frac{3\d}{8M^3})$
$$\|\pi^u_{fx}|_{E^u_{fy}}-I\|\le \frac{\|(Dh_{fx})_{\xi_{fy}}\|}{1-\|(Dh_{fx})_{\xi_{fy}}\|}\le \frac{M|\xi_{fy}|}{1-M|\xi_{fy}|}\le \frac{\frac43M^3|\xi_y|}{1-\frac43M^3|\xi_y|}\le \frac83M^3|\xi_y|,$$
and
$$ \|(Dh_x)_{\xi_y}\|\le M|\xi_y|.$$
Then,  we have
\begin{equation}\label{E:DeterminateLip2}
\left(1-\frac83M^3|\xi_y|\right)^{\dim E^u_x}\le |\det (\pi^u_{fx}|_{E^u_{fy}})|\le \left(1+\frac83M^3|\xi_y|\right)^{\dim E^u_x},
\end{equation}
and
\begin{equation}\label{E:DeterminateLip3}
(1-M|\xi_y|)^{\dim E^u_x}\le |\det(I+(Dh_x)_{\xi_y})|\le (1+M|\xi_y|)^{\dim E^u_x}.
\end{equation}

Finally, it is easy to see that (\ref{E:DeterminateLip1}),(\ref{E:DeterminateLip2}) and (\ref{E:DeterminateLip3}) imply the desired result. This completes the proof of the lemma.
\end{proof}

By (3) of Lemma \ref{L:UnstableManifold} and Lemma \ref{L:DeterminateLip}, and also noting that $J^{u}(\cdot)>1$, we derive the following corollary immediately:
\begin{cor}\label{C:DetUniform}
There is a constant $C>0$ such that for any $x\in\Lambda$, $y,\ z \in W_\delta^u(x)$ and $n \geq 1$
$$
\frac{1}{C} \leq \prod\limits_{k=1}^n
\frac{J^u(y_{-k})}{J^u(z_{-k})} \leq C,
$$
where $y_{-k}$ is the unique point in $\Lambda$ such that $f^ky_{-k}=y$ and $z_{-k}$ is defined similarly, and moreover,
$$
\Delta (x, y):= \lim_{n\to\infty}\prod\limits_{k=1}^{n}
\frac{J^u(x_{-k})}{J^u(y_{-k})}
$$
is a well-defined function of $y \in W_\delta^u(x)$ and the limit above converges uniformly on both $y$ and $x$.
\end{cor}

Fix a point $\hat{x} \in \Lambda$ and write $L=W_\delta^u(\hat{x})$. Let $\lambda_L$ be the normalized Lebesgue measure on $L$. Let $\mu$ be a limit measure of $\frac{1}{n} \sum_{k=0}^{n-1} f^k \lambda_L$, $n \geq 1$ and assume that
$$
\frac{1}{n_i} \sum_{k=0}^{n_i-1} f^k \lambda_L \to \mu
$$
as $i \to +\infty$ for some subsequence $\{n_i\}_{i \geq 1}$ of the positive integers. Note that the existence of $\mu$ follows from the compactness of $\Lambda$.
\\

Let $x \in \Lambda$. Set $\Sigma_{x, \varepsilon}=Exp_x(E^{cs}_x (\varepsilon)) \bigcap \Lambda$ and let
$$
V_{x, \varepsilon}=\bigcup\limits_{y \in \Sigma_{x, \varepsilon}} W_\rho^u(y).
$$
By (4) of Lemma \ref{L:UnstableManifold},  we know that, when $\varepsilon$ is small enough, $ V_{x, \varepsilon}$ is the union of pairwisely disjoint pieces $W_\rho^u(y)$ for all $y \in \Sigma_{x, \varepsilon}$,  and it contains a neighborhood of $x$ in $\Lambda$. Since  $\Lambda$ is compact, we have a finite number of such sets which cover $\Lambda$. With a bit abuse of notation, let $V=V_{x, \varepsilon}$ be one of such sets with $\mu(V)>0$. Moreover, since
$W_\rho^u(y)$ is contained in $\Lambda$ for every $y \in \Lambda$, by shrinking $\varepsilon$ and $\rho$ if necessary, without losing any generality, one can assume that $\mu(\partial V)=0$ where $\partial V$ is the boundary of $V$ as a subset in $\Lambda$.
Divide $V$ into pieces  $\{W_\rho^u(y)\}_{y \in \Sigma_{x, \varepsilon}}$. This produces a measurable partition of $V$. Let $(\mu|_V)_y$ be the conditional probability measure of $\mu|_V$ (the restriction of $\mu$ on $V$) on $W_\rho^u(y)$. Then, it is easy to see that
$\mu$ is an SRB measure if,
neglecting a $\mu|_V$-null set, the following holds
\begin{equation}\label{E:(1.0)}
(\mu|_V)_y \ll \lambda^u_y\text{ on every piece } W_\rho^u(y),
\end{equation}
where $\lambda^u_y$ is the Lebesgue measure on  $W_\rho^u(y)$.

For each $n \geq 0$, let
$$L_n=\{ z \in L: f^nz \in  W_\rho^u(y)\ {\rm for\ some}\ y \in \Sigma_{x, \varepsilon}\
{\rm but}\ f^nL \not\supset  W_\rho^u(y) \}.$$
From (3) and (4) of Lemma \ref{L:UnstableManifold}, we have that, for any $z\in L_n$, $d^u(z,\partial L)\le \frac43\d\gamma_0e^{-n(\l_0-\e_0)}$. Otherwise, since for any $z'\in W^u_y(\rho)$, by (4) of Lemma \ref{L:UnstableManifold}, $z'\in W^u_z(\d)$. Then,  by (3) of Lemma \ref{L:UnstableManifold}, we have  $$d^u(z,z'_{-n})\le\gamma_0e^{-n(\l_0-\e_0)}d^u(f^nz,z')\le \frac43\d\gamma_0e^{-n(\l_0-\e_0)},$$
which implies that $z'\in f^nL$. Thus,  $W^u_y(\rho)\subset f^nL$, which is a contradiction. Therefore, we know that $\lambda_L(L_n) \to 0$ as $n \to +\infty$, exponentially. And thus
\begin{equation}\label{E:(1.1)}
\lim\limits_{i \to +\infty} \frac{1}{n_i} \sum_{k=0}^{n_i-1} f^k (\lambda_L|_{(L\setminus L_k)}) = \mu .  \end{equation}
Also, since $\mu(\partial V)=0$, we have
\begin{equation}\label{E:(1.2)}
\lim\limits_{i \to +\infty} \left( \frac{1}{n_i} \sum_{k=0}^{n_i-1} f^k (\lambda_L|_{(L\setminus L_k)})\right)(V) = \mu (V) .
\end{equation}
For every $n \geq 1$ and $z \in W_\rho^u(y)$, we define,
\begin{equation}\label{E:(1.3)}
h_{n}(z)=\frac{\prod\limits_{k=1}^n \frac{1}{J^u(z_{-k})}}{\int_{W_\rho^u(y)} \prod\limits_{k=1}^n \frac{1}{J^u(w_{-k})}\ d\lambda^u_y (w)}
=\frac{\prod\limits_{k=1}^n \frac{J^u(y_{-k})}{J^u(z_{-k})}}{\int_{W_\rho^u(y)} \prod\limits_{k=1}^n \frac{J^u(y_{-k})}{J^u(w_{-k})}\ d\lambda^u_y (w)} .
\end{equation}
Suppose that $f^n(L \setminus L_n) \supset W_\rho^u(y)$  for some  $y \in \Sigma_{x, \varepsilon}$. Let $m_{n, y}$ be the conditional probability measure of  $ [f^n (\lambda_L|_{(L\setminus L_n)})]|_V$ on $W_\rho^u(y)$. Then, by definition, we have that
$$
h_{n}|_{W^u_\rho(y)}=\frac{d m_{n, y}}{d\lambda^u_y}.
$$
For each $n \geq 1$, let $h_{n}: V \to (0, +\infty)$ be defined by (\ref{E:(1.3)}). Then,  it is clearly measurable, and, by Corollary \ref{C:DetUniform}, $h_{n}$ converges uniformly (as $n \to +\infty$) to a measurable function   $h: V \to (0, +\infty)$ defined by
$$
h(z)=\frac{\prod\limits_{k=1}^{+\infty} \frac{J^u(y_{-k})}{J^u(z_{-k})}}{\int_{W_\rho^u(y)} \prod\limits_{k=1}^{+\infty} \frac{J^u(y_{-k})}{J^u(w_{-k})}\ d\lambda^u_y (w) }
$$
if $z \in W_\rho^u(y)$ and $y \in \Sigma_{x, \varepsilon}$. Also note that, for any $n\in \mb N$,  $J^u$ is a continuous function on $\L$ and $(f|_{\L})^{-1}$ is continuous (which follows from compactness of $\L$ and the fact that $f|_\L$ is invertible and continuous). Therefore, $h_n$ is continuous. Thus,  $h$ is continuous since $h_n\to h$ uniformly as $n\to\infty$.

Define now a Borel measure $\nu$
on $V$ by
$$
\nu(A)=\int_{V} \left[ \int_{W_\rho^u(y)\bigcap A} h(z) d\lambda^u_y (z) \right] d (\mu|_V)(y)
$$
for Borel $A \subset V$.

For any continuous function $g:\L\to \mb R$, and $z\in W^u_\rho(y)$ with $y\in \Sigma_{x,\e}$, let
$$\overline g(z)=\int_{ W^u_\rho(y)}g(z')h(z')d\l^u_y(z').$$
By Lemma \ref{L:UnstableManifold} and Remark \ref{R:ContinuousInC2}, we have that
$\overline g$ is a continuous function on $V$ and is constant on each $W^u_\rho(y)$ for $y\in \Sigma_{x,\e}$.

 Denote $\L_{k}=f^{k}(L\setminus L_{k})\cap \Sigma_{x,\e}$. Now, by applying (\ref{E:(1.1)}), (\ref{E:(1.2)}) and the uniform convergence of $h_n$,  we have that for any continuous function $g:\L\to \mb R$
\begin{align*}
\int_Vgd(\mu|_V)=&\lim\limits_{i \to +\infty}\int_V g d\left(\left( \frac{1}{n_i} \sum_{k=0}^{n_i-1} f^k (\lambda_L|_{(L\setminus L_k)})\right)\Big|_V\right)\\
=&\lim_{i\to\infty}\frac1{n_{i}}\sum_{k=1}^{n_i}\sum_{y\in\Lambda_k}\left(f^k(\l_L|_{L\setminus L_k})(W^u_\rho(y))\right)\int_{ W^u_\rho(y)}g(z)h_{k}(z)d\l^u_y(z)\\
=&\lim_{i\to\infty}\frac1{n_{i}}\sum_{k=1}^{n_i}\sum_{y\in\Lambda_k}\left(f^k(\l_L|_{L\setminus L_k})(W^u_\rho(y))\right)\int_{ W^u_\rho(y)}g(z)h(z)d\l^u_y(z)\\
=&\lim_{i\to\infty}\int_V\overline gd\left(\frac1{n_{i}}\sum_{k=1}^{n_i}\sum_{y\in\Lambda_k}\left(f^k(\l_L|_{L\setminus L_k})(W^u_\rho(y))\right)\right)\\
=&\int_V\overline g(z)d(\mu|_V)=\int_V\int_{ W^u_\rho(y)}g(z)h(z)d\l^u_y(z)d(\mu|_V)(y)\\
=&\int_Vgd\nu.
\end{align*}
Hence,
$$
\nu=\mu|_V
$$
which implies (\ref{E:(1.0)}). This proves that $\mu$ is an SRB measure.  $\Box$



\section{Proof of Theorem \ref{T:SRBCont} and \ref{T:ErgodicAttractor}}\label{S:ErgodicAttractors}
In this section, we study hyperbolic measures of $f$ satisfying C1)-C3). In subsection \ref{S:ACW}, we formulate and prove the absolute continuity of stable foliations, which is the main technical result used in the proof of Theorem \ref{T:SRBCont} and \ref{T:ErgodicAttractor}. Then,  in subsection \ref{S:ExiErgodicAttractor}, we prove Theorem \ref{T:SRBCont} and \ref{T:ErgodicAttractor}.

\subsection{Absolute Continuity of $ W^s$}\label{S:ACW}
In this subsection, we prove the following proposition
\begin{prop}\label{P:AbsoluteContinuity}
$\{W^s(x)\}_{x\in supp(\mu)}$ is absolute continuous for all hyperbolic $f$-invariant measure $\mu$ supported by $\L$.
\end{prop}

In the following, we will introduce basic tools and the concept of absolute continuous stable foliation for nonuniformly hyperbolic system. Then,  we prove Proposition \ref{P:AbsoluteContinuity}. The main techniques we used here are the multiplicative ergodic theorem, the Lyapunov charts, the invariant manifold theorems and the graph transform method (to derive Disc Lemmas).

\subsubsection{Multiplicative Ergodic Theorem.} We first recall the multiplicative ergodic theorem (MET). MET for finite dimensional maps
or matrix-valued cocycles was first proved by Oseledec \cite{O}. This result
has since been generalised, with the matrices in
Oseledec's theorem replaced by linear maps of Hilbert and Banach spaces;
see \cite{R}, \cite{M}, \cite{T}, \cite{LL}. We state a
version that will be used
in this paper. It is a simplified version, in which we distinguish
 Lyapunov exponents of different signs, i.e., positive, zero, or negative, and also distinguish positive Lyapunov exponents.

\begin{theorem}[\bf MET]\label{T:MET}
Let $f$ and $\L$ satisfy C1)-C3) as in Section \ref{S:Setting}. Then,
there is a full measure Borel set $\G\subset \L$ with $f$-invariant measurable functions  $k:\G\to \mb N\cup\{0\}$ and $\l_{i}:\G\to \mb R^+$ where $i\in \mb N\cup\{0\}$
such that, for every $x\in\G$, $0<\l_{0}(x)<\l_{1}(x)<\cdots<\l_{k(x)}$, and there is a splitting of the
tangent space ${\mb H}_x$ at $x$ into
$${\mb H}_x=E^u(x)\oplus E^c(x)\oplus E^s(x)\text{ with } E^{u}(x)=E_{1}(x)\oplus \cdots\oplus E_{k(x)}(x)$$
(some of these factors may be trivial), and the following properties are satisfied:

\begin{itemize}
\item[1.] 

\noindent (a)  $\dim E^\tau(x) < \infty$ for $\tau=u,c$;

\noindent (b) $Df_xE^c(x)=E^c(fx)$, $Df_{x}E_{i}=E_{i}(fx)$ for all $i\in\{1,\ldots,k(x)\}$ and $Df_xE^s(x)\subset E^s(fx)$.

\medskip
\item[2.] For $v\in E^\tau(x),\
\tau=u,c$, and $n>0$, there is a unique $v'\in
E^\tau(f^{-n}x)$, denoted $Df^{-n}_xv$, such that
$Df^n_{f^{-n}x}v'=v$.

\smallskip
\noindent
(a) For $v\in E_{i}(x) \setminus \{0\}$,
$$
\lim_{n\to \pm \infty}\frac1n\log|Df^{n}_xv|\ge\l_i(x)\ .
$$
\noindent  (b) For $v\in E^c(x) \setminus \{0\}$,
$$\lim_{n\to\pm\infty}\frac1n\log|Df^{n}_xv|=0\ .
$$
\noindent (c) For $v\in E^s(x) \setminus \{0\}$,
$$\limsup_{n\to\infty}\frac1n\log \|Df^{n}_x|_{E^s(x)}\|\le-\l_0(x)\ .
$$



\medskip
\item[3.] The projections $\pi_x^u$, $\pi_x^c$, $\pi_x^s$ with respect to the
splitting ${\mb H}_x=E^u(x)\oplus E^c(x)\oplus E^s(x)$ are Borel, and
if for closed subspaces $E,F \subset {\mb H}$, we define
$$\an(E,F)=\inf\left\{\frac{|v'\wedge
v|}{|v'||v|}\right\}_{v'\in E\setminus \{0\},v\in F \setminus \{0\}},$$
then for $(E,F) = (E^u, E^c), (E^c, E^s), (E^u, E^c \oplus E^s)$ and
$(E^u \oplus E^c, E^s)$, we have
$$\lim_{n\to\pm\infty}\frac1n\log\an(E(f^n(x)), F(f^n(x))) =0\ .
$$
\end{itemize}
\end{theorem}


\subsubsection{Invariant Manifolds for Nonuniformy Hyperbolic System in Charts}\label{S:InvariantManifolds}

In this section, we state a version of stable and
unstable manifold theorems. { The setting and conclusions in Theorems \ref{L:LocalUnstableManifolds},
\ref{L:LocalStableManifolds} and \ref{L:ContinuityOnMaps} are borrowed from \cite{LY}, which is clearly motivated by
chart maps $\{\tilde f_{f^ix}, i \in \mb Z\}$.

\medskip
\noindent {\bf Setting.} Let $\lambda_1>0$ be fixed, and let $\d_1$ and
$\d_2>0$ be
as small as needed  depending on $\lambda_1$.
We assume that there is a splitting of $\mb H$ into orthogonal subspaces:
$\mb H = E^u \oplus E^s$ with dim$(E^u)< \infty$.
For $i \in \mb Z$, let $r_i$ be positive numbers such
that $r_{i+1}e^{-\d_1} < r_i < r_{i+1}e^{\d_1}$ for all $i$, and let
$B_i = B^u_i \times B^s_i$ where $B^\tau_i=B^\tau(0,r_i), \tau=u,s$.
We consider a sequence of differentiable maps
$$
g_i: B_i \to \mb H, \qquad i= \cdots, -1,0,1,2,\cdots,
$$
such that for each $i$, $g_i = \Lambda_i + G_i$ where $\Lambda_i$ and $G_i$
are as the following:

\vskip0.1in

\begin{itemize}
\item[(I)] $\Lambda_i \in \mc L({\mb H},{\mb H})$ and splits into
$\Lambda_i = \Lambda^u_i \oplus \Lambda^s_i$ where
$\Lambda^u_i \in \mc L(E^u,E^u)$, $\Lambda^s_i \in \mc L(E^s,E^s)$,
and $\|(\Lambda^u_i)^{-1}\|, \|\Lambda^s_i\| \le e^{-\lambda_1}$;
\item[(II)] $|G_i(0)| < \delta_2 r_{i+1}$, and $\|DG_i(x)\|< \delta_2$ for all $x \in B_i$.
\item[(III)] there are positive numbers $\ell_i$ with
$\ell_{i+1}e^{-\d_1} < \ell_i < \ell_{i+1}e^{\d_1}$ such that
Lip$(DG_i)< \ell_i$.
\end{itemize}

\vskip0.1in

Throughout the section, the orthogonal projections from $\mb H$ to $E^u$ and $E^s$ are denoted by
$\pi^u$ and $\pi^s$ respectively.

\begin{lemma}\label{L:LocalUnstableManifolds} {\rm (Local unstable manifolds)}
 Assume (I) and (II), and let
 $\d_1$ and $\delta_2$ (depending only on $\l_1$) be sufficiently small.
 Then,  for each $i$
there is a differentiable function $h^u_i:B^u_i \to B^s_i$ depending only
on $\{g_j, j < i\}$, with

(i) $|h^u_i(0)| < \frac12 r_i$ and

(ii) $\|Dh^u_i\| \le \frac{1}{10}$

\noindent such that if $W^u_i = $graph$(h^u_i)$, then

\smallskip
(a) $g_i(W^u_i) \supset W^u_{i+1}$;

(b) for  $x,y \in W^u_i$ such that $g_ix, g_iy \in B_{i+1}$,
$$|\pi^u(g_ix)-\pi^u(g_iy)| > (e^{\lambda_1} - 2\delta_2)|\pi^ux-\pi^uy|.$$

\smallskip
\noindent If (III) holds additionally, then $h^u_i \in C^{1+{\rm Lip}}$ with Lip$(Dh^u_i)<$ const$ \cdot \ell_i$, where the "const" only depends on $\l_1,\d_1$ and $\d_2$.

\end{lemma}



\begin{lemma}\label{L:LocalStableManifolds} {\rm (Local stable manifolds)}
Assume (I) and (II), and let
 $\d_1$ and $\delta_2$ (depending only on $\l_1$) be sufficiently small.
 Then,  for each $i$
there is a differentiable function $h^s_i:B^s_i \to B^u_i$ depending only on
$\{g_j, j \ge i\}$, with

(i) $|h^s_i(0)|< \frac12 r_i$ and

(ii) $\|Dh^s_i\| \le \frac{1}{10}$

\noindent such that if $W^s_i = $graph$(h^s_i)$, then

\smallskip
(a) $g_{i}W^s_i \subset W^s_{i+1}$;

(b) for $x,y \in W^s_i$, $|\pi^s(g_ix)-\pi^s(g_iy)|< (e^{- \lambda_1} +2\delta_2) |\pi^sx-\pi^sy|$.

\smallskip
\noindent If (III) holds additionally, then $h^s_i \in C^{1+{\rm Lip}}$ with Lip$(Dh^s_i)<$ const$ \cdot \ell_i$, where the "const" only depends on $\l_1,\d_1$ and $\d_2$.
\end{lemma}

The following results will be used later, which tells that
$h^s_0$ and $h^u_0$ depend on $\{g_i\}$ continuously in the $C^1$-topology as in the setting at the beginning of this subsection. For the proof we also refer to \cite{LY}.

\begin{lemma}\label{L:ContinuityOnMaps} \  \ {\it Let $\l_1, \d_1$ and $\d_2$ be as
in Lemma \ref{L:LocalStableManifolds}, and let
$r_0$ and $\ell_0$ be fixed.
Given $\e>0$, there exists $N$ which is only depending on $\e,\l_1,\d_1,\d_2,r_0$ and $\ell_0$ such that if
$\{g_i\}$ and $\{\hat g_i\}$ are two sequences of maps satisfying Conditions
(I)--(III) and
$g_i = \hat g_i$ for all $0 \le i \le N$, then
$\|h_0^s - \hat h_0^s\|_{C^1}< \e$ where $h^s_0$ and
$\hat h_0^s$ are as in Lemma \ref{L:LocalStableManifolds} for
$\{g_i\}$ and $\{\hat g_i\}$ respectively.

Analogous results hold for $h^u_0$ provided $g_i = \hat g_i$ for
$-N < i < 0$ for sufficiently large $N$.}
\end{lemma}


\subsubsection{Unstable Disc Lemmas}\label{S:DiscLemmas}

In the following, we will derive some results of unstable discs (almost paralleled to $E^u$) under the evolution of the maps defined in Section \ref{S:InvariantManifolds}. The main tools used here is the so called graph transform method. Define
$$\mc V_i=\left\{v_i:B^u_i\to B^s_i\big|\ |v_i(0)|\le\frac12 r_i,\ \text{Lip}(v_i)\le \frac1{10}\right\}.$$
$\mc V_i$ is  a complete metric space when it is equipped with the $C^0$ norm. Now we consider $g_i(graph(v_i))\cap B_{i+1}$ for which if one can show that there exists $v_{i+1}\in\mc V_{i+1}$ such that $graph(v_{i+1})=g_i(graph(v_i))\cap B_{i+1}$, then one can define maps $\{T_i:\mc V_i\to \mc V_{i+1}\}_{i\in \mb Z}$ by letting
$$ graph(T_i(v_i))=g_i(graph(v_i))\cap B_{i+1},$$
which is so called {\em graph transform} and is well-defined due to the following lemma:
\begin{lemma}\label{L:Graph1}
Assume (I), (II), and let $\d_1,\d_2$ (only depending on $\l_1$) be sufficient small, then for any $v_i\in\mc V_i$, there exists $v_{i+1}$ such that
$$ graph(v_{i+1})=g_i(graph(v_i))\cap B_{i+1},$$
and furthermore for any $y_1,y_2\in graph(v_i)$
$$|\pi^ug_i(y_2)-\pi^ug_i(y_1)|\ge (e^{\l_1}-\frac{11}{10}\d_2)|\pi^uy_2-\pi^uy_1|.$$
\end{lemma}
\begin{proof}
First, the inequality follows from assumption (I) and (II) of $g_i$ immediately.\\

Next, we will show that $v_{i+1}$ can be defined as follows: for any $y\in graph(v_i)$ so that $f(y)\in B_{i+1}$ $$v_{i+1}(\L^u_i\pi^uy+\pi^uG_i(y))=\L^s_i\pi^sy+\pi^sG_i(y).$$
Note that for any $\tilde \xi\in B^u_{i+1}$, there exists a unique $\xi\in B^u_i$ such that
\begin{equation}\label{E:Preimage}
\L^u\xi+\pi^uG_i((\xi,v_i(\xi)))=\tilde \xi.
\end{equation}
The proof of (\ref{E:Preimage}) follows directly from the Banach fixed point theorem if we consider the following contracting map
$$F_{\tilde \xi}(\xi)=(\L^u_i)^{-1}\tilde \xi-(\L^u_i)^{-1}\pi^uG_i((\xi,v_i(\xi))),$$
whose fixed point exists and satisfies (\ref{E:Preimage}). Furthermore, if $f(y)\in B_{i+1}$, then
$$|\L^u_i\pi^uy+\pi^uG_i(y)|\le r_{i+1},$$
which implies that
$$|\pi^uy|\le e^{-\l_1}r_{i+1}(1+\d_2+\d_2 e^{\d_1}).$$
Since $v_i\in\mc V_i$, for any $y\in graph(v_i)$, we have that if $\L^u\pi^uy+\pi^uG_i(y)=0$. Thus,
$$|\pi^uy|\le e^{-\l_1}\d_2(1+e^{\d_1})r_{i+1},$$ therefore,
$$|\L^s_i\pi^sy+\pi^sG_i(y)|\le \frac12r_i(e^{-\l_1}+\d_2)+\frac{11}{10}(e^{-\l_1}+\d_2)|\pi^uy|+\d_{2}r_{i+1}. $$
Since $e^{-\l_1}<e^{-\d_1}<1$, we can choose $\d_2$ small enough to make the right hand side of above inequality to be $<\frac12 r_{i+1}$.\\

For any $y_1,y_2\in graph(v_i)$, by a straightforward computation, we have that
\begin{align*}
&|\pi^sg_i(y_2)-\pi^sg_i(y_1)|=|\L^sy_2+\pi^sG_i(y_2)-(\L^sy_1+\pi^sG_i(y_1))|\\
&\le \frac1{10}(e^{-\l_1}+11\d_2)|\pi^uy_2-\pi^uy_1|
\le\frac1{10}\frac{e^{-\l_1}+11\d_2}{e^{\l_1}-\frac{11}{10}\d_2}|\pi^ug_i(y_2)-\pi^ug_i(y_1)|,
\end{align*}
which completes the proof by choosing $\d_2$ small enough so that $\frac{e^{-\l_1}+11\d_2}{e^{\l_1}-\frac{11}{10}\d_2}<1$.

\end{proof}

For $k\in \mb Z^+$, denote $T^k_i:\mc V_i\to\mc V_{i+k}$ the composition $T_{i+k-1}\circ\cdots\circ T_i$

\begin{lemma}\label{L:DiscLemma1}
Assume (I), (II), and let $\d_1,\d_2$ (only depending on $\l_1$) be sufficient small, then for any $v^1_i,v^2_i\in \mc V_i\cap C^1$,  $\|T^k_i(w^1_i)-T^k_i(w^2_i)\|_{C^0}\le 2(e^{-\l_1}+2\d_2)^kr_i$.
\end{lemma}
\begin{proof}
Note that it is sufficient to show that for any $v^1_i,v^2_i\in\mc V_i$,
$$\|T_i(v^1_i)-T_i(v^2_i)\|_{C^0}\le (e^{-\l_1}+2\d_2)\|v^1_i-v^2_i\|_{C^0}.$$
By a straightforward computation, one has the following estimates:
for any $x^u\in B^u_{i+1}$, let $x_j^u=\pi^ug_i^{-1}(x^u,T_i(v_i^j)(x^u)),\ j=1,2$, then
\begin{equation}\label{E:GraphTransform1}
|x_1^u-x_2^u|\le\frac{e^{-\l_1}\d_2}{1-\frac{11}{10}e^{-\l_1}\d_2}\|v_i^1-v_i^2\|_{C^0},
\end{equation}
and
$$|T_i(v_i^1)(x^u)-T_i(v_i^2)(x^u)|\le\left(e^{-\l_1}+
\d_2+\frac{(e^{-\l_1}+11\d_2)e^{-\l_1}\d_2}{10-11e^{-\l_1}\d_2}\right)\|v_i^1-v_i^2\|_{C^0}.$$
Since $x^u$ is arbitrary, and by taking $\d_2$ sufficient small (only depending on $\l_1$), we have the desired estimate.
\end{proof}

\begin{lemma}\label{L:DiscLemma2}
Assume (I), (II) and (III), let $\d_1,\d_2$ (only depending on $\l_1$) be sufficient small. Then,  there exists a constant $C$ which is only depending on $\l_1,\d_1,\d_2,r_{0},\ell_{0}$ such that for any $v_0^1,v_0^2\in \mc V_0\cap C^1$
\begin{equation}\label{E:C1Coverging}
\|D(T_0^nv_0^1)-D(T_0^nv_0^2)\|_{C^0}\le Ce^{-n\frac{49}{50}\l_1}.
\end{equation}
\end{lemma}
\begin{proof}
Denote $T_0^iv_0^\tau$ by $v_i^\tau$ for $\tau=1,2$. We first derive a formula of $Dv_{i+1}^\tau$ in terms of $Dv_i^\tau$. For any $x^u\in B^u_{i+1}$, let $x_\tau=g_i^{-1}(x^u,v_{i+1}^\tau(x^u))$ and $x^u_\tau=\pi^u x_\tau$, for $\tau=1,2$. Then,  one has that
\begin{align*}
(\L_i+DG(x_\tau))diag(I,Dv_i^\tau(x^u_\tau))
=diag(\L_i^u+\pi^uDG_i(x_\tau),(\L_i^s+\pi^sDG_i(x_\tau))Dv_i^\tau(x^u_\tau)),
\end{align*}
which implies that
$$Dv_{i+1}^\tau(x^u)=(\L_i^s+\pi^sDG_i(x_\tau))Dv_i^\tau(x^u_\tau)(\L_i^u+\pi^uDG_i(x_\tau))^{-1}.$$
It is straightforward to obtain the following estimates: for $x,y\in B(0,r_i)$
\begin{itemize}
\item[a)] $\|(\L_i^u+\pi^uDG_i(x))^{-1}-(\L_i^u+\pi^uDG_i(y))^{-1}\|\le(e^{\l_1}-\d_2)^{-2}\ell_i|x-y|$;
\item[b)] $\|(\L_i^s+\pi^sDG_i(x))-(\L_i^s+\pi^sDG_i(y))\|\le \ell_i|x-y|$;
\item[c)] $\|(\L_i^u+\pi^uDG_i(x))^{-1}\|\le \frac{e^{-\l_1}}{1-e^{-\l_1}\d_2}$;
\item[d)] $\|\L_i^s+\pi^sDG_i(x)\|\le e^{-\l_1}+\d_2$;
\item[e)] $\|Dv_i^\tau(\pi^ux)\|\le \frac1{10}$.
\end{itemize}
We first consider the case when $v^{1}_i=h_i$ whose graph is the unstable manifold for $i=0,1,2,\ldots$. Then, by Lemma \ref{L:LocalUnstableManifolds} and Condition (III),  $Dv_i$'s are Lipchitz with $\text{Lip}<const\cdot \ell_i$. Set
$$x_{1,2}=(x_2^u,v_i^1(x_2^u))\text{ and }
Q_i^{1,2}=(\L_i^s+\pi^sDG_i(x_{1,2}))Dv_i^1(x^u_2)(\L_i^u+\pi^uDG_i(x_{1,2}))^{-1}.$$ By a computation, one  obtains the following estimates:\\
\begin{align*}
&\|Dv_{i+1}^2(x^u)-Dv_{i+1}^1(x^u)\|\le\|Dv_{i+1}^2(x^u)-Q_i^{1,2}\|+\|Q_i^{1,2}-Dv_{i+1}^1(x^u)\|\\
&\le\frac1{10}\left[\frac{e^{-\l_1}+\d_2}{(e^{\l_1}-\d_2)^{2}}+\frac{e^{-\l_1}}{1-e^{-\l_1}\d_2}\right]\ell_i(|x_2-x_{1,2}|+|x_1-x_{1,2}|)\\
&+\frac{e^{-\l_1}(e^{-\l_1}+\d_2)}{1-e^{-\l_1}\d_2}\|Dv_i^2-Dv_i^1\|_{C^0}+\frac{e^{-\l_1}(e^{-\l_1}+\d_2)}{1-e^{-\l_1}\d_2}\|Dv_i^1(x_1^u)-Dv_i^1(x_2^u)\|.
\end{align*}
By inequality (\ref{E:GraphTransform1}), and noting that $\text{Lip }v_i^\tau\le\frac1{10}$, $|x_2-x_{1,2}|\le\|v_i^1-v_i^2\|_{C^0}$ and $|x_1-x_{1,2}|\le\frac{11}{10}|x_2^u-x_1^u|$, one has the following estimate:
\begin{align*}
&\|Dv_{i+1}^2(x^u)-Dv_{i+1}^1(x^u)\|\\
\le&\frac1{10}\left[\frac{e^{-\l_1}+\d_2}{(e^{\l_1}-\d_2)^{2}}+\frac{e^{-\l_1}}{1-e^{-\l_1}\d_2}\right]
\frac{\ell_i}{1-\frac{11}{10}e^{-\l_1}\d_2}\|v_i^1-v_i^2\|_{C^0}\\
&+\frac{e^{-\l_1}(e^{-\l_1}+\d_2)}{1-e^{-\l_1}\d_2}\|Dv_i^2-Dv_i^1\|_{C^0}+\frac{e^{-\l_1}(e^{-\l_1}+\d_2)}{1-e^{-\l_1}\d_2}\|\frac{e^{-\l_1}\d_2}{1-\frac{11}{10}e^{-\l_1}\d_2}
const\cdot \ell_i\|v_i^1-v_i^2\|_{C^0}\\
\le& p\ell_i\|v_i^1-v_i^2\|_{C^0}+q\|Dv_i^2-Dv_i^1\|_{C^0},
\end{align*}
where $$p=\frac{\left[(\frac{e^{-\l_1}+\d_2}{(e^{\l_1}-\d_2)^{2}}+\frac{e^{-\l_1}}{1-e^{-\l_1}\d_2}+
\frac{10e^{-2\l_1}(e^{-\l_1}+\d_2)\d_2}{1-e^{-\l_1}\d_2}const\right]}{10-11e^{-\l_1}\d_2}$$
and
$$q=\frac{e^{-\l_1}(e^{-\l_1}+\d_2)}{1-e^{-\l_1}\d_2}(\approx e^{-2\l_1}\text{ as }\d_2 \text{ being sufficient small}).$$
By Lemma \ref{L:DiscLemma1} and arbitrariness of $x^u$, one obtains that
\begin{equation}\label{E:C1Estimate1}
\|Dv_{i+1}^1-Dv_{i+1}^2\|_{C^0}\le 2p \ell_i(e^{-\l_1}+2\d_2)^{i+1}r_0+q\|Dv_i^2-Dv_i^1\|_{C^0}.
\end{equation}
Noting that $\ell_i\le e^{i\d_1}\ell_0$, then by induction we obtain that for $n\in \mb N$
\begin{align}\begin{split}\label{E:C1Estimate2}
\|Dv_{n}^1-Dv_{n}^2\|&\le 2p\ell_0r_0\sum_{j=1}^{n} \hat q^{j}q^{n-j}+q^n\|Dv_0^1-Dv_0^2\|_{C^0}\le \frac{2p\ell_0r_0}{1-\frac{q}{\hat q}}\hat q^{n}+\frac15q^n,
\end{split}\end{align}
where we set $\hat q=(e^{-\l_1}+2\d_2)e^{\d_1}$.
Note that by taking $\d_1,\d_2$ sufficient small, $q\approx \hat q^2\ge e^{-\frac{99}{50}\l_1}$, so if one sets $C=4\max\{\frac{2p \ell_0r_0}{1-\frac{q}{\hat q}},\frac15\}$, then
\begin{equation}\label{E:C1Estimate2}
\|Dv_i^1-Dv_i^2\|_{C^0}\le \frac12Ce^{-\frac{49}{50}i\l_1}.
\end{equation}
Here we finish the proof when $v_i^1=h_i$. For the general case, one has that
$$ \|Dv_i^1-Dv_i^2\|_{C^0}\le\|Dv_i^1-Dh_i\|_{C^0}+\|Dh_i-Dv_i^2\|_{C^0}\le Ce^{-\frac{49}{50}i\l_1}.$$
The proof is complete.
\end{proof}

\begin{remark}\label{R:HolderCase}
 1. "$\frac{49}{50}$" in the previous lemma is not a sharp bound, actually, one can make the number arbitrarily close to $1$ by taking sufficient small $\d_1,\d_2$;\\
 2. For the case that the $f$ being $\b$-H$\ddot{o}$lder continuous, using the same idea, one can obtain the similar result of which $\d_1,\d_2$ will depend on $\b$ as well as $\l_1$, and the rate of convergence will be $e^{-\frac{49}{50}i\b\l_1}$ consequently.
\end{remark}

For any $\d_3>\d_1$, let $\hat r_i=e^{-i\d_3}r_0$ for $i\ge 0$.

\begin{lemma}\label{L:DiscLemma3}
Assume (I), (II) and (III), let $\d_1,\d_2$ be sufficient small, and $\d_3>\d_1$. Then,  there exists constant $C'$ which is only depending on $\l_1,\d_1,\d_2,r_{0},\ell_{0}$ such that for any $v\in {\mc V_0}\cap C^1$ and $n\ge 0$,
$$\sup_{x\in  B^{u}(0,\hat r_{n})}\{\|D( T_0^n(v))(x)-Dh_n(0)\|\}<C'e^{-n\min\{\d_3-\d_1,\frac{49}{50}\l_1\}}.$$
\end{lemma}
\begin{proof}
This lemma follows from Lemma \ref{L:DiscLemma2} and \ref{L:LocalUnstableManifolds} immediately, since for any $x\in  B^{u}(0,\hat r_{n})$,
\begin{align*}
&\|D( T_0^n(v))(x)-Dh_n(0)\|\le\|D( T_0^n(v))(x)-Dh_n(x)\|+\|Dh_n(x)-Dh_n(0)\|\\
\le& Ce^{-\frac{49}{50}n\l_1}+ const\cdot \ell_n|x-0|\le Ce^{-\frac{49}{50}n\l_1}+const\cdot r_0\ell_0e^{n(\d_1-\d_3)}\\
\le&(C+const\cdot r_0\ell_0)e^{-n\min\{\d_3-\d_1,\frac{49}{50}\l_1\}},
\end{align*}
where $C, const$ are from Lemma \ref{L:DiscLemma2} and \ref{L:LocalUnstableManifolds} respectively.
The proof is complete by letting $C'=C+const\cdot r_{0}\ell_{0}$.
\end{proof}

\begin{remark}\label{R:UnstableManifoldCase}
1. Roughly speaking, Lemma \ref{L:DiscLemma1} and \ref{L:DiscLemma2} tell that, locally, the pushed forward unstable disc becomes more and more "parallel" to each other, and converges to unstable manifold exponentially fast in sense of $C^1$-norm.\\
2. If one assumes that $G_i(0)=0$ and $DG_i(0)=0$ in Condition (II), then $Dh_n(0)=0$. Thus,  $\sup_{x\in B^{u}(0,\hat r_{n})}\{\|D( T_0^i(g))(x)\|\}$ tends to zero exponentially fast, which is the local version of inclination lemma.
\end{remark}


\subsubsection{Estimates on Jacobians }\label{S:JacobianEstimates}
In this section we keep the setting as in Section \ref{S:InvariantManifolds} and establish some estimates of Jacobians of $Dg_i$ restricted on the tangential spaces of graphs.
To do so, we need to further assume that:

\vskip0.1in

\begin{itemize}
\item[(IV)] There are positive numbers $l_i'$ with $l'_{i+1}e^{-\d_1}<l'_i<l'_{i+1}e^{\d_1}$ such that
$\|Dg_i\|_{C^0}\le l'_i$.
\end{itemize}

\vskip0.1in

For the rest of this subsection, we assume (I), (II), (III) and (IV).

\begin{lemma}\label{L:Jacob1}
Let $\d_1,\d_2$ being sufficient small (only depending on $\l_1$). Then,  for any $v_0\in\mc V_0\cap C^1$, there exists constant $C\ge 1$ which depends only on $\dim E^u,\l_1,\d_1,\d_2,r_0,\ell_0,l_0'$ such that for any $x_1,x_2\in graph(v_0)$ if
$$g_0^k(x_1),g_0^k(x_2)\in graph(T_0^k(v_0)),\text{ for all }1\le k\le n,$$
then
$$\frac1C\le \frac{\left|\det\left(Dg_0^n\big|_{graph(Dv_0(\pi^ux_1))}\right)\right|}
{\left|\det\left(Dg_0^n\big|_{graph(Dv_0(\pi^ux_2))}\right)\right|}\le C.$$
\end{lemma}

\begin{lemma}\label{L:Jacob2}
For any $v_0^1,v_0^2\in \mc V_0\cap C^1$, there exists constant $C\ge 1$ which depends only on $\dim E^u, \l_1,\d_1,\d_2,r_0,\ell_0,l_0'$ such that for $n\ge 1$
$$\frac1C\le \frac{\left|\det\left(Dg_0^n\big|_{graph(Dv_0^1(\pi^ux_1))}\right)\right|}
{\left|\det\left(Dg_0^n\big|_{graph(Dv_0^2(\pi^ux_2))}\right)\right|}\le C,$$
where $x_\tau=W^s_0\cap graph(v_0^\tau),\ \tau=1,2$.
\end{lemma}

  In the following proof we will use $C$ as a generic constant to denote a function of $\dim E^u, \l_1,\d_1,\d_2,r_0,\ell_0,l_0'$.
\begin{proof}
We first prove Lemma \ref{L:Jacob2}. Denote $x_\tau^k=g_0^k(x_\tau)$ and $v^\tau_k=T_0^k(v^\tau_0)$ for $k\ge 0$.We need to compare $\det(Dg_k(x_1^k)|_{graph(Dv^1_k(\pi^ux_1^k))}))$ and $\det(Dg_k(x_2^k)|_{graph(Dv^2_k(\pi^ux_2^k))}))$. First by the properties of determinate we have that
\begin{align*}
&\frac{\det(Dg_i(x_1^k)|_{graph(Dv^1_k(\pi^ux_1^k))})}{\det(Dg_i(x_2^k)|_{graph(Dv^2_k(\pi^ux_2^k))})}\\
=&\frac{\det(\pi^uDg_k(\pi^ux_1^k)|_{E^u})}{\det(\pi^uDg_k(\pi^ux_2^k)|_{E^u})}\cdot
\frac{\det(\pi^u|_{graph(Dv^2_{k+1}(\pi^ux_2^{k+1}))})}{\det(\pi^u|_{graph(Dv^1_{k+1}(\pi^ux_1^{k+1}))})}
\cdot\frac{\det((I,Dv^2_k(\pi^ux_2^k)))}{\det((I,Dv^1_k(\pi^ux_1^k)))}.
\end{align*}
Then,  noting that $\pi^u|_{graph(Dv^\tau_k(\pi^ux_\tau^k))}(I,Dv^\tau_k(\pi^ux_\tau^k))=id|_{E^u}$, and by Lemma \ref{L:DiscLemma2}, we have that
\begin{align}\begin{split}\label{E:Norm1}
&\|(I,Dv^2_k(\pi^ux_2^k))\pi^u|_{graph(Dv^1_k(\pi^ux_1^k))}\|\le 1+\|Dv_k^2(\pi^ux_2^k)-Dv_k^1(\pi^ux_1^k)\|.
\end{split}\end{align}
To estimate the right hand side of above inequality, applying Lemma \ref{L:DiscLemma2}, the Lipchitz property of $Dh_{i}$ in Lemma \ref{L:LocalUnstableManifolds} and Lemma \ref{L:LocalStableManifolds}, we obtain that
\begin{align}\begin{split}\label{E:Norm1'}
&\|Dv_k^2(\pi^ux_2^k)-Dv_k^1(\pi^ux_1^k)\|\\
\le&\|Dv_k^2-Dh_{k}\|_{C_{0}}+\|Dv_k^1-Dh_{k}\|_{C_{0}}+\|Dh_{k}(\pi^ux_1^k)-Dh_{k}(\pi^ux_2^k)\|\\
\le& Ce^{-k\frac{49}{50}\l_1}+const\cdot \ell_{i}|\pi^ux_1^k-\pi^ux_2^k|\le Ce^{-k\frac{49}{50}\l_1}+C(e^{-\l_1}+2\d_{2})^{k}\\
\le &Ce^{-k\frac{49}{50}\l_1}\text{ for small enough } \d_{2}.
\end{split}\end{align}

Again by Lemma \ref{L:LocalStableManifolds} and Condition (III), we have
\begin{align}\begin{split}\label{E:Norm2}
\left\|\pi^uDg_k(\pi^ux_1^k)|_{E^u}-\pi^uDg_k(\pi^ux_2^k)|_{E^u}\right\|
\le\ell_{k}|x_1^k-x_2^k|
\le\ell_{0}e^{k\d_{1}}(e^{-\l_1}+2\d_2)^kr_0\le Ce^{-k\frac{49}{50}\l_1}.
\end{split}\end{align}

Also note that for all $ k\ge 0$ and $\tau=1,2$
\begin{align}\begin{split}\label{E:Norm3}
&\left\|\pi^uDg_k(\pi^ux_\tau^k)|_{E^u}\right\|\le l'_0e^{k\d_1},\|\pi^u|_{graph(Dv_{k}^{\tau}(\cdot))}\|\le 1,
\|(I,Dv_k^\tau)\|_{C^0}\le \frac{11}{10},\\
&\text{and }\left|\det\left(\pi^uDg_k(\pi^ux_\tau^k)|_{E^u}\right)\right|>1.
\end{split}\end{align}
Then,  (\ref{E:Norm1}), (\ref{E:Norm1'}), (\ref{E:Norm2}) and (\ref{E:Norm3}) together imply that, for sufficient small $\d_1,\d_2$, the following holds
 $$\left|\frac{\det(Dg_k(x_1^k)|_{graph(Dv^1_k(\pi^ux_1^k))})}
 {\det(Dg_k(x_2^k)|_{graph(Dv^2_k(\pi^ux_2^k))})}\right|\le 1+Ce^{-k\frac{49}{50}\l_1},$$
which complete the proof of Lemma \ref{L:Jacob2}.\\

The proof of Lemma \ref{L:Jacob1} is similar, so we omit the detailed argument and just give the outlines the proof:\\

Denote $x_\tau^k=g^k_0(x_\tau)$ and $v_k=T_0^k(v_0)$ for $0\le k\le n$, $\tau=1,2$. Since $x_\tau^k\in graph(v_k)$, there exists $y_\tau^k\in W_i^u$ such that $\pi^ux_\tau^k=\pi^uy_\tau^k$. Then,  the proof is divided  into two main steps:
\begin{itemize}
\item[Step 1.] For each $0\le k\le n-1$ and $\tau=1,2$, show that for sufficient small $\d_2,\d_1$
 $$\left(1+Ce^{-k\frac{49}{50}\l_1}\right)^{-1}\le\left|\frac{\det (Dg_k(x_\tau^k)|_{graph(Dv_k(\pi^ux_\tau^k))})}{\det (Dg_k(y_\tau^k)|_{\mc T_{y_\tau^k}W^u_k})}\right|\le 1+Ce^{-k\frac{49}{50}\l_1},$$
 where $\mc T_{y_\tau^k}W^u_k$ is the tangential space of $W^u_k$ on point $y_\tau^k$. To prove this, one can apply the same arguments in proving (\ref{E:Norm1}), (\ref{E:Norm1'}), (\ref{E:Norm2}) and (\ref{E:Norm3}), and the fact that $\|v_k-h^u_k\|_{C^0},\|Dv_k-Dh^u_k\|_{C^0}\le Ce^{-n\frac{49}{50}\l_1}$ which follows from Lemma \ref{L:DiscLemma1} and \ref{L:DiscLemma2}.
 \item[Step 2.] For each $0\le k\le n-1$, show that for sufficient small $\d_2,\d_1$
 $$\left|\frac{\det (Dg_k(y_2^k)|_{\mc T_{y_2^k}W^u_k})}{\det (Dg_k(y_1^k)|_{\mc T_{y_1^k}W^u_k})}\right|\le 1+Ce^{-(n-k)\frac{49}{50}\l_1}.$$
 This follows from the Lipchitz property of $Dh_i$ and property (b) stated in Lemma \ref{L:LocalUnstableManifolds}.
\end{itemize}

\end{proof}

\subsubsection{Hyperbolic Blocks}\label{S:HyperbolicBlocks}
In this subsection, we will set up the Lyapunov charts and then define the so called hyperbolic blocks.

First, we define a set of subsets $\{\G_i\subset \G\}_{i\in I}$, where $\G$ is given in the above multiplicative theorem, by the following:
$$ x,y\in\G_i,\text{ if and only if } \dim E^{u}(x)=\dim E^{u}(y).$$
Since $\mu$ is hyperbolic, up to an $\mu$-null set, $\{\G_i\}_{i\in I}$ forms a countable partition of $\G$ (then one can choose the $I$ to be natural numbers), and each $\G_i$ is $f$-invariant because of 1(b) in Theorem \ref{T:MET}(from Appendix A).

Another way to classify points is to look at their Lyapunov exponents: For $n\ge 1$,  where we define
$$\G^n=\{x\in\G|\ \l_0\in[\frac1n,\frac{1}{n-1})\}, \text{ let }\frac1{n-1}=+\infty, \text{ when }n=1 .$$
Here $\l_0$ is the function given in Theorem \ref{T:MET}. Also since $\mu$ is hyperbolic, up to an $\mu$-null set, $\{\G^n\}_{n\ge 1}$ forms a countable partition of $\G$, and  each $\G_i$ is $f$-invariant since so is $\l_0$.

Let $\G_{n,m}=\G^n\cap\G_m$, then $\mu=\sum_{n,m\ge 1}\mu(\G_{n,m})\mu_{n,m}$ and each $\G_{n,m}$ is $f$-invariant, where $\mu_{n,m}$ is the conditional measure of $\mu$ on $\G_{n,m}$.

Now we are able to construct Lyapunov charts. Let $\d_0:\G\to (0,1)$ be $f$-invariant functions satisfies that
$$\d_0|_{\G_{n,m}}=\text{ constant, and }\le \frac1{200n}.$$
Then,  let $\l(x)=\l_0(x)-2\d_0(x)$, where $\l_0(x)\ge\frac1n$ for any $x\in\G^n$. Note that on $\G_m$, $\dim E^u=codim E^s$ are constants. Then,  one can fix a measurable splitting $\mb H= \tilde E^u(x)\oplus \tilde E^s(x)$ as the following: for $\tau=u,s$
\begin{itemize}
\item[(i)] $\tilde E^\tau(x)$ is constant subspace on $\G_m$;
\item[(ii)] $\dim\tilde E^\tau(x)=\dim E^\tau(x)$;
\item[(iii)] $\tilde E^u(x)\perp \tilde E^s(x)$.
\end{itemize}

Now we define functions $\d_1,\d_2$ and $\d$ on $\G_{n,m}$: For any $x\in \G_{n,m}$, take $\l_1=\frac3{4n}$ as under the setting in Section \ref{S:InvariantManifolds}, then we choose $\d_1(x),\d_2(x)$ sufficiently small to satisfy all the conditions of $\d_1,\d_2$ proposed in Section \ref{S:InvariantManifolds}, \ref{S:DiscLemmas} and \ref{S:JacobianEstimates}. Since $\d_1(x)$ and $\d_2(x)$ only depend on $\l_1$,  we choose them to be constants on $\G_{n,m}$ and to be $f$-invariant.  To construct the chart we also need a control on $\|Df_x^2\|$, so in the rest of this section we always assume that $\d(x)<<r_0$ where $r_0$ is from Lemma \ref{L:C2Bound}. Finally, set $\d(x)=\frac18\min\{\d_1(x),\d_2(x)\}$. And then we apply Theorem \ref{T:LyapunovChart}, and use the same symbol $\Phi$ to denote the chart map.
\begin{remark}\label{R:ForLargerLambda1}
Note that all the results in Section \ref{S:InvariantManifolds}, \ref{S:DiscLemmas} and \ref{S:JacobianEstimates}  still hold if one keeps $\d_1$ and $\d_2$ but use a larger $\l_1$. And we will use this fact if needed in the rest of Section \ref{S:ErgodicAttractors} without further explanation.
\end{remark}

 Then,  by  applying Theorem \ref{T:LyapunovChart} to construct the Lyaponuv charts, we obtain a Borel function $l:\G\to [1,\infty)$ and the chart maps $\Phi_x$ satisfying the properties in Theorem \ref{T:LyapunovChart}.

Let $\G_{n,m,l_0}=\{x\in \G_{n,m}|\ l(x)\le l_0\}$. This set usually is not $f$-invariant. Nevertheless, since $\mu(\cup_{l_0\ge 1}\cup_{n,m=1}^\infty \G_{n,m,l_0})=1$, we have that for some $n,m\ge 1$, and large enough $l_0$, $\mu(\G_{n,m,l_0})>0$.

\begin{remark}\label{R:SplittingOnGnml0}
The following properties follow from the construction of chart maps (we refer the reader to \cite{LY} for more details): For any $x\in \G_{n,m,l_0}$,
\begin{itemize}
\item[(a)] $\left\|\left(D^pf_{f^qx}|_{E^u(f^qx)}\right)^{-1}\right\|\le l_0e^{-p\frac{99}{100n}+|q|\d(x)}$ for all $p\ge 0$ and $q\in\mb Z$;
\item[(b)] $\left\|D^pf_{f^qx}|_{E^{s}(f^qx)}\right\|\le l_0e^{-p\frac{99}{100n}+|q|\d(x)}$ for all $p\ge 0$ and $q\in\mb Z$;
\item[(c)] $\max\{\|\pi^u(f^q(x))\|,\|\pi^{s}(f^q(x))\|\}\le l_0e^{|q|\d(x)}$ for all $q\in\mb Z$.
\end{itemize}

\end{remark}
 We summarize the properties of $\G_{n,m,l_0}$ in the following lemma which is an analog of the Continuity Lemma in \cite{Pu}.
\begin{lemma}\label{L:ContinuityLemma}
The $Df$-invariant hyperbolic splitting $\mb H=E^u(x)\oplus E^s(x)$ over $\G_{n,m,l_0}$ can be extended to a unique continuous $Df$-invariant splitting on $\overline{\G_{n,m,l_0}}$, where $\overline{\G_{n,m,l_0}}$ is the closure of $\G_{n,m,l_0}$ in $\L$. The growth controls along orbits through $\G_{n,m,l_0}$ can be extended to growth controls along orbits through $\overline{\G_{n,m,l_0}}$, precisely (a), (b) and (c) in Remark \ref{R:SplittingOnGnml0} hold for $x\in \overline{\G_{n,m,l_0}}$.
\end{lemma}
\begin{proof}
By Propositions \ref{P:G^uConti} and \ref{P:G^sConti}, we have  that the splitting $\mb H=E^u(x)\oplus E^s(x)$ is uniformly continuous varying on $\G_{n,m,l_0}$, thus such continuous splitting can be extended uniquely to $\overline{\G_{n,m,l_0}}$. The second part the lemma follows from the compactness of $\Lambda$, $C^1$-ness of $f$, and continuity of $\left(f|_{\Lambda}\right)^{-1}$ straightforwardly.
\end{proof}

By an $\e_0-Hyperbolic\ Block$, we mean an element of a countable measurable partition $\{P_i\}_{i=1}^\infty$ of $\G$ such that
\begin{itemize}
\item[(i)] $P_i\subset \G_{n,m,l_0}$ for all $i\ge 1$ and for some $n,m,l_0$;
\item[(ii)] $diam(P_i)\le \e_0$.
\end{itemize}
By the compactness of $\Lambda$, such partition exists for any $\e_0$.


\subsubsection{Invariant Manifolds and Holonomy maps of nonuniformly hyperbolic systems}\label{S:HolonomyMaps}
In this subsection, we define the so called  holonomy map and formulate the concept of absolute continuity of stable foliations for nonuniformly hyperbolic systems.
We keep the setting of charts defined in subsection \ref{S:HyperbolicBlocks}.
The following result  is an immediate
corollary of
Lemmas \ref{L:LocalUnstableManifolds} and \ref{L:LocalStableManifolds}, which gives local stable and unstable manifolds
$\mu$-a.e. for a given hyperbolic measure $\mu$ of $f$ satisfying C1)-C3).
Various versions of this result have been
proved before; see e.g. \cite{R}, \cite{LL}.
\begin{cor}[\bf Local Stable and Unstable Manifolds for Nonuniformly Hyperbolic System] \label{C:LocalInvariantManifolds}
For $\mu$-a.e. $x$, one has that
$$ W^\tau_{loc}(x)=\Phi_x(\tilde W^\tau_x),\  \tilde W^\tau_x=\text{graph}(\tilde h^\tau_x), \tau=u,s$$
where $\tilde h^s_{x}:\tilde B^s(0,\d(x)l(x)^{-1})\to \tilde B^u(0,\d(x)l(x)^{-1})$ and $\tilde h^u_{x}:\tilde B^u(0,\d(x)l(x)^{-1})\to \tilde B^s(0,\d(x)l(x)^{-1})$ satisfy the following properties:
\begin{itemize}
\item[(i)] $\tilde h^\tau_x(0)=0,\ (D\tilde h^\tau_x)_0=0$, $\tau =u,s$;
\item[(ii)] $\|D\tilde h^\tau_{x}\|\le \frac1{10}$, $\tau=u,s$;
\item[(iii)] (a)$f (W^u_{loc}(x))\supset W^u_{loc}(f(x))$, (b)$f (W^s_{loc}(x))\subset W^s_{loc}(f(x))$;
\item[(iv)] (a) For any $y_1,y_2\in \tilde W^u_x$ such that $\tilde f_x(y_1), \tilde f_x(y_1)\in \tilde W^u_{f(x)}$,
    $$|\tilde \pi^u_{f(x)}(\tilde f_x(y_1))-\tilde \pi^u_{f(x)}(\tilde f_x(y_2))|>(e^{\l_1(x)}-2\d_2(x))|\tilde \pi^u_{x}y_1-\tilde \pi^u_{x}y_2|;$$
    (b)For any $y_1,y_2\in \tilde W^s_x$,
    $$|\tilde \pi^s_{f(x)}(\tilde f_x(y_1))-\tilde \pi^s_{f(x)}(\tilde f_x(y_2))|<(e^{-\l_1(x)}+2\d_2(x))|\tilde \pi^s_{x}y_1-\tilde \pi^s_{x}y_2|;$$
    where $\tilde \pi^u_{x},\tilde \pi^s_{x}$ are the projections associated to the splitting $\mb H=\tilde E^u(x)\oplus\tilde E^s(x)$;
\item[(v)] $\tilde h^\tau_x\in C^{1+lip}$ with $Lip(D\tilde h^\tau_x)<const\cdot l(x)$, $\tau=u,s$.
\end{itemize}
\end{cor}
The $\Phi_x$-images of $\tilde W^s_x$ and $\tilde W^u_x$ are
called the {\it local stable} and {\it unstable manifolds} at $x$.

 To define the holonomy map, we first fix numbers $n_0,m_0,l_0$ so that $\mu(\G_{n_0,m_0,l_0})>0$ and then an $x_0\in \G_{n_0,m_0,l_0}\cap supp\ \mu$, and let $r_0=\d(x_0)l_0^{-1}$. Define that
  $$\tilde{\mc V}=\{\tilde v:\tilde B^u(0,r_0)\to \tilde B^s(0,r_0)|\ |\tilde v(0)|\le \frac14 r_0, \ Lip(\tilde v)\le \frac1{20}\},$$
  $$\tilde{\mc V}'=\{\tilde v:\tilde B^u(0,\frac34r_0)\to \tilde B^s(0,\frac34r_0)|\ |\tilde v(0)|\le \frac38 r_0, \ Lip(\tilde v)\le \frac1{10}\},$$
  and for any small $c>0$
  $$\tilde{\mc W}'=\{\tilde w:\tilde B^s(0,\frac34r_0)\to \tilde B^u(0,\frac34r_0)|\ |\tilde w(0)|\le \frac38 r_0, \ Lip(\tilde w)\le \frac1{10}+c\}.$$


 \begin{lemma}\label{L:HolonomyMap}
 There exists $\e\in(0,\frac{1}{10}r_0)$ such that for any  $x,y\in B(x_0,\e)\cap \G_{m_0,n_0,l_0}$ with $|x-y|<\e$ and $\tilde v,\tilde v_1, \tilde v_2 \in\tilde{\mc V}$, the following holds:
 \begin{itemize}
 \item[(i)] There is a unique $\tilde v'\in\tilde {\mc V}'$ such that $$\Phi_{y}^{-1}\left(\Phi_{x}(graph(\tilde v))\right)\cap\tilde B(0,\frac34r_0)=graph(\tilde v');$$
 \item[(ii)] $W_{loc}^s(y)$ meets $\Phi_{x}(graph(\tilde v))$ at exact one point transversally;
 \item[(iii)] There exists a homeomorphism $T_{v_1,v_2}:U_1\to U_2$, where
 \[U_i=\left(\cup_{y\in B(x_0,\e)\cap \G_{m_{0},n_{0},l_{0}}}W^s_{loc}(y)\right)\cap \Phi_{x_0}(graph(v_i)),i=1,2,
 \] such that
 $T_{v_1,v_2}$ maps $W_{loc}^s(y)\cap\Phi_{x_0}(graph(\tilde v_1))$ to $W_{loc}^s(y)\cap\Phi_{x_0}(graph(\tilde v_2))$.
 \end{itemize}
 \end{lemma}
 \begin{proof}
 First, we borrow a result following Proposition 17 in \cite{LY}  for switching charts: for any $\Delta>0$, there is an $\e>0$ such that if $x,y\in \G_{m_0,n_0,l_0}$ satisfying $|x-y|<\e$, then for any $v\in \tilde E^\tau$, $\tau=u,s$, the following holds:
 \begin{itemize}
 \item[(a)] $(1-\Delta)|v|<|\tilde\pi^\tau L_y L_x^{-1} v|\le (1+\Delta)|v|$;
 \item[(b)] $|\tilde\pi^{\tau'} L_y L_x^{-1} v|\le \Delta|v|$ when $\tau'\neq \tau$.
 \end{itemize}
 Here $L_x$ is defined for the chart map and $\Phi_x=Exp\circ L_x^{-1}$ (for definitions of $\Phi_{x}$ we refer to Section \ref{S:LyapunovCharts}).
 Then,  (i) is easily seen to be true for small enough $\Delta$. To show (ii), note that for $\Delta$ small enough, if $x,y\in \G_{m_0,n_0,l_0}$ satisfying $|x-y|<\e$, by Corollary \ref{C:LocalInvariantManifolds}, there exists $\tilde w'\in \tilde {\mc W}'$ such that
 $$\Phi_{x}^{-1}(\Phi_y(\tilde W^s_y))\cap \tilde B(0,\frac34 r_0)=graph(\tilde w').$$
 Note that the intersection of $W_{loc}^s(y)$ and $\Phi_{x}(graph(\tilde v))$ is the fixed point of the mapping $\tilde x^u\to \tilde w'(\tilde v (\tilde x^u))$ for $\tilde x^u\in \tilde\pi^u(graph(\tilde w'))$. Since this mapping has Lipschitz constant $\frac1{20}(\frac1{10}+c)<1$, (ii) follows from the contraction mapping theorem.

For (iii), it remains to show $T_{v_1,v_2}$ being continuous which follows from  the next Lemma which gives the continuous dependence of $W^\tau_{loc}(\cdot)$ on $\G_{n_0,m_0,l_0}$ for $\tau=u,s$. This completes the proof of the lemma.
 \end{proof}

 \begin{lemma}\label{L:ContinuityOfW}
 For any $\Delta'>0$, there exists $\e\in(0,\frac{1}{10}r_0)$ such that if  $y,x\cap \G_{m_0,n_0,l_0}$ with $|y-x|<\e$, then
  $$\left\|\Phi_{x}^{-1}(W^{\tau}_{loc}(y))\cap \tilde B(0,\frac34r_0)-\tilde W^{\tau}_{x}\cap \tilde B(0,\frac34r_0)\right\|_{C1}<\Delta'$$ for $\tau=s,u$.
 \end{lemma}
 \begin{proof}
 Here we only give the outline of the proof for $\tau=u$ since the proof is mainly same for $\tau=s$ which are both based on graph transforms.  Note that $f^{q}x\in \G_{n_0,m_0,l_0e^{|q|\d}}$ for all $q\in\mb Z$. For a small $c>0$ and $q\in \mb Z$, define
 {\small
 \begin{align*}
  &\tilde{\mc W}_q'=\left\{\tilde v:\tilde B^s(0,\frac34r_0e^{-|q|\d})\to \tilde B^u(0,\frac34r_0e^{-|q|\d})\Big|\ |\tilde v(0)|\le \frac38 r_0e^{-|q|\d}, \ Lip(\tilde v)\le \frac1{10}+c\right\}.
  \end{align*}}
  For any $n\in\mb N$, since $(f|_\L)^{-1}$ is continuous, and also by the switching chart results in the proof of Lemma \ref{L:HolonomyMap}, there exists $\e'>0$ such that  if $x,y\in \G_{m_0,n_0,l_0}$ satisfying $|x-y|<\e'$ then there exists $\tilde w'_k\in \tilde {\mc W}_k'$ for each $1\le k\le n$ the following holds
 $$\Phi_{x}^{-1}(\Phi_{f^{-k}y}(\tilde W^u_{f^{-k}y}))\cap \tilde B(0,\frac34 r_0)=graph(\tilde w'_k).$$
 Then,  by going through the same procedure (the graph transform) as in proving Lemma \ref{L:DiscLemma1} and \ref{L:DiscLemma2}, and noting that $\Phi_{x}^{-1}(W^{u}_{loc}(y))\cap \tilde B(0,\frac34r_0)$ is the image of $graph(\tilde w'_n)$ under $n$th graph transforms along charts along orbits through $x$, we have that
$$\left\|\Phi_{x}^{-1}(W^{u}_{loc}(y))\cap \tilde B(0,\frac34r_0)-\tilde W^{u}_{x}\cap \tilde B(0,\frac34r_0)\right\|_{C1}\le C e^{-n(\frac{9}{10n_0})},$$
where $C$ is depending on $n_0,l_0,r_0,\d$ only. To complete the proof, one needs only to choose an $n$ satisfying $C e^{-n(\frac{9}{10n_0})}<\Delta'$ for a given $\D'$.

For $\tau=s$, one needs the "backward graph transform", for details we refer to \cite{LY}.
 \end{proof}


Now we are ready to formulate the concept of absolute continuity of stable foliations $\{W^s(x)\}_{x\in supp(\mu)}$:
\begin{definition}\label{D:AbsContinuity}
The stable foliation $\{W^s(x)\}_{x\in supp(\mu)}$ is said to be absolutely continuous if for any $\G_{m_0,n_0,l_0}\subset supp(\mu)$ with $\mu(\G_{m_0,n_0,l_0})>0$, and any $\tilde v_1,\tilde v_2\in \tilde {\mc V}$, the map $T_{v_1,v_2}$ defined in (iii) of Lemma \ref{L:HolonomyMap} is absolute continuous i.e., for any Borel set $A\subset U_1$ with $0$-Lebesgue measure in $\Phi_{x_0}(graph(\tilde v_1))$, $T_{v_1,v_2}(A)$ is a null Lebesgue set in $\Phi_{x_0}(graph(\tilde v_2))$.
\end{definition}

 The last lemma in this subsection is to extend the local unstable manifolds defined on $\G_{n,m,l_0}$ to  $\overline{\G_{n,m,l_0}}$, which will be used in subsection \ref{S:ProofofACW} to derive the measurability of the partition of hyperbolic block into local unstable manifolds. The main difference between the following lemma and Corollary \ref{C:LocalInvariantManifolds} is that, unlike $supp(\mu)$, $\overline{\G_{n,m,l_0}}$ is not $f$-invariant.

\begin{lemma}\label{L:UnstableManifoldsExtension}
 There exists a continuous family of $C^2$ embedded discs  $\{W_r^u(x)\}_{x\in \overline{\G_{n,m,l_0}}}$ with $r$ only depending on $n,m,l_0$ such that the following holds  for each $x \in \overline{\G_{n,m,l_0}}$:
\begin{itemize}
\item[(1)] $W_r^u(x)=\exp_x\left({\rm graph} (h^u_x)\right)$ where
$$
h^u_x:  E_x^u (r)  \to E_x^{s},
$$
is a $C^{2}$ map with $h^u_x(0)=0$, $Dh^u_x(0)=0$, $\|Dh^u_x\| \leq \frac{1}{10}$, $\|D^2h^u_x\|$ being uniformly bounded on $x$ and $E_x^u (r)=   \{ \xi \in E_x^u: |\xi|<r\}$;

\item[(2)] $ W_r^u(x) \subset W_{loc}^u(x)$ for any $x\in \G_{n,m,l_0}$;

\item[(3)] $|y_{-p}-z_{-p}| \leq \gamma_0 (e^{\frac{99}{100n}-2\d(x)}-3\d(x))^{-p} |y-z|$ for any $y,\ z \in  W_r^u(x)$, where $y_{-k}$ is the unique point in $\Lambda$ such that $f^ky_{-k}=y$, $z_{-k}$ is defined similarly and $\gamma_0>0$is a constant depending on $n,m,l_0,\d(x)$ only;

\item[(4)] there is $0<\rho <r$ which is only depending on $n,m,l_0$ such that, if  \[W_\rho^u(x):=\exp_x\left({\rm graph} (h_x|_{E_x^u (\rho)})\right)\]
     intersects $W_\rho^u(\bar{x})$ for $\bar{x} \in \Lambda$ for any $x,\bar x\in \G_{n,m,l_0}$, then
$W_\rho^u(x) \subset W_r^u(\bar{x})$.
\end{itemize}

\end{lemma}
\begin{proof}
Actually, we only need to construct $W_r^u(x)$'s for $x\in\overline{\G_{n,m,l_0}}-\G_{n,m,l_0}$ and trim the $W^u_{loc}$ given in Corollary \ref{C:LocalInvariantManifolds} to the one with proper size. Note that the Peoperties (a), (b) and (c) in Remark \ref{R:SplittingOnGnml0} hold for any $x\in \G_{n,m,l_0}$. We will construct the Lyapunov charts along the backward orbit $\{f^{-p}x\}_{p\ge 0}$ for any given $x\in\overline{\G_{n,m,l_0}}-\G_{n,m,l_0}$.  The procedure of doing this is the same as in \cite{LY}, so the details are omitted here. In fact, focusing on what we need here, one does not need to construct a globally defined chart maps and then do not need to deal with the issue of measurability, which makes the constructing simpler.

Before we continue the proof, we summarise the properties of the chart maps in the following: For any $p\ge 0$, we denote $\mb H=\tilde E^u\oplus \tilde E^s$ the sample orthogonal splitting, $\Phi_p$ the chart map on $f^{-p}x$ and $\tilde f_{p+1}=\Phi_p^{-1} \circ f\circ \Phi_{p+1}$, which satisfy  the following properties
\begin{itemize}
\item[(i)] $l_0^6e^{-6p\d(x)}|y-y'|\le |\Phi_p(y)-\Phi_p(y')|\le \sqrt 2 |y-y'|$ for any $y,y'\in \mb H$;
\item[(ii)] $(D\tilde f_{p+1})_{0}(\tilde E^u)=\tilde E^u$, $(D\tilde f_{p+1})_{0}\tilde E^s\subset \tilde E^s$ and
 $$\left\|\left((D\tilde f_{p+1})_0\big|_{\tilde E^u}\right)^{-1}\right\|, \left\|(D\tilde f_{p+1})_0\big|_{\tilde E^s}\right\|\le e^{\frac{99}{100n}-2\d(x)};$$
 \item[(iii)] For any $r'\in(0,\d(x))$, the following hold on $\tilde B_{p+1}=\tilde B^u(0,r'l_{p+1})\oplus \tilde B^s(0,r'l_{p+1})$ where $l_p=6(M_2+1)l_0^6e^{6(p+2)\d(x)}$ and $M_2$ is from Lemma \ref{L:C2Bound}:
 \begin{itemize}
 \item[(a)] $Lip(\tilde f_p-(D\tilde f_p)_0)<r'$;
 \item[(b)] $Lip(D\tilde f)\le l_p$.
 \end{itemize}
 \end{itemize}
By the choice of $\d_1(x),\d_2(x)$ and $\d(x)$ and also by Remark \ref{R:ForLargerLambda1} in subsection \ref{S:HyperbolicBlocks}, it is obvious that by choosing  $r'$ small enough (only depend on $n,m,l_0$), $f_p$'s can be adapted into the setting in section \ref{S:InvariantManifolds} and one can derive a sequence of differentiable functions $h_i^u:\tilde B^u(0,r'l_{i})\to \tilde B^s(0,r'l_{i})$ for $i\ge 0$ with
\begin{itemize}
\item[(h1)] $h^u_i(0)=0, Dh^u_i(0)=0, \|Dh^u_i\|\le \frac1{10}, \|D^2h_i\|\le const\cdot l_i$ where $const$ depends on $n,,m,l_0,\d(x)$ only;
\item[(h2)] Set $W_i^u=graph(h_i^u)$, then for any $x,y\in W_i^u$ $\tilde f_{i-1}^{-1} x$ and $\tilde f_{i-1}^{-1} y$ exist and
$$|\tilde \pi^u\tilde f_{i-1}^{-1} x-\tilde \pi^u\tilde f_{i-1}^{-1} y|\le (e^{\frac{99}{100n}-2\d(x)}-2r')^{-1}|\tilde \pi^u x-\tilde \pi^u y|.$$
\end{itemize}

Now, together with Corollary \ref{C:LocalInvariantManifolds}, it is seen that Properties (1), (2) and (3) are satisfied if one takes $r$ small enough (only depending on $\l_0)$ and set
$$W^u_r(x)=\Phi_0(graph(h_0^u))\cap Exp_x(E^u_x(r)\oplus E^s_x)\text{ for }x\in\overline{\G_{n,m,l_0}}-\G_{n,m,l_0}$$  and $$W^u_r(x)=W^u_{loc}(x)\cap  Exp_x(E^u_x(r)\oplus E^s_x)\text{ for }x\in\G_{n,m,l_0}.$$
Then,  Property (4) is an immediate consequence of (3) if one takes $\rho<\frac14r$.

And furthermore, by choosing a smaller $r$, one has
$$\|Dh_x^u\|\le \frac1{20}\left(6(M_2+1)l_0^6e^{12\d(x)}\right)^{-1}.$$
Then,  all these $f^{-i}W^u_r(x)$'s are close enough to $Exp_{f^{-i}x}E^u_{f^{-i}x}$ in the $C^1$ topology to ensure that the similar argument used in the proof of Lemma \ref{L:ContinuityOfW} is applicable. Note that, instead of using the technique of switching charts,
 the uniformly continuity of the splitting $\mb H=E^u_\cdot\oplus E^s_\cdot$ alone will guarantee the same argument goes through.
 The proof is complete.
\end{proof}

\subsubsection{ Proof of Proposition \ref{P:AbsoluteContinuity}}\label{S:ProofofACW}

 Before going to the main proof of Proposition \ref{P:AbsoluteContinuity}, we first derive  technical estimates given in the following Lemma \ref{L:DetEstimateHolonomyMap1} and \ref{L:DetEstimateHolonomyMap2}. These lemmas follow from the disc lemmas proved in Section \ref{S:InvariantManifolds} once we set up the proper sequence of maps as in Section \ref{S:InvariantManifolds}.

We keep all the settings of Lyapunov charts in subsection \ref{S:HyperbolicBlocks}. For any $y\in B(x_0,\e)\cap \G_{m_0,n_0,l_0}$,  set up the maps by defining $r'_i=\frac34\d(x_0)l_0^{-1}e^{-i\d(x_0)}$ for $i\ge 0$, and then define
 \begin{itemize}
 \item[(i)] $g_{y,i}=\tilde f_{f^i(y)}=\Phi_{f^{i+1}(y)}^{-1}\circ f\circ \Phi_{f^i(y)}$;
 \item[(ii)] $\L_{y,i}=(D\tilde f_{f^i(y)})_0$, $G_{y,i}=\tilde f_{f^i(y)}-(D\tilde f_{f^i(y)})_0$;
 \item[(iii)] $\L^\tau_{y,i}=\L_{y,i}|_{\tilde E^\tau}$, $\tau=u,s$.
\end{itemize}
Note that in the following proof, one  goes forward only, so definitions of $g_{y,i}$ for $i<0$ is not needed. By Theorem \ref{T:LyapunovChart} and the choice of $\d_0(x_0),\d(x_0)$, one has that $g_{y,i}$'s satisfy conditions (I)-(III) in Section \ref{S:InvariantManifolds} with $\l_1=\frac1{n_0}-2\d_0(x_0)$ and $\d_1=\d_2=\d(x_0)$, and $\ell_i$ can be chosen as $\ell_i=l_0e^{i\d(x_0)}$. Since $Df$ is continuous and $\L$ is compact, (IV) in Section \ref{S:InvariantManifolds} follows from Theorem \ref{T:LyapunovChart}. Furthermore, in this setting we also have that $|G_{y,i}(0)|=0$ additional to Condition (II) in Section \ref{S:InvariantManifolds}.

Note that under this setting, $\tilde{\mc V}'=\mc V_0$, where $\mc V_0$ is the one defined in Section \ref{S:InvariantManifolds}. By Lemma \ref{L:HolonomyMap} (i), for any $\tilde v_1,\tilde v_2\in \tilde{\mc V}$, there exist $\tilde v_1',\tilde v_2'\in \tilde{\mc V}'$ such that  $$\Phi_{y}^{-1}\left(\Phi_{x_0}(graph(\tilde v_i))\right)\cap\tilde B(0,\frac34r_0)=graph(\tilde v_i'), i=1,2.$$

As the same as in Section  \ref{S:InvariantManifolds}, we denote $\tilde T_{y,i}$ and $\tilde T_{y,0}^n$ the graph transform operator and its iterations defined for $g_{y,i}$'s respectively. Then,  by applying Lemma \ref{L:Jacob1}, \ref{L:Jacob2} and Theorem \ref{T:LyapunovChart}, we obtain Lemma \ref{L:DetEstimateHolonomyMap1} and \ref{L:DetEstimateHolonomyMap2}. In fact,  one needs to compare $\det(L_{f^{n}y}|_{L_{y}^{-1}(graph(D\tilde v'_0(\tilde\pi^u\Phi_y^{-1}(x_i))))}$'s $i=1,2$, for Lemma \ref{L:DetEstimateHolonomyMap1}. It can be shown that
$$\frac{|\det(L_{f^{n}y}|_{L_{y}^{-1}(graph(D\tilde v'_0(\tilde\pi^u\Phi_y^{-1}(x_1))))}|}{|\det(L_{f^{n}y}|_{L_{y}^{-1}(graph(D\tilde v'_0(\tilde\pi^u\Phi_y^{-1}(x_2))))}|}\lessapprox 1+Ce^{-n\l_{1}},$$ where $C$ is a constant. To obtain this, one can employ the same argument as in the proof of Lemma \ref{L:Jacob1} and \ref{L:Jacob2}, and use the facts that
$$d\left(graph\left(\left(D\tilde T^{k}_{0}(\tilde v'_0)\right)_{\tilde\pi^u\Phi_{f^{n}y}^{-1}(f^{n}x_1)}\right),graph\left(\left(D\tilde T^{k}_{0}(\tilde v'_0)\right)_{\tilde\pi^u\Phi_{f^{n}y}^{-1}(f^{n}x_2)}\right)\right)\approx e^{-n\l_{1}},$$
and $\|L_{f^{n}y} \|,\|(L_{f^{n}y})^{-1} \|\le l_{0}e^{n\d_{1}}$. Similar arguments hold for Lemma \ref{L:DetEstimateHolonomyMap2}, and the details in proving these lemmas are omitted.
\begin{lemma}\label{L:DetEstimateHolonomyMap1}
There exists a constant $C\ge 1$ such that for any $y\in B(x_0,\e)\cap \G_{m_0,n_0,l_0}$, $\tilde v'_0\in \tilde{\mc V}'\cap C^1$, $x_1, x_2\in \Phi_y\left(graph(\tilde v'_0)\right)$, and $n\ge 1$, if $f^k(x_1),f^k(x_2)\in \Phi_{f^ky}\left(graph(\tilde T_0^k(\tilde v'_0))\right)$ for all $0\le k\le n$, then
$$C^{-1}\le \frac{\det\left(Df_{x_1}^n\big|_{L_{y}^{-1}(graph(D\tilde v'_0(\tilde\pi^u\Phi_y^{-1}(x_1))))}\right)}
{\det\left(Df_{x_2}^n\big|_{L_{y}^{-1}(graph(D\tilde v'_0(\tilde\pi^u\Phi_y^{-1}(x_2))))}\right)}\le C.$$
\end{lemma}

\begin{lemma}\label{L:DetEstimateHolonomyMap2}
There exists a constant $C\ge 1$ such that for any $y\in B(x_0,\e)\cap \G_{m_0,n_0,l_0}$, $\tilde v'_{1},\tilde v'_{2}\in \tilde{\mc V}'\cap C^1$, and $n\ge 1$
$$C^{-1}\le \frac{\det\left(Df_{x_1}^n\big|_{L_{y}^{-1}(graph(D\tilde v'_1(\tilde\pi^u\Phi_y^{-1}(x_1))))}\right)}
{\det\left(Df_{x_2}^n\big|_{L_{y}^{-1}(graph(D\tilde v'_2(\tilde\pi^u\Phi_y^{-1}(x_2))))}\right)}\le C,$$
where $x_i=\Phi_y\left(\tilde W^s_{loc}(y)\cap graph(\tilde v_{i}')\right)$ for $i=1,2$.
\end{lemma}

\bigskip

 Now we are ready to prove  Proposition \ref{P:AbsoluteContinuity}. We will keep the notations introduced previously in subsection \ref{S:ACW}.
We will show that there exists a constant $C>1$ such that the following holds: For
  all $\tilde v_1,\tilde v_2\in\tilde{\mc V}$, denote that $D_i=\Phi_{x_0}(graph(\tilde v_i))$ for $i=1,2$, and let $A\subset (D_1\cap \cup_{y\in B(x_0,\e)\cap \G_{n_0,m_0,l_0}}W^s_{loc}(y))$ be an arbitrary Borel set with $0$ Lebesgue measure, then for any $\e_1>0$
  \begin{equation}\label{E:ABS}
  m_{D_2}(T_{v_1,v_2}(A))\le C\e_1,
  \end{equation}
  where  $m_{D_i}$ are Lebesgue measures on the finite dimensional embedded discs $D_i$, $i=1,2$.

Since the Lebesgue measure is regular, every bounded Borel set can be approximated from inside by a compact subset in measure, it suffices to prove (\ref{E:ABS}) for compact sets.
For any given small $\e_1>0$, there exists a neighborhood $U$ of $A$ in $D_1$ such that $m_{D_1}(U)\le \e_1$. For any $y\in B(x_0,\e)\cap \G_{n_0,m_0,l_0}$ with $W^s_{loc}(y)\cap D_1\in A$, by (i) of Lemma \ref
{L:HolonomyMap}, there exist $\tilde v'_{1,y}\in\tilde{\mc V}'$ such that
$$\Phi_y^{-1}(D_1)\cap \tilde B(0,\frac34r_0)=graph(\tilde v_y').$$
Denote that $\tilde v_{1,y,n}'=\tilde T^n_{y,0}(\tilde v'_{1,y})$ for $n\ge 0$. Since $dist(A,\partial U)>0$, by Lemma \ref{L:Graph1}, there exists $N_1$ such that for any $n>N_1$ if $f^n(x)\in \Phi_{f^n(y)}\left(graph(\tilde v_{1,y,n}')\right)$ then $x\in U$.

By definitions of $\d$ and $\d_0$, we have that $\d(x_0)<<\frac1{2n_0}<\l(x_0)$. Then,  by the properties of the Lyapunov charts and the choice of $y$, there exists $N_{2}>N_{1}$ such that for $n>N_{2}$
$$\Phi_{f^n(y)}^{-1}B^u\left(f^n(y),\frac34r_0e^{-\frac{n}{2n_0}}\right)\subset \tilde B^u(0,r'_n).$$
Therefore, for such $n$, we define the map $$\phi_{f^n(y)}:B^u\left(f^n(y),\frac34r_0e^{-\frac{n}{2n_0}}\right)\to f^n(D_1)$$ by letting
$$\phi_{f^n(y)}(w)=\Phi_{f^n(y)}\left((\Phi_{f^n(y)})^{-1}w,\tilde v'_{1,n,y}\left((\Phi_{f^n(y)})^{-1}w\right)\right)$$
for $w\in  B^u\left(f^n(y),\frac34r_0e^{-\frac{n}{2n_0}}\right)$. For this $\phi_{f^n(y)}$, we have the following lemma:
 \begin{lemma}\label{L:LocalHomeomorphism}
 For any $\e_2>0$, there is an $N_3>N_2$ such that for any $n>N_3$ and each $y\in B(x_0,\e)\cap\G_{n_0,m_0,l_0}$, $\phi_{f^n(y)}$ is well-defined and the following holds:
 \begin{itemize}
 \item[(i)] $f^n(W^s_{loc}(y)\cap D_1)\in \phi_{f^n(y)}\left( B^u\left(f^n(y),\frac34\e_2r_0e^{-\frac{n}{2n_0}}\right)\right)$;
 \item[(ii)] For any  $w_1,w_2\in B^u(f^n(y),\frac34r_0e^{-\frac{n}{2n_0}})$
 $$1-\e_2\le \frac{|\phi_{f^n(y)}(w_2)-\phi_{f^n(y)}(w_1)|}{|w_2-w_1|}\le 1+\e_2.$$
 \end{itemize}
 \end{lemma}
 \begin{proof}
 It is obvious that such $\phi_{f^{n}(y)}$ is well-defined for all $n\ge 1$.

For (i), the existence of such $N_{3}$ follows from Corollary \ref{C:LocalInvariantManifolds} and Theorem \ref{T:LyapunovChart} immediately if one noting that
 $$|f^n(W^s_{loc}(y)\cap D_1)-f^{n}(y)|\lessapprox e^{-\frac{n}{n_{0}}}.$$

For (ii), the existence of such $N_{3}$ follows from Lemma \ref{L:DiscLemma3}. For any $w_1,w_2\in B^u(f^n(y),\frac34r_0e^{-\frac{n}{2n_0}})$, one obtains that
\begin{align*}
&\frac{|\phi_{f^n(y)}(w_2)-\phi_{f^n(y)}(w_1)|}{|w_2-w_1|}
=\frac{|\Phi_{f^n(y)}\Phi^{-1}_{f^n(y)}(w_2,v'_{1,y,n}(w_2))-\Phi_{f^n(y)}\Phi^{-1}_{f^n(y)}(w_1,v'_{1,y,n}(w_1))|}{|w_2-w_1|}\\
&=\frac{|w_2+\Phi_{f^n(y)}(\tilde v'_{1,y,n}(\Phi^{-1}_{f^n(y)}w_2))-(w_1+\Phi_{f^n(y)}(\tilde v'_{1,y,n}(\Phi^{-1}_{f^n(y)}w_1)))|}{|w_2-w_1|},
\end{align*}
where $v'_{1,y,n}=\Phi_{f^n(y)}(\tilde v'_{1,y,n})$.\\

Also note that for $n\ge \frac{\log l_{0}}{\d_{1}}$
\begin{align*}
&|\Phi_{f^n(y)}(\tilde v'_{1,y,n}(\Phi^{-1}_{f^n(y)}w_2))-\Phi_{f^n(y)}(\tilde v'_{1,y,n}(\Phi^{-1}_{f^n(y)}w_1))|\\
\le&\sqrt3 \sup_{\tilde w\in \Phi_{f^n(y)}^{-1}(B^u(f^n(y),\frac34r_0e^{-\frac{n}{2n_0}}))}\|D\tilde v'_{1,y,n}(\tilde w)\|l(f^n(y))|w_2-w_1|\\
\le&\sqrt3\sup_{\tilde w\in \tilde B(0,\frac34r_0e^{n(2\d_1-\frac{1}{2n_0})})}\|D\tilde v'_{1,y,n}(\tilde w)\|l_0e^{n\d_1}|w_2-w_1|.
\end{align*}
And then by Lemma \ref{L:DiscLemma3}, one has
$$\sup_{\tilde w\in \tilde B(0,\frac34r_0e^{n(2\d_1-\frac{1}{2n_0})})}\|D\tilde v'_{1,n,y}(\tilde w)\|\le C'e^{-n\min\{\frac{49}{50n_0},\frac{1}{2n_0}-3\d_1\}}.$$
By the choice of $\d_1$, for any $\e_2$, there exists $N_3>\max\{N_1,\frac{\log l_{0}}{\d_{1}}\}$ such that (ii) is satisfied. The proof is complete.
\end{proof}

The following technical lemma is a modification of the overcovering Lemma from \cite{Pu}. Since it is under a different setting, we include the proof here for the sake of completeness.
\begin{lemma}\label{L:OverCover}
For a fixed small  $\e_{2}\in(0,\frac1{24})$ (where $\e_{2}$ is as in Lemma \ref{L:LocalHomeomorphism}), there exists $N_4>N_{3}$ such that for any $n>N_4$ there exist a finite set $\{y_i\}_{i\in I_{n}}\subset B(x_0,\e)\cap\G_{n_0,m_0,l_0}$ such that the following are satisfied:
\begin{itemize}
\item[(i)] For any $w_1,w_2\in B^u(f^n(y_i),\frac3{4}r_0e^{-\frac{n}{2n_0}})$
$$\frac34 \le \frac{|\phi_{f^n(y_i)}(w_2)-\phi_{f^n(y_i)}(w_1)|}{|w_2-w_1|}\le\frac43; $$
 \item[(ii)] For any $x\in A$, there exists $i\in I_{n}$ such that
 \begin{equation*}f^n(x)\in \phi_{f^n(y_i)}(B^u(f^n(y_i),\frac3{16}r_0 e^{-\frac{n}{2n_0}}));\end{equation*}
 \item[(iii)] There are at most $32^{\dim E^u(x_0)}$ elements of
 \[\left\{\phi_{f^n(y_i)}(B^u(f^n(y_i),\frac3{16}r_0 e^{-\frac{n}{2n_0}}))\right\}_{i\in I_{n}}
 \] having nonempty intrsection.
 \end{itemize}
\end{lemma}
\begin{proof}
First note that $D_1$ is a finite dimensional embedded compact disc and $f$ is an injective and continuous map, so $f^n(D_1)$ is an embedded disc with the same dimension. By Lemma \ref{L:LocalHomeomorphism}, for any $y\in B(x_0,\e)\cap\G_{n_0,m_0,l_0}$ and $n>N_1$, $\phi_{f^n(y)}$ is a local homeomorphism from $B^u(f^n(y_i),\frac34r_0e^{-\frac{n}{2n_0}})$ to $f^n(D_1)$.\\
A natural metric $d_{arc}$ on $f^n(D_1)$ is that $d_{arc}(x,y)$ equals the infimum of the arc-length of $C^1$ paths in $f^n(D_1)$ connecting $x$ and $y$. By Zorn's lemma, there exists a maximal $\frac1{8}r_0e^{-\frac{n}{2n_0}}$-separated set $A'_{n}\subset f^n(A)$ such that $d_{arc}(x',y')\ge\frac1{8}r_0e^{-\frac{n}{2n_0}}$  for any distinct $x',y'\in A'_{n}$, and for any $x\in f^{n}(A)$ there exists an $x'\in A'_{n}$ such that $d_{arc}(x,x')\le \frac1{8}r_0e^{-\frac{n}{2n_0}}$. If we use the notation $B_d$ to denote the ball in $f^n(D_1)$ with respect to the metric $d_{arc}$, then $\{B_d(x',\frac1{8}r_0e^{-\frac{n}{2n_0}})\}_{x'\in A'}$ covers $f^{n}(A)$ and $\{B_d(x',\frac1{16}r_0e^{-\frac{n}{2n_0}})\}_{x'\in A'}$ are disjoint. \\

By the compactness of $f^n(D_1)$, $A'_{n}$ is a finite set. So one can label all the elements of the set as $A'_{n}=\{x'_i\}_{i\in I_{n}}$ with a finite index set $I_{n}$. For any $x'_i$, there exists a $y_i\in B(x_0,\e)\cap \G_{n_0,m_0,l_0}$ such that $f^{-n}(x'_i)=W^s_{loc}(y_i)\cap D_1$. From Lemma \ref{L:LocalHomeomorphism}, since $\e_2<\frac1{24}$, by a straightforward computation, we obtain that if $n>N_2$, the following holds
\begin{itemize}
\item[(i)] For any $w_1,w_2\in B^u(f^n(y_i),\frac3{4}r_0e^{-\frac{n}{2n_0}})$
$$\frac34 \le \frac{|\phi_{f^n(y_i)}(w_2)-\phi_{f^n(y_i)}(w_1)|}{|w_2-w_1|}\le\frac43; $$
\item[(ii)] $\{\phi_{f^n(y_i)}(B^u(f^n(y_i),\frac3{16}r_0e^{-\frac{n}{2n_0}}))\}_{i\in I}$ covers $f^n(A)$;
\item[(iii)] Elements in $ \{\phi_{f^n(y_i)}(B^u(f^n(y_i),\frac1{24}r_0e^{-\frac{n}{2n_0}}))\}_{i\in I}$ are disjoint.
\end{itemize}
Next we estimate the maximum number of the covering balls having nonempty intersection.
Suppose that $\cap_{1\le i\le k}\phi_{f^n(y_i)}(B^u(f^n(y_i),\frac3{16}r_0e^{-\frac{n}{2n_0}}))\neq \emptyset$. Then, by a straightforward computation, one has that for all $ i=1,\ldots,k$
$$\phi_{f^n(y_i)}(B^u(f^n(y_i),\frac1{24}r_0e^{-\frac{n}{2n_0}}))\subset \phi_{f^n(y_1)}(B^u(f^n(y_1),\frac34r_0e^{-\frac{n}{2n_0}})).$$
Note that for $i=1,\ldots,k$, $\phi_{f^n(y_1)}^{-1}\phi_{f^n(y_i)}(B^u(f^n(y_i),\frac1{24}r_0e^{-\frac{n}{2n_0}}))$ contains a ball with radius $\frac{9}{512}r_0e^{-\frac{n}{2n_0}}$, and these balls are disjoint and belong to a $\frac34r_0e^{-\frac{n}{2n_0}}$-ball. So $k\le 32^{\dim E^u(x_0)}$. The proof is complete.
\end{proof}

In the following, we define $\psi$ for $D_2$ as the same as $\phi$ defined for $D_1$. Since $D_1,D_2$ are arbitrarily chosen, one can assume $\psi$ has the same properties of $\phi$ as shown in Lemma \ref{L:LocalHomeomorphism} and \ref{L:OverCover} with given $\e_{1}$ and $\e_{2}$ (maybe with larger $N_1,N_2,N_3,N_{4}$). By Lemma \ref{L:LocalStableManifolds}, for any $x\in  \phi_{f^n(y)}(B^u(f^n(y),\frac{3}{16}r_0e^{-\frac{n}{2n_0}}))$ on which $T_{v'_{1,y,n},v'_{2,y,n}}$ is well-defined, we have
$$|T_{v'_{1,y,n},v'_{2,y,n}}(x)-x|\le r_{0}l_{0}(e^{-\frac1{n_{0}}}+2\d_{2})^{n}.$$
Therefore, there exists $N_{5}>N_{4}$ such that for any $n>N_5$ and each $y\in B(x_0,\e)\cap \G_{n_0,m_0,l_0}$
\begin{equation}\label{E:Inclision}
T_{v_{1,y,n},v_{2,y,n}}(\phi_{f^n(y)}(B^u(f^n(y),\frac{3}{16}r_0e^{-\frac{n}{2n_0}})))\subset \psi_{f^n(y)}(B^u(f^n(y),\frac{3}{8}r_0e^{-\frac{n}{2n_0}})).
\end{equation}

For an arbitrary fixed $n>N_5$, let $\{y_i\}_{i\in I_{n}}$ be the finite set derived from Lemma \ref{L:OverCover}. Since $\cup_{i\in I_{n}} f^{-n}(\phi_{f^n(y_i)}(B^u(f^n(y_i),\frac{3}{16}r_0e^{-\frac{n}{2n_0}})))\subset U$, we have that
\begin{equation}\label{E:Measure1}
\e_1\ge m_{D_{1}}(U)\ge \frac{\sum_{i\in I_{n}} m_{D_{1}}(f^{-n}(\phi_{f^n(y_i)}(B^u(f^n(y_i),\frac{3}{16}r_0e^{-\frac{n}{2n_0}}))))}{32^{\dim E^u(x_0)}}.
\end{equation}
Here we divide the right side by $32^{\dim E^u(x_0)}$ to compensate for potential over-counting due to overlaping. On the other hand, (\ref{E:Inclision}) implies that
\begin{equation}\label{E:Measure2}
m_{D_{2}}(T_{v_{1,n,y},v_{2,n,y}}(A))\le \sum_{i\in I_{n}} m_{D_{2}}(f^{-n}(\psi_{f^n(y_i)}(B^u(f^n(y_i),\frac{3}{8}r_0e^{-\frac{n}{2n_0}})))).
\end{equation}
 Comparing (\ref{E:Measure1}) and (\ref{E:Measure2}), one can see that to complete the proof, it remains to produce a constant $C$ which does not depend on $n$, $i$ or $I_{n}$ such that for each $i$,
 \begin{align}\begin{split}\label{E:Measure3}
 &m_{D_{2}}(f^{-n}(\psi_{f^n(y_i)}(B^u(f^n(y_i),\frac{3}{8}r_0e^{-\frac{n}{2n_0}}))))\le C
 m_{D_{1}}(f^{-n}(\phi_{f^n(y_i)}(B^u(f^n(y_i),\frac{3}{16}r_0e^{-\frac{n}{2n_0}})))).
 \end{split}
\end{align}
By Lemma \ref{L:LocalHomeomorphism} (ii), we have that
$$\left(\frac{9}{8}\right)^{\dim E^u(x_0)}\le \frac{m_{f^{n}D_{2}}\left(\psi_{f^n(y_i)}\left(B^u\left(f^n(y_i),\frac{3}{8}r_0e^{-\frac{n}{2n_0}}\right)\right)\right)}
{m_{f^{n}D_{1}}\left(\phi_{f^n(y_i)}\left(B^u\left(f^n(y_i),\frac{3}{16}r_0e^{-\frac{n}{2n_0}}\right)\right)\right)}
\le\left(\frac{32}{9}\right)^{\dim E^u(x_0)}$$
Then,  (\ref{E:Measure3}) follows from Lemma \ref{L:DetEstimateHolonomyMap1} and \ref{L:DetEstimateHolonomyMap2}: (In the following estimate, we set $x_\tau=W^s_{loc}(y_i)\cap D_{\tau}$, $\tau =1,2$, and use $C$ as a generic constant depending on the system constants only.)
{\small\begin{align*}
&m_{D_{2}}(f^{-n}(\psi_{f^n(y_i)}(B^u(f^n(y_i),\frac{3}{8}r_0e^{-\frac{n}{2n_0}}))))\\
=&\int_{\psi_{f^n(y_i)}\left(B^u\left(f^n(y_i),\frac{3}{8}r_0e^{-\frac{n}{2n_0}}\right)\right)}
\left|\det\left(Df_x^{-n}\Big|_{\mc T_x\psi_{f^n(y_i)}\left(B^u\left(f^n(y_i),\frac{3}{8}r_0e^{-\frac{n}{2n_0}}\right)\right)}\right)\right|dm_{f^{n}D_{2}}(x)\\
\le&Cm_{f^{n}D_{2}}\left(\psi_{f^n(y_i)}\left(B^u\left(f^n(y_i),\frac{3}{8}r_0e^{-\frac{n}{2n_0}}\right)\right)\right)
\left|\det\left(Df_{x_2}^{n}\Big|_{\mc T_{x_2}D_2}\right)\right|^{-1}\\
\le& Cm_{f^{n}D_{1}}\left(\phi_{f^n(y_i)}\left(B^u\left(f^n(y_i),\frac{3}{16}r_0e^{-\frac{n}{2n_0}}\right)\right)\right)
\left|\det\left(Df_{x_1}^{n}\Big|_{\mc T_{x_1}D_1}\right)\right|^{-1}\\
\le&C\int_{\phi_{f^n(y_i)}\left(B^u\left(f^n(y_i),\frac{3}{16}r_0e^{-\frac{n}{2n_0}}\right)\right)}
\left|\det\left(Df_x^{-n}\Big|_{\mc T_x\phi_{f^n(y_i)}\left(B^u\left(f^n(y_i),\frac{3}{16}r_0e^{-\frac{n}{2n_0}}\right)\right)}\right)\right|dm_{f^{n}D_{1}}(x)\\
=&Cm_{D_{1}}(f^{-n}(\phi_{f^n(y_i)}(B^u(f^n(y_i),\frac{3}{16}r_0e^{-\frac{n}{2n_0}})))).
\end{align*}}
The proof of Proof of Proposition \ref{P:AbsoluteContinuity} is complete.

$\hfill$ $\square$


\subsection{Proof of Theorem \ref{T:ErgodicAttractor}}\label{S:ExiErgodicAttractor}
Throughout this section, $\mu$ is assumed to be a hyperbolic and SRB measure. In constructing ergodic attractors, we take the approach used in \cite{Pu}.

 Let $\mc G_{\pm}$ and $\mc G$ be the sets of points $x\in \Lambda$ such that for all continuous function $\phi:U\to \mc R$ where $U$ is the basin of $\Lambda$, the following three Birkhoff averages
\begin{align*}
&\overline{\phi_{\pm}}(p):= \lim_{n\to\pm\infty}\frac{1}{|n|}\sum_{i=0}^{n-1}\phi(f^ip),\\
&\overline{\phi}(p):=\lim_{n\to\infty}\frac1{2n-1}\sum_{i=n-1}^{n-1}\phi(f^ip),
\end{align*}
exist. By \cite{Peterson}, the set $\mc G^*(\subset\mc G\cap\mc G_{-}\cap\mc G_{+})$ on which the above three Birkhoff averages equal is of full $\mu$-measure and is $f$-invariant.

Following the argument in subsection \ref{S:HyperbolicBlocks}, $supp(\mu)$ can be divided into at most countable many hyperbolic blocks of positive $\mu$-measure. Apply Lemma \ref{L:UnstableManifoldsExtension} to a given $\G_{n_0,m_0,l_0}$ and choose the diameter of the hyperbolic block $\e_0$ small enough (e.g. $<\frac12\rho$ where $\rho$ is as in Lemma \ref{L:UnstableManifoldsExtension}). Arbitrarily take one of the $\e_0$-hyperbolic blocks and denote it by $P$. Without losing of generality, we set $P\subset \mc G^*$, since $\mc G^*$ is an $f$-invariant full $\mu$-measure set. Denote $\overline P$ the closure of $P$, which is a closed subset of $\Lambda$ thus is a compact set.

By Lemma \ref{L:UnstableManifoldsExtension}, we have that: a) $W^u_{r}(y_1)\cap W^u_{r}(y_2)\cap \overline P=\emptyset$ or $W^u_{r}(y_1)\cap \overline P=W^u_{r}(y_2)\cap \overline P$ for any $y_1,y_2\in\overline P$; and b) $W^u_{r}$ varies continuously on $\overline P$.


Thus, also by the compactness of $\overline P$, we have that
 $$\{W^u_{r}(y)\cap \overline P|\ y\in \overline P\}\text{ forms a Borel measurable partition of }\overline P.$$
We relabel the elements  and denote the partition by $\zeta=\{\zeta_\a\}_{\a\in \mc I}$ where $\mc I$ is the index set. Denote $\pi$ the canonical projection from $\overline P$ to $(\overline P/\zeta)$, where $(\overline P/\zeta)$ is the quotient space induced by partition $\zeta$. Since the partition $\zeta$ is Borel measurable, $\pi$ is Borel.  Now we are able to disintegrate $\mu|_{\overline P}$ into $\mu_\a\otimes \eta$, where $\mu_\a$ is the conditional probability measure of $\mu|_{\overline P}$ on $\zeta_\a$ and $\eta=\mu|_{\overline P}\pi^{-1}$ is the pullforward measure. So for any Borel set $A\in \overline P$,
\begin{equation}\label{E:FubiniEqu}
\mu(A)=\int_{\overline P/\zeta}\mu_\a(A\cap \zeta_\a)d \eta(\zeta_\a).
\end{equation}
Note that $\overline P$ is the closure of $P$, thus the above Fubini type equation holds for any measurable subset $A\subset P$, in particular, it holds for $P$. Since $\mu$ is SRB, we have that $\mu_\a<< \rho_\a$ for $\eta$-a.e. $\zeta_{\a}$, where $\rho_\a$ is a Lebesgue measure on $\zeta_\a$. Consider the subset $\zeta_\a'\subset \zeta_\a$ on which the Radon Nikodym derivative $d\mu_\a/d\rho_\a>0$, which means that $\mu_\a$ and $\rho_\a$ are equivalent on $\zeta_\a'$. And by Equation (\ref{E:FubiniEqu}), we have that
\begin{equation}\label{E:FubiniEquP}
\mu(P)=\int_{\overline P/\zeta}\mu_\a(P\cap\zeta_\a')d\eta(\zeta_\a).
\end{equation}
So,  for the sake of convenience, we assume $P=\cup_{\a\in\mc I}(P\cap \zeta_\a')$ and $\zeta_\a=\zeta_a'$ in the rest of the proof without losing any generality.

Now we will construct a countable partition $\{\L_n\}_{n\ge 1}$ of $P$, each of the saturated sets of which ($:=\cup_{n\in\mb Z}f^n(\L_n)$)  will produce an ergodic attractor.
We first define an equivalence relation in $\zeta$ as the following:
\vskip0.1in

$\zeta_{\a_1}\sim \zeta_{\a_2}$ if there exist finite indexes $\{\b_1,\cdots,\b_k\}\subset \mc I$ satisfying that a) $\b_1=\a_1,\b_k=\a_2$; and b) for any $1\le i\le k-1$, there exist $x_i\in P\cap \zeta_{\b_i}$ and $x_{i+1}\in P\cap \zeta_{\b_{i+1}}$ such that $W^s_{loc}(x_i)\cap P=W^s_{loc}(x_{i+1})\cap P$.

\vskip0.1in

It is easy to see this is a well-defined equivalent relation. For a given $x_0\in P$, let $T$ be the projection which maps $P$ into $W^u_{loc}(x_0)$ by sliding along $\{W^s_{loc}(x)\}_{x\in P}$. Such projection is well-defined because of Lemma \ref{L:HolonomyMap}. By the definition of $\sim$, it is easy to see that $T$ maps equivalent classes into disjoint sets. For $x\in P\cap \zeta_\a$, $\mu_\a(\zeta_\a)>0$ implies $\rho_\a(\zeta_\a)>0$, and then the absolute continuity of the holonomy map induced by stable manifolds (Proposition \ref{P:AbsoluteContinuity}) implies that $\rho_{x_0}(T(\zeta_\a))>0$. Therefore,  there are only countable many elements of these equivalent classes, say $\{\Delta_n\}_{1\le n<a,n\in \mb N}$ (where $a\in \mb N\cup\infty$ and $a=\infty$ is permitted), such that for each $\Delta_n$ there exists $\zeta_\a\subset \Delta_n$ with $\mu_\a(\zeta_\a)>0$. Furthermore, by (\ref{E:FubiniEquP}), we have that
$$\mu(P\setminus \cup_{n=1}^{a-1}\D_{n})=0\ i.e.\ \mu(P)=\sum_{n=1}^{a-1}\mu(\Delta_n).$$

Noting that $f|_{\L}$ is invertible, we set $\Lambda_n=\cup_{i\in \mb Z}f^i (\Delta_n)$. Since there exists $\zeta_\a\in \zeta$ such that $\rho_\a(\Delta_n\cap \zeta_\a)>0$, by Definition \ref{D:Observable} (observable set) and Proposition \ref{P:AbsoluteContinuity} (absolute continuity of stable foliations), it is easy to see that $W^s(\Delta_n):=\cup_{x\in \Delta_n}W^s(x)$ contains an observable set since $\cup_{x\in \Delta_n}W^s_{loc}(x)$ does.
Since $\Delta_n\subset P\subset \mc G^*$ and $\mc G^*$ is $f$-invariant, we have that $\L_n\subset \mc G^*$. For any continuous function $\phi:H\to \mb R$, the three Birkhoff averages $\overline{\phi_{\pm}},\overline{\phi}$ exit and coincide at all points $x\in \L_n$. By the continuity of $\phi$ and also noting that for any $\zeta_\a\subset \Delta_n$ one has that $\zeta_\a\subset W^u_{loc}(x)$ for some $x\in P\cap \zeta_\a$, then we obtain that $\overline{\phi_{-}}$ is constant on $\zeta_\a$. Therefore, by the definition of equivalent relation, we have that $\overline{\phi_+}(x_1)=\overline{\phi_+}(x_2)$ for any $x_1\in \zeta_{\a_1},x_2\in \zeta_{\a_2}$ with $\zeta_{\a_1},\zeta_{\a_2}\in \Delta_n$. \\

Thus, for all continuous function $\phi:H\to \mb R$, we have that $\overline{\phi_{\pm}}$ and $\overline\phi$ are measurable functions being constants and equal on $\Delta_n$, thus so is on $\L_n$, the value of which is denoted by $\phi_{ave}(n)$. Moreover, by continuity of $\phi$, $\phi_+(x)=\phi_{ave}(n)$ for all $x\in W^s(\L_n):=\cup_{y\in\L_n}W^s(y)$.\\

Let $\nu_n=(\mu(\L_n))^{-1}\mu|_{\L_n}$, then $\nu_n$ is an $f$-invariant Borel probability measure on $\L_{n}$. We have shown that for any continuous function $\phi:\L\to \mb R$ and $\nu_n$-a.e. $x\in \L_n$
$$\overline{\phi_{\pm}}(x)=\overline{\phi}(x)=\phi_{ave}(n)=\int\phi d\nu_n,$$
where the last equality follows from the Birkhoff Ergodic Theorem.

By \cite{Peterson}, one has that the Birkhoff averaging with respect to an $f$-invariant probability measure $\nu$ gives a continuous onto projection from $\mc L^1(\L,\nu)$ to the space of $f$-invariant $\nu$-integrable functions with the $\mc L^1(\L,\nu)$-norm being preserved. Note that $\nu_n$ can be also seen as a Borel measure on $\L$. Given any $\psi\in \mc L^1(\L_n,\nu_n)$, $\psi$ can be extended to a measurable function defined on $\L$ simply by assigning constant value on points outside $\L_n$, which is obvious in $\mc L^1(\L,\nu_n)$. Since the space of continuous functions $C(\L)$ is dense in $\mc L^1(\L,\nu_n)$ and the images of $C(\L)$ under the Birkhoff averaging with respect to $\nu_n$ are all $\nu_n$-a.e. constant functions, we have that each $f$-invariant function is $\nu_n$-a.e. constant. Hence,  $\nu_n$ is ergodic.

Also note that, up to a $\mu$-null set, $P$ is the union of at most countable many $\Delta_n$'s. Thus,  $P$ is contained in the union of at most countable many ergodic attractors all of which are contained in $supp\ \mu$. Since $supp\ \mu$ is the union of at most countable many such hyperbolic blocks, and each hyperbolic block gives at most countable many ergodic attractors, we have that, up to a set of $\mu$-null set, $supp\ \mu$ is a union of at most countable many ergodic attractors. This completes the proof of Theorem \ref{T:ErgodicAttractor}.

$\hfill\square$

\subsection{Proof of Theorem \ref{T:SRBCont}}
Noting that the basin of two different ergodic attractors are disjoint. And by above argument, any ergodic hyperbolic SRB measure will produce an ergodic attractor whose basin contains an observable set. Then,  the countability of cardinality of hyperbolic ergodic SRB measures can be proved by applying the following lemma:
\begin{lemma}\label{L:ContObsSet}
The cardinality of any collection of disjoint observable sets in a separable Banach space is at most countable.
\end{lemma}
\begin{proof}
Note that the collection of all finite dimensional subspaces of a separable Banach space, say $\mc K$,  forms a separable complete metric space which is endowed by Kato's gap metric. Actually, $\mc K=\cup_{n=1}^\infty \mc K_n$ where $\mc K_n$ is the collection of all $n$-dimensional subspaces. Each $\mc K_n$ is a separable complete metric space in the Kato's metric, and the distance between elements from $\mc K_n$ and $\mc K_m$ are $\ge 1$ if $n\neq m$. Let $\{E_j\}_{j\ge 1}$  be a countable dense set in $\mc K$, and let $\rho_j$ be a Lebesque measure on $E_j$.  Let $\{K_i\}_{i\in\mc I}$ be a collection of disjoint observable set, we will show that the index set $\mc I$ contains at most countable many elements. By the definition of observable set, we have that
$$\text{ for any } i\in\mc I, \text{ there exists } E_j \text{ such that }\rho_j(E_j\cap K_i)>0.$$
 For any $i\in \mc I$, we pick one of the $E_j$'s above and assign $j$ to $i$, which induce a map $\varpi:\mc I\to \mb N$.\\
  On the other hand, for any $j\in \mb N$ and $i\in\varpi^{-1}(j)$, $\rho_j(K_i\cap E_j)>0$, which implies that $\varpi^{-1}(j)$ contains at most countably many elements since $\rho_j$ is $\sigma$-finite. Thus,  $\mc I=\cup_{j\in\mb N}\varpi^{-1}(j)$ contains at most countably many elements.

\end{proof}


\section{Proof of Theorem \ref{T:SRBUnique} and \ref{T:SRBFinite}}\label{S:UniFiniSRB}
This section is decided into three parts: in the first part, we state some fundamental properties of uniformly hyperbolic system; In the second part, we prove  Theorem \ref{T:SRBFinite} by assuming Theorem \ref{T:SRBUnique}; in the last part, we prove  Theorem \ref{T:SRBUnique}.

\subsection{General Property of Uniformly Hyperbolic Attractor.}\label{S:AxiomA}
In this section, we assume that $f|\L$ is uniformly hyperbolic, and for the sake of simplicity, we use $E^{s}$ instead of $E^{cs}$. We first produce the stable and unstable manifolds theorem for uniformly hyperbolic case. These results can be derived from Lemma \ref{L:LocalUnstableManifolds} and \ref{L:LocalStableManifolds} once we construct proper Lyapunov charts. Under the uniformly hyperbolic setting, constructing Lyapunov charts becomes much simpler, the only issue is to orthogonalize the splitting $E^u\oplus E^s$ on each $x\in \L$. So we only online the main steps and summarise the main properties for charts.

Since the splitting $E^u\oplus E^s$ is continuous, without losing any generality, we assume $\dim E^u=m$ being a constant, otherwise one can work on each connected components of $\L$.

First, fix an orthogonal splitting $\mb H=\tilde E^u\oplus \tilde E^s$ and an orthonormal  basis $\{\tilde u_n\}_{n\ge 1}$ of $\mb H$ satisfying that $\tilde E^u=\spa\{\tilde u_1,\cdots,\tilde u_m\}$ and $\tilde E^s=\spa\{\tilde u_n\}_{n> m}$.

Second, for each $x\in \L$ take a unit basis $\{u_n(x)\}_{n\ge 1}$ of $\mb H$ satisfying that $ E^u_x=\spa\{u_1(x),\cdots,u_m(x)\}$, $ E^s_x=\spa\{u_n(x)\}_{n> m}$, and $\{u_1(x),\cdots,u_m\}$ and $\{u_n(x)\}_{n> m}$ are orthogonal sets respectively. Note that $u_{i}(\cdot)$ can be taken to be measurable by measurable selection theorem for which we refer to \cite{LL}; However, usually one can not expect to make $u_{i}(\cdot)$ to be continuous although the splitting $\mb H=E^{u}\oplus E^{s}$ is continuous.

For any $u=\sum_{n=1}^\infty c_i(x)u_i(x)$, define
$$L_xu=\sum_{n=1}^\infty c_i(x)\tilde u_i.$$
It is easily seen that for $\tau=u,s$
$$\|L_x|_{E^\tau_x}\|=\|L_x^{-1}|_{\tilde E^\tau}\|=1, \|L_x\|\le \sqrt 2 \max\{\|\pi^u_x\|,\|\pi^s_x\|\} \text{ and }\|L_x^{-1}\|\le 2.$$
Note that $E^u, E^s$ are varying uniformly continuously on $\L$, then for any $\e>0$, there exists $\d>0$ such that if $x,y\in \L$ satisfying $|x-y|<\d$, then
 for any $v\in \tilde E^\tau$, $\tau=u,s$, the following holds:
 \begin{itemize}
 \item[(a)] $(1-\e)|v|<|\tilde\pi^\tau L_y L_x^{-1} v|\le (1+\e)|v|$;
 \item[(b)] $|\tilde\pi^{\tau'} L_y L_x^{-1} v|\le \e|v|$ when $\tau'\neq \tau$.
 \end{itemize}
Now define
$$\Phi_x=Exp_x\circ L^{-1}_x\text{ and } \tilde f_x=\Phi_x^{-1}\circ f\circ \Phi_x.$$
Let $M=\max\left\{\sup_{x\in \L}\{\|D^2f_x\|\},\sup_{x\in \L}\{\|\pi^u_x\|\},\sup_{x\in \L}\{\|\pi^s_x\|\}\right\}$ and set $r_\d=(2\sqrt2M^2)^{-1}\d$. Then,  for any $x\in \L$ the bi-infinite sequence of maps $\{\tilde f_{f^ix}|_{\tilde B^u(0,r_\d)\oplus \tilde B^u(0,r_\d)}\}_{i\in \mb Z}$ satisfies all the conditions (I)-(III) of $g_i$'s defined in Section \ref{S:InvariantManifolds} with
$$\l_1=\l_0, G_i(0)=0, \d_1=0,\d_2=\d, \ell_i=2\sqrt2M^2.$$

We recall the definition of local stable and unstable manifolds here with a split difference: for any $x\in \L$
\begin{align*}
&W^s_\e(x)=\{y\in U|\ |f^ny-f^nx|\le \e \text{ for all } n\ge 0\}\\
&W^u_\e(x)=\{y\in U|\ |f^ny-f^nx|\le \e \text{ for all } n\le 0\},
\end{align*}
where $U$ is the basin of $\L$ and is open.

By applying Lemma \ref{L:LocalUnstableManifolds} and \ref{L:LocalStableManifolds}, and also employing the same argument used in proving Lemma \ref{L:ContinuityOfW}, we have the following results which are needed for this section.
\begin{lemma}\label{L:SUManifolds}
For any $\l\in(0, \l_0)$, there exists $\e_0>0$ such that for any $\e\in(0,\e_0]$
\begin{itemize}
\item[i)] $W^s_\e(x),W^u_\e(x)$ are $C^{2}$ embedded disks for all $x\in\L$ with $T_xW^\tau(x)=E^\tau(x)$, $\tau=s,u$;
\item[ii)] $|f^nx-f^ny|\le e^{-n\l}|x-y|$ for $y\in W^s_\e(x)$, $n\ge 0$, and \\
$|f^{-n}x-f^{-n}y|\le e^{-n\l}|x-y|$ for $y\in W^u_\e(x)$, $n\ge 0$;
\item[iii)] $W^s_\e(x),W^u_\e(x)$ vary continuously in $x$ (in $C^{1}$ topology).
\end{itemize}
\end{lemma}
By C3), we note that $W^u_\e(x)\subset \L$ for all $x\in \L$, and in the following proof we will use this fact without further explanations. And by the definition, $W^u_\e(x)$ and $W^s_\e(x)$ are closed subsets of $U$.\\

The following lemma provides local {\it canonical coordinates} on $\L$.
\begin{lemma}\label{L:CanonicalCoor}
For any small $\e\in(0,\e_0)$ there is a $\d>0$ such that $W^s_\e(x)\cap W^u_\e(y)$ consists of a single point $[x,y]$ whenever $x,y\in \L$ and $|x-y|<\d$. Furthermore $[x,y]\in \L$ and
$$[\cdot,\cdot]:\{(x,y)\in \L\times\L:|x-y|<\d\}\to \L \text{ is continuous}.$$
\end{lemma}
\begin{proof}
The proof employs the same idea in proving (ii) of Lemma \ref{L:HolonomyMap} and the existence of such $\d$ follows from the uniform continuity of $W^s_\e(x)$ and $W^u_\e(y)$.
\end{proof}

\begin{lemma}\label{L:Expansive}
There is an $\e>0$ such that for any $x,y\in \L$ with $y\neq x$
$$|f^kx-f^ky|>\e \text{ for some } k\in\mb Z.$$
\end{lemma}
\begin{proof}
Otherwise $y\in W^s_\e(x)\cap W^u_\e(x)$, and also note that $\{x\}=W^s_\e(x)\cap W^u_\e(x)$, so $y=x$.
\end{proof}

We say a sequence $\underline{x}=\{x_i\}^b_{i=a}(a=-\infty \text{ or } b=\infty\text{ is permitted })$ of points in $\L$ is an $\a$-$pseudo$-$orbit$ if
$$|fx_i-x_{i+1}|<\a \text{ for all }i\in[a,b-1).$$
We say a point $x\in \L$ $\b$-$shadows$ $\underline{x}$ if
$$|f^ix-x_i|\le \b\text{ for all }i\in[a,b].$$
\begin{lemma}\label{L:Shadow}
If $f|_\L$ is uniformly hyperbolic, then for every $\b>0$ there is an $\a>0$ so that every $\a$-$pseudo$-$orbit$ $\{x_i\}_{i=a}^b\subset \L$ is $\b$-$shadowed$ by a point $x\in \L$.
\end{lemma}
\begin{proof}
This result mainly follows from  the shadowing theorem in \cite{ChLP}. But the results from \cite{ChLP} does not guarantee $x\in\L$. To avoid this, one can extend the pseudo orbit backward by assigning $x_{a-i}=f^{-i}x_a$ for all $i\ge 1$ (such operation is not needed when $a=-\infty$). Obviously, this extended pseudo orbit is an $\a$-$pseudo$-$orbit$ too. Note that the point, denoted by $x'$, $\b$-shadowing the extended orbit is also  $\b$-shadowing the original one. As long as $\b$ small enough such that $B(\L,\b)\subset U$, one observes that $x'\in \cap_{n\ge 0}f^n(U)=\L$. The proof is complete.
\end{proof}

\begin{lemma}\label{L:DensePeriodicOrb}
If $f|_\L$ is uniformly hyperbolic and non-wandering, then
$$\L=closure\{x \;\big |\ \text{ there exists }n\in\mb N \text{ such that } f^{n}x=x\},$$
where such $x$ is called a periodic point.
\end{lemma}
\begin{proof}
Note that all periodic points near $\L$ are in $\L$ since $\L$ is an attractor.
For any $x\in \L$, we will show that for any $\e>0$ there is a periodic point $y\in \L$ with $|y-x|<\e$.   By Lemma \ref{L:Shadow}, there exists $\a\in(0,\e)$ so that every $\a$-pseudo-orbit is $\frac12 \e$-shadowed by a point in $\L$. Since $f|_{\L}$ is non-wandering, there exists $n\in \mb N$ such that
$$f^{n}\left(B\left(x,\frac12\a\right)\cap \L\right)\cap\left( B\left(x,\frac12\a\right)\right)\cap \L\neq \emptyset.$$
Pick an $x'$ from the above set and define
$$x_{nk+l}=f^{l}x' \text{ for all integers } l\in[0,n-1], k\in \mb N,
$$
which is seen to be an $\a$-pseudo-orbit. Thus,  there is a point $y\in \L$ which is $\frac12\e$-shadowing $\{x_{m}\}_{m\in\mb Z}$. Therefore,
$$|f^{k}(f^{n}y)-f^{k}y|\le |f^{n+k}y-x_{n+k}|+|x_{k}-f^{k}y|\le \e \text{ for all }k\in \mb Z.$$
For any $\e>0$ smaller than the $\e$ in Lemma \ref{L:Expansive}, Lemma \ref{L:Expansive} implies that $f^{n}y=y$ with $|y-x|\le |y-x'|+|x'-x|<\e$.

\end{proof}

The following result is the so-called {\em Spectral Decomposition} of the system $f|_\L$.
\begin{lemma}\label{L:SpectralDecomposition}
If $f|_\L$ is uniformly hyperbolic and non-wandering, then $\L=\L_{1}\cup \L_{2}\cup \cdots\cup \L_{k}$  where $\L_{i}$'s are pairwise disjoint compact sets with the following properties:
\begin{itemize}
\item[(a)] $f(\L_{i})=\L_{i}$ and $f|_{\L_{i}}$ is topological transitive;
\item[(b)] $\L_{i}=\L_{1,i}\cup \L_{2,i}\cup\cdots\cup\L_{n_{i},i}$ with $\L_{j,i}$'s being pairwise disjoint compact sets, $f(\L_{j,i})=\L_{j+1,i}$ (set $\L_{n_{i}+1,i}=\L_{1,i}$) and $f^{n_{i}}|_{\L_{j,i}}$ being topological mixing.
\end{itemize}
\end{lemma}
\begin{proof}
By Lemma \ref{L:DensePeriodicOrb}, the set of periodic points is dense in $\L$. Then,  the proof of Theorem 3.5 in \cite{Bowen} can be adapted under the setting of this lemma, which we refer the reader to \cite{Bowen} for details.

\end{proof}

\subsection{Proof of Theorem \ref{T:SRBFinite}}\label{S:ProofofT7}
In this section, we prove Theorem \ref{T:SRBFinite} by assuming Theorem \ref{T:SRBUnique}. By Lemma \ref{L:SpectralDecomposition}, it is sufficient to prove the following result:
\begin{prop}\label{P:SRBUniqueTransitive}
If $f|_{\L}$ is uniformly hyperbolic, topologically transitive and $(f|_\L)^{-1}$ is Lipchitz continuous, then there is a unique ergodic SRB measure of $f$ with support in $\L$.
\end{prop}
\begin{proof}
Note that $f|_{\L}$ being transitive implies $f|_{\L}$ being non-wandering. Since $f|_{\L}$ is topologically transitive, by Lemma \ref{L:SpectralDecomposition}, there exist pairwisely disjoint compact sets $\{\L_{i}\}_{1\le i\le n}$ such that $f(\L_{i})=\L_{i+1}$ (set $\L_{n+1}=\L_{1}$) with $f^{n}|_{\L_{i}}$ being topological mixing. By Theorem \ref{T:SRBUnique}, there is a unique SRB $\mu_{i}$ of $f^{n}|_{\L_{i}}$ for each $1\le i\le n$, which is mixing.  Under the setting of this section, an equivalent statement of Theorem \ref{T:SRBUnique} is (\ref{E:UniqueEntropyFormula}), for details we refer the reader to \cite{LSh}. Then,  we have that for any $1\le i\le n$, $\mu_{i}$ is the unique measure satisfying
$$h_{\mu_{i}}(f^{n})=\int_{\L_{i}}\sum\l_{j,n}^{+}m_{j}d\mu_{i},$$
where $\l_{j,n}$ are Lyapunov exponents of $f^{n}|_{\L_{i}}$, $m_{j}$ is the multiplicity of $\l_{j,n}$, $\l_{j,n}^{+}=\max\{\l_{i},0\}$ and the summation is over all Lyapunov exponents. It follows from assumption C2) (ii) that this summation is finite on each $x\in \L$. And the uniform hyperbolicity and transitivity of $f|_{\L}$ implies that this summation is uniformly bounded. By the definition of Lyapunov exponents, we have that they are invariant along any trajectory and  $\l_{j,n}=n\l_{j,1}$, where $\l_{j,1}$'s are the corresponding Lyapunov exponents of $f|_{\L}$.

Let   $\mu$ be an SRB measure of $f$ with support in $\L$.  Since $\mu$ is $f$-invariant, $\mu(\L_{i})=\frac1n$ for all $1\le i\le n$. Denote $\mu'_{i}=n\mu|_{\L_{i}}$ which is the normalized measure of $\mu$ restricted on $\L_{i}$.
  By Theorem \ref{T:EntropyFormula} and the $f$-invariance of Lyapunov exponents, we have that
\begin{align*}
h_{\mu}(f)&=\int_{\L}\sum\l_{j,1}^{+}m_{j}d\mu
=n\int_{\L_{i}}\sum\l_{j,1}^{+}m_{j}d\mu=\int_{\L_{i}}\sum\l_{j,n}^{+}m_{j}d\mu=\frac1nh_{\mu'_{i}}(f^{n}),\ 1\le i\le n.
\end{align*}
Thus,  for all $1\le i\le n$, $$h_{\mu'_{i}}(f^{n})=\int_{\L_{i}}\sum\l_{j,n}^{+}m_{j}d(n\mu)=\int_{\L_{i}}\sum\l_{j,n}^{+}m_{j}d\mu'_{i}.$$
Hence, by (\ref{E:UniqueEntropyFormula}), $\mu|_{\L_{i}}=\frac 1n \mu'_{i}=\frac1n \mu_{i}$ for each $1\le i\le n$, which implies that $\mu=\frac1n\sum_{i=1}^n\mu_i$ is unique. Since each $\mu_{i}$ is $f^{n}$ mixing, $\mu$ is ergodic. The proof is complete.
\end{proof}

\subsection{Proof of Theorem \ref{T:SRBUnique}}\label{S:ProofofT5}
The proof of uniqueness of SRB measure follows from the uniqueness of equilibrium state for a given H\"older continuous function which will be precisely stated later in Proposition  \ref{P:GibState0}. The uniqueness of equilibrium state is established  by using  a Markov partition of the system and the H\"older continuity of the invariant splitting.  The construction of a Markov partition given by Bowen for Axiom A systems is also valid for hyperbolic attractors in a Hilbert space. The proof is almost identical, which will be presented in Appendix  \ref{S:Markov}.
\begin{prop}\label{P:GibState0}
Assume that $f|_{\L}$ is topologically mixing. For every H\"{o}lder continuous function  $\phi:\L\to \mb R$, there exists a unique equilibrium state $\mu_{\phi}\in \mc M_f(\L)$ with respect to $f$.
\end{prop}

Now we set $\phi(x)=-\log|\det(Df|_{E^u_x})|$ for all $x\in \L$. By the multiplicative ergodic theorem and Birkhoff Ergodic Theorem, we have that, for any $\mu\in \mc M_f(\L)$
\begin{align}\begin{split}\label{E:PositiveLyaExp}
&\sum\l_i^+(x)m_i(x)=\lim_{n\to\infty}\frac1n\log\prod_{i=0}^{n-1}|\det(Df|_{E^u_{f^ix}})|\\
&=\lim_{n\to\infty}\frac1n\sum_{i=0}^{n-1}\left(-\phi(f^ix)\right)=-\overline{\phi}(x),\ \text{ for }\mu-a.e.\ x\in \L,
\end{split}
\end{align}
where $\l_i$'s are Lyapunov exponents, $m_i$ is the multiplicity of $\l_i$, $\l_i^+=\max\{\l_i,0\}$, the summation is over all the Lyapunov exponents, and $\overline{\phi}$ is an $f$-invariant Borel function satisfying $\int_\L \overline \phi d\mu =\int_\L  \phi d\mu$. By Theorem \ref{T:RuelleIne}, we have that
$$h_\mu(f)\le \int_\L \sum\l_i^+m_i d\mu=\int_\L(-\overline\phi)d\mu=-\int_\L\phi d\mu\  i.e.\  h_\mu(f)+\int_\L\phi d\mu\le 0.$$
Actually, "$=$" holds in above inequality when $\mu$ is an SRB measure (due to Theorem \ref{T:EntropyFormula}). By Theorem \ref{T:SRBExist} proved in Section \ref{S:ExistSRB}, a SRB measure exists under the current setting. Then,  by applying Theorem \ref{T:VarPrin2}, we have that
$$P_{f}(\phi)=\sup_{\mu\in \mc M_f(\L)}\left(h_\mu(f)+\int_\L\phi d\mu\right)=0.$$
Since the splitting $\mb H=E^{u}\oplus E^{s}$ is H\"older continuous on $\L$ by Proposition \ref{P:HolderSplit} given in the next subsection,
and $f$ is $C^{2}$, we have that $\phi:\L\to \mb R$ is H\"{o}lder continuous. Then,  Proposition \ref{P:GibState0} implies that
\begin{align}\begin{split}\label{E:UniqueEntropyFormula}
&\text{There exist a unique }\mu\in \mc M_f(\L)\text{ such that }\\
&h_\mu(f)+\int_\L\phi d\mu=P_f(\phi)=0.
\end{split}\end{align}
Let $\nu$ be a SRB measure supported on $\Lambda$. Then, by  Theorem \ref{T:EntropyFormula}, one has
\[
h_\nu(f)=\int_\L \sum\l_i^+(x)m_i(x)d\nu,
\] which together with (\ref{E:PositiveLyaExp}) yield
\[
h_\nu(f)+\int_\L\phi d\nu=0.
\] Using (\ref{E:UniqueEntropyFormula}), we have $\mu=\nu$. This completes the
  proof of Theorem \ref{T:SRBUnique}. $\hfill\square$

\medskip

\subsection{Continuity of Invariant Distributions.}\label{S:ContinuousSplit}
In this subsection, we investigate the continuity of invariant splittings.  In the first part, the results proved is in a more general manner and the setting is independent of the main setting; In the second part, we go back to the main setting of this paper.
\subsubsection{General Results}\label{S:GeneralResults}
Let $f: X\to  X$ be a $C^{1+\a}$ map on a separable Banach space $X$ and $\L$ be a $f$-invariant compact set on which $f$ is homeomorphism. Assume that $Df_x$ is injective for all $x\in \L$.

For given numbers $l_0>0$, $\l_2<\l_1$ and $0<\d<<\l_1-\l_2$. Let $\G^u$ be the collection of points $x\in\L$ satisfying that
\begin{itemize}
\item[(U1)]  $X=E^u(x)\oplus E^s(x)$;
\item[(U2)] There exist splittings $X=E^u_n(x)\oplus E^s_n(x)$ with $E^u_0(x)=E^u(x), E^s_0(x)=E^s(x)$,\\
 and $Df_{f^{-(n+1)}x}E^u_{n+1}(x)=E^u_n(x)$, $Df_{f^{-(n+1)}x}E^s_{n+1}(x)\subset E^s_n(x)$ for all $n\ge 0$;
\item[(U3)] $\left|Df^n_{f^{-n}x}v\right|\ge l_0^{-1}e^{n\l_1}|v|$ for all $v\in E^u_n(x)$ and $n\ge 0$;
\item[(U4)] $\left\|Df^n_{f^{-n}x}|_{E^{s}_n(x)}\right\|\le l_0e^{n\l_2}$ for all $n\ge 0$;
\item[(U5)] $\max\{\|\pi^u_n(x)\|,\|\pi^{s}_n(x)\|\}\le l_0e^{n\d}$ for all $n\ge 0$, where $\pi^u_n(x)$ and $\pi^s_n(x)$ are the projections associating to splitting $X=E^u_n(x)\oplus E^s_n(x)$.
\end{itemize}

\begin{prop}\label{P:E^uCont}
$E^u$ is uniformly continuous on $\G^u$. If further assume that $(f|_\L)^{-1}$ is Lipchitz continuous, then $E^u$ is H\"older continuous on $\G^u$ with exponent
$$\b=\frac{\l_1-\l_2-\d}{\log a-\l_1}\a,$$
where $a>sup_{x\in \L}\{\|Df_x\|\}\left(\max\left\{Lip\left((f|_\L)^{-1}\right),1\right\}\right)^\a$.
\end{prop}

\begin{proof}
In what follows, the tangent spaces
are identified with $X$, so it makes sense to write
$u-v$ where $u$ is in the tangent space at $x$ and $v$ is in the tangent space at $y$ with
$x \ne y$.

To prove the uniform continuity of $E^u$, we  consider
$x,y\in\G^u$ and a unit vector $v\in E^u(y)$,  estimate
$|\pi^{s}_xv|$ by iterating {\it backwards}, obtaining
\begin{align*}
 \begin{split}
&|\pi^{s}_0(x) v| =|\pi^{s}_0(x)Df^n_{f^{-n}(y)}Df^{-n}_yv|\\
& \le |\pi^{s}_0(x)Df^n_{f^{-n}(x)}Df^{-n}_yv| +
|\pi^{s}_0(x)(Df^n_{f^{-n}(y)}-Df^n_{f^{-n}
(x)})Df^{-n}_yv|\\
& = |Df^n_{f^{-n}(x)}\pi^{s}_n(x)Df^{-n}_yv| +
|\pi^{s}_0(x)(Df^n_{f^{-n}(y)}
-Df^n_{f^{-n}
(x)})Df^{-n}_yv|\\
&\le l_0^3e^{n(-\l_1+\l_2+\d)}+l_0^2e^{-n\l_1}\|Df^n_{f^{-n}(y)}
-Df^n_{f^{-n}
(x)}\|.
 \end{split}
\end{align*}
Given $\e>0$, there exists $n$ such that
$$l_0^3e^{n(-\l_1+\l_2+\d)}\le\frac12\e<l_0^3e^{(n-1)(-\l_1+\l_2+\d)}.$$
For such $n$, since $x\mapsto f^{-n}(x)$ is continuous on $\L$ and
$f$ is $C^{1+\a}$, there exists $\D>0$ such that if $|x-y|\le
\D$, then $l_0^2e^{n(-\l_1)}\|Df^{n}_{f^{-n}(x)}-Df^{n}_{f^{-n}(y)}\|\le
\frac12\e$. Note that $\D$ depends on the system constants and $\e$ only, thus $E^u$ is uniformly continuous.

In the second part, we assume that $(f|_\L)^{-1}$ is Lipchitz continuous and show that $E^u$ is H\"older continuous on $\G^u$. We begin with the following two lemmas:

\begin{lemma}\label{L:E^nGap}
Let $\{A_n\}_{n\ge 0}$ and $\{B_n\}_{n\ge 0}$ be two sequences of bounded linear operators in $\mc L( X)$ such that for some $\D>0$ and $c>0$,
$$\|A_n-B_n\| \le \D c^n, \text{ for all }n\ge 0.$$
Assume that there exist sequences of subspaces $ \{E_{A_n}^u\}_{n\ge 0}$, $\{E_{A_n}^s\}_{n\ge 0}$, $\{E_{B_n}^u\}_{n\ge 0}$, $\{E_{B_n}^s\}_{n\ge 0}$ and numbers $\mu_2<\mu_1$, $\d<<\mu_1-\mu_2$ and $C>1$ such that for all $n\ge 0$,
\begin{align*}
& X= E_{A_n}^u\oplus E_{A_n}^s \text{ with } \max\{ \|\pi_{A_n}^u\|,\|\pi_{A_n}^s\|\}\le Ce^{n\d};\\
& X= E_{B_n}^u\oplus E_{B_n}^s \text{ with } \max\{ \|\pi_{B_n}^u\|,\|\pi_{B_n}^s\|\}\le Ce^{n\d};\\
&A_n(E_{A_n}^u)=E_{A_0}^u; \quad A_n(E_{A_n}^s)\subset E_{A_0}^s;\\
&B_n(E_{B_n}^u)=E_{B_0}^u;\quad B_n(E_{B_n}^s)\subset E_{B_0}^s;\\
&|A_nv|\le C e^{n\mu_2}|v| \text{ for all }v\in E_{A_n}^s;\quad |A_nv|\ge C^{-1} e^{n\mu_1}|v| \text{ for all }v\in E_{A_n}^u;\\
&|B_nv|\le C e^{n\mu_2}|v| \text{ for all }v\in E_{B_n}^s;\quad |B_nv|\ge C^{-1} e^{n\mu_1}|v| \text{ for all }v\in E_{B_n}^u.
\end{align*}
Then,  for $\D<\frac12C^{-1}$ and $c>e^{\mu_1}$,
$$d(E_{A_0}^u,E_{B_0}^u)< 4C^3e^{\mu_1-\mu_2-\d}(2C)^{\frac{\mu_1-\mu_2-\d}{\log c-\mu_1}}\D^{\frac{\mu_1-\mu_2-\d}{\log c-\mu_1}}.$$
\end{lemma}
\begin{proof}
Set
$$R_{A_n}=\left\{v\in X\big|\ |A_nv|\ge \frac1{2C} e^{n\mu_1}|v|\right\},$$
and
$$R_{B_n}=\left\{v\in X\big|\ |B_nv|\ge \frac1{2C} e^{n\mu_1}|v|\right\}.$$
For $v\in  X$, set $v_1=\pi_{A_n}^uv\in E_{A_n}^u$ and  $v_2=\pi_{A_n}^sv\in E_{A_n}^s$, then
$$\max\{|v_1|,|v_2|\}\le Ce^{n\d}|v|,$$
and
$$|A_nv_2|\le Ce^{n\mu_2}|v_2|\le C^2e^{n(\mu_2+\d)}|v|.$$
If  $v\in R_{A_n}$, then
\begin{equation}\label{E:GapA_nv}
\di(A_nv,E_{A_0}^u)\le |A_nv_2|\le C^2e^{n(\mu_2+\d)}|v|\le 2C^3e^{n(\mu_2+\d-\mu_1)}|A_nv|.
\end{equation}
For a given $\D\in (0,\frac12C^{-1})$ and $c>e^{\mu_1}$, there exists $n\ge 0$ such that
$$\frac12C^{-1}\frac{e^{(n+1)\mu_1}}{c^{n+1}}\le \D<\frac12C^{-1}\frac{e^{n\mu_1}}{c^n}.$$
For such $n$ and each $w\in E_{B_n}^u$, we have that
\begin{align*}
|A_nw|&\ge |B_nw|-\|A_n-B_n\||w|\ge C^{-1}e^{n\mu_1}|w|-\D c^n|w|\\
&\ge C^{-1}e^{n\mu_1}|w|-\frac12 C^{-1}e^{n\mu_1}|w|\ge  \frac12 C^{-1}e^{n\mu_1}|w|.
\end{align*}
So $w\in R_{A_n}$ and therefore $E_{B_n}\subset R_{A_n}$. Alternatively, for such $n$, $E_{A_n}\subset R_{B_n}$. Then,  (\ref{E:GapA_nv}) and the choice of $n$ implies that
\begin{align*}
d(E_{A_0}^u,E_{B_0}^u)&\le 4C^3e^{n(\mu_2+\d-\mu_1)}=4C^3e^{\mu_1-\mu_2-\d}\left(e^{(n+1)(\mu_1-\log c)}\right)^{\frac{\mu_1-\mu_2-\d}{\log c-\mu_1}}\\
&\le 4C^3e^{\mu_1-\mu_2-\d}(2C\D)^{\frac{\mu_1-\mu_2-\d}{\log c-\mu_1}}.
\end{align*}
\end{proof}

\begin{lemma}\label{L:Df^nBackHolder}
Let $f:X\to  X$ be a $C^{1+\a}$ map on a  Banach space and $\L$ be a $f$-invariant compact set on which $f$ is homeomorphism. Assume that $(f|_\L)^{-1}$ is Lipchitz continuous. Then,  for every $a>bL^\a$, there exists $M$ such that for every $n\in \mb N$ and $x,y\in \L$ , we have
$$ \|Df^n_{f^{-n}x}-Df^n_{f^{-n}y}\|\le M a^n|x-y|^\a,$$
where $b=\sup_{x\in \L}\{\|Df_x\|\}$ and $L=\max\left\{Lip\left((f|_\L)^{-1}\right),1\right\}$.
\end{lemma}
\begin{proof}
Let $L'$ be such that
$$\|Df_x-Df_y\|\le L'|x-y|^\a,\ \forall x,y\in \L.$$
Note that for any $x,y\in \L$
\begin{equation}\label{E:HolderN=1}
\|Df_{f^{-1}x}-Df_{f^{-1}y}\|\le L'|f^{-1}x-f^{-1}y|^\a\le L'L^\a|x-y|^\a.
\end{equation}
Inductively, for $a>bL^\a$, by taking $M>\frac{\frac {L'}b}{1-\frac ba}$, we have that
\begin{align}\begin{split}\label{E:HolderN}
&\|Df^{(n+1)}_{f^{-(n+1)}x}-Df^{(n+1)}_{f^{-(n+1)}y}\|\\
\le &\|Df_{f^{-(n+1)}x}\| \|Df^n_{f^{-n}x}-Df^n_{f^{-n}y}\|+\|Df_{f^{-(n+1)}x}-Df_{f^{-(n+1)}y}\|  \|Df^n_{f^{-n}y}\|\\
\le &b \|Df^n_{f^{-n}x}-Df^n_{f^{-n}y}\|+b^n\|Df_{f^{-(n+1)}x}-Df_{f^{-(n+1)}y}\|\\
\le &b  Ma^n|x-y|^\a+L'b^nL^{(n+1)\a}|x-y|^\a=Ma^{n+1}|x-y|^\a\left(\frac ba+\frac{L'}{Mb}\left(\frac{b L^\a}{a}\right)^{n+1}\right)\\
\le& Ma^{n+1}|x-y|^\a.
\end{split}
\end{align}
\end{proof}

Now we are ready to prove Proposition \ref{P:E^uCont}. For $x,y\in \G^u$, set $C=l_0$, $c=a$, $\mu_1=\l_1$, $\mu_2=\l_2$, $A_n=Df^n_{f^{-n}x}$, $B_n=Df^n_{f^{-n}y}$, $E^u(x)=E_{A_0}^u$, $E^u(y)=E_{B_0}^u$ and $\D=M|x-y|^\a$. Note that $\l_1\le b<a$. Then,  by applying Lemma \ref{L:E^nGap}, we have that if $|x-y|<\left(\frac12l_0^{-1}M^{-1}\right)^{\frac1\a}$ then
$$d(E^u(x),E^u(y))<8l_0^3e^{\l_1-\l_2-\d}(2l_0M)^{\frac{\l_1-\l_2-\d}{\log a-\l_1}}|x-y|^{\frac{\l_1-\l_2-\d}{\log a-\l_1}\a}.$$
Since $\left(\frac12l_0^{-1}M^{-1}\right)^{\frac1\a}$ does not depend on $x$ or $y$ and $d(E^u(x),E^u(y))\le 2$, we have that $ E^u$ is H\"older continuous on $\G^u$ with exponent $\b=\frac{\l_1-\l_2-\d}{\log a-\l_1}\a$.
\end{proof}

In the following, we investigate the continuity of $E^{s}$.
For given numbers $l_0>0$ and $\l_2<\l_1$. Let $\G^s$ be the collection of points $x\in\L$ satisfying that
\begin{itemize}
\item[(S1)] $X=E^u(x)\oplus E^s(x)$;
\item[(S2)] $\left|Df^n_{x}v\right|\ge l_0^{-1}e^{n\l_1}|v|$ for all $v\in E^u(x)$ and $n\ge 0$;
\item[(S3)] $\left\|Df^n_{x}|_{E^{s}(x)}\right\|\le l_0e^{n\l_2}$ for all $n\ge 0$;
\item[(S4)] $\max\{\|\pi^u(x)\|,\|\pi^{s}(x)\|\}\le l_0$.
\end{itemize}

\begin{prop}\label{P:E^sHolderCont}
 $E^s$ is H\"{o}lder continuous on $\G^s$ with exponent
$$\b=\frac{\l_1-\l_2}{\log a-\l_2}\a,$$
 where $a>\sup_{x\in\L}\{\|Df_x\|\}\max\left\{1,(Lip(f|_\L))^\a\right\}.$
\end{prop}
\begin{proof}
Let $L'$ be such that for all $x,y\in\L$
$$\|Df_x-Df_y\|\le L'|x-y|^\a.$$
Set $b=\sup_{x\in\L}\{\|Df_x\|\}$ and $L=max\{1,Lip(f|_\L)\}$. Given $a>bL^\a$, taking $M>\frac{\frac {L'}b}{1-\frac ba}$, by using the same argument as in (\ref{E:HolderN}), we have that for any $x,y\in\L$ and $n\ge 0$
\begin{equation}\label{E:SHolderN}
\|Df^{n}_x-Df^{n}_y\|\le Ma^{n}|x-y|^\a.
\end{equation}

For $x\in \G^s$, set
$$R^n_x=\{ v\in \mb X|\ |Df^n_xv|\le 2l_0e^{n\l_2}|v|\}.$$
For $v\in  X$, set $v_1=\pi^u(x)v$ and $v_2=\pi^s(x)v$, then
$$\max\{|v_1|,|v_2|\}\le l_0|v|.$$
 If $v\in R^n_x$, then
$$|Df^n_xv|\ge |Df^n_x v_1|-|Df^n_xv_2|\ge l_0^{-1}e^{n\l_1}|v_1|-l_0e^{n\l_2}|v_2|,$$
and thus
$$|v_1|\le l_0e^{-n\l_1}\left(|Df^n_x v|+l_0e^{n\l_2}|v_2|\right)\le (2+l_0)l_0^2e^{-n(\l_1-\l_2)}|v|.$$
Therefore,
\begin{equation}\label{E:E^sGap}
\di(v,E^s(x))\le (2+l_0)l_0^2e^{-n(\l_1-\l_2)}|v|.
\end{equation}
Note that $e^{\l_2}<e^{\l_1}\le b<a$. For any $x,y\in \G^s$ with $|x-y|<M^{-\frac1\a}$, there exists a unique integer $n\ge 0$ such that
$$e^{-(n+1)(\log a-\l_2)}\le M|x-y|^\a<e^{-n(\log a-\l_2)}.$$
If $w\in E^s(y)$, then
\begin{align*}
|Df^n_xw|\le& |Df^n_yw|+\|Df^n_x-Df^n_y\||w|\le l_0e^{n\l_2}|w|+a^nM|x-y|^\a|w|\\
\le& (l_0e^{n\l_2}+e^{n\l_2})|w|\le 2l_0e^{n\l_2}|w|,
\end{align*}
which implies that $w\in R^n_x$, and thus $E^s(y)\in R^n_x$. The same argument implies that $E^s(x)\in R^n_y$. By (\ref{E:E^sGap}) and the choice of $n$, we have
\begin{align*}
d(E^s(x),E^s(y))\le& 2(2+l_0)l_0^2e^{-n(\l_1-\l_2)}
\le 2(2+l_0)l_0^2e^{\l_1-\l_2}M^{\frac{\l_1-\l_2}{\log a-\l_2}} |x-y|^{\frac{\l_1-\l_2}{\log a-\l_2}\a}.
\end{align*}
Since $M^{-\frac1\a}$ does not depend on $x$ or $y$ and $d(E^s(x),E^s(y))\le 2$, we have the  $E^s$ is H\"older continuous on $\G^u$ with exponent $\b=\frac{\l_1-\l_2}{\log a-\l_2}\a$.

\end{proof}
In the proof of Proposition \ref{P:E^sHolderCont}, we also used   the ideas from \cite{Brin} and  Section 5.3 of \cite{BP}.

\subsubsection{Under the Main Setting}\label{S:UnderMainSetting}
In the following, we go back to the main setting of the paper: $f$ and $\L$ satisfy Conditions C1)-C3).
For given positive numbers $\d,\d_0,\l_0,l_0$ and $m\in\mb N$ with $\d,\d_0<<\l_0$. Let $\G^u_{\d,\d_0,\l_0,l_0,m}$ be the collection of points $x\in\L$ satisfying that
\begin{itemize}
\item[(Ua)] $dim E^u(x)=m$;
\item[(Ub)] $\left\|\left(D^nf_{f^kx}|_{E^u(f^kx)}\right)^{-1}\right\|\le l_0e^{-n\l_0+|k|\d}$ for all $n\ge 0$ and $k\in\mb Z$;
\item[(Uc)] $\left\|D^nf_{f^kx}|_{E^{cs}(f^kx)}\right\|\le l_0e^{n\d_0+|k|\d}$ for all $n\ge 0$ and $k\in\mb Z$;
\item[(Ud)] $\max\{\|\pi^u(f^k(x))\|,\|\pi^{cs}(f^k(x))\|\}\le l_0e^{|k|\d}$ for all $k\in\mb Z$.
\end{itemize}
Similarly, we define $\G^s_{\d,\d_0,\l_0,l_0,m}$ to be the collection of points $x\in\L$ satisfying that
\begin{itemize}
\item[(Sa)] $dim E^{cu}(x)=m$;
\item[(Sb)] $\left\|\left(D^nf_{f^kx}|_{E^{cu}(f^kx)}\right)^{-1}\right\|\le l_0e^{n\d_0+|k|\d}$ for all $n\ge 0$ and $k\in\mb Z$;
\item[(Sc)] $\left\|D^nf_{f^kx}|_{E^s(f^kx)}\right\|\le l_0e^{-n\l_0+|k|\d}$ for all $n\ge 0$ and $k\in\mb Z$;
\item[(Sd)] $\max\{\|\pi^{cu}(f^k(x))\|,\|\pi^{s}(f^k(x))\|\}\le l_0e^{|k|\d}$ for all $k\in\mb Z$.
\end{itemize}

In the following, Propositions \ref{P:G^uConti}, \ref{P:G^sConti} and \ref{P:HolderSplit} are derived from Propositions \ref{P:E^uCont} and \ref{P:E^sHolderCont} straightforwardly, so the proofs are omited.
\begin{prop}\label{P:G^uConti}
On $\G^u_{\d,\d_0,\l_0,l_0,m}$, $E^{u}$ and $E^{cs}$ vary uniformly continuously, and so do the projections $\pi^{u}$ and $\pi^{cs}$. Moreover, $E^{cs}$ is H\"older continuous on $\G^u_{\d,\d_0,\l_0,l_0,m}$.
\end{prop}
\begin{prop}\label{P:G^sConti}
On $\G^s_{\d,\d_0,\l_0,l_0,m}$, $E^{cu}$ and $E^{s}$ vary uniformly continuously, and so do the projections $\pi^{cu}$ and $\pi^{s}$. Moreover, $E^s$ is H\"older continuous on $\G^s_{\d,\d_0,\l_0,l_0,m}$.
\end{prop}

\begin{prop}\label{P:G^cConti}
On $\G^u_{\d,\d_0,\l_0,l_0,m}\cap\G^s_{\d,\d_0,\l_0,l_0,m}$, $E^{c}$ vary uniformly continuously.
\end{prop}
\begin{proof}
Since $E^c(\cdot)=E^{uc}(\cdot)\cap E^{cs}(\cdot)$,
we have, for a unit vector $v\in E^c(y)$,
$$|v-\pi^{us}_xv|\le |\pi^s_xv|+|\pi^u_xv|,
$$
which tends to $0$ uniformly as $|y-x|\to 0$ uniformly by the uniform continuity of
$E^u$, $E^s$, $E^{uc}$ and $E^{cs}$.

\end{proof}

\begin{prop}\label{P:HolderSplit}
Let $f|_\L$ be uniformly hyperbolic, then $E^s(\cdot)$ is H\"{o}lder continuous on $\L$. If additionally assume that $(f|_\L)^{-1}$ is Lipchitz continuous, then $E^u(\cdot)$ is H\"oder continuous on $\L$.
\end{prop}

\appendix
\setcounter{section}{0}
\renewcommand{\thesection}{\Alph{section}}


\section{Lyapunov Charts}\label{S:LyapunovCharts}
Lyapunov exponents are, by definition, asymptotic quantities. It simplifies the proofs
greatly if one works in coordinates in which these values are reflected
in a single iteration. In finite dimensions, Lyapunov metrics were
introduced to do so. These metrics were first used in \cite{P} and later
 in e.g. \cite{K}, \cite{LY}; see also the exposition in \cite{Y}. Such tools in Hilbert spaces was constructed in \cite{LY}, which is used in this paper. In \cite{LY}, the invariant measure is assumed ergodic, while the non-ergodic case can be derived by the ergodic decomposition in a straightforward way. For the sake of convenience, the results in the following will be stated for the non-ergodic case. Formally, all the constants in \cite{LY} will be replaced by $f$-invariant measurable functions.

\medskip

Let $\d_0:\G\to \mb R^+$ be an $f$-invariant measurable function such that $0< \d_0(x) < \frac{1}{100} \l_0(x)$, which denotes an accepted margin of error for the
Lyapunov exponents and will be fixed throughout. Let $\d:\G\to \mb R^+$ be another $f$-invariant measurable function, which is the measure of the nonlinearity in charts and variation of
chart sizes along orbits. We will require the value of these functions small enough depending on the purpose at hand, and
will specify conditions on $\d_0,\d$ each time when a chart system is used.

Let $\l(x)=\l_0(x)-2\d_0(x)$.  From \cite{LY}, one can define an inner product point-wisely.
Recall that the (original) inner product and norm on $\mb H$ are denoted by
$<\cdot,\cdot>$ and $|\cdot|$.

\begin{lemma}\label{L:InnerProduct}
For $\mu$-a.e. $x$, there is an inner product
$<\cdot,\cdot>_x'$ on ${\mb H}_x$
with induced norm $|\cdot|_x'$ such that
\begin{itemize}
\item[(i)] $|Df_xu|'_{fx}\ge e^{\l(x)}|u|'_x$ for all $u\in E^u(x)$;
\item[(ii)] $e^{-2\d_0(x)}|u|_x'\le |Df_xu|_{fx}'\le e^{2\d_0(x)}|u|'_x$ for all $u\in
E^c(x)$;
\item[(iii)] $|Df_xu|_{fx}'\le e^{-\l(x)}|u|_x'$ for all $u\in E^s(x)$;
\item[(iv)] Identifying ${\mb H}_x$ with $\mb H$, the function
 $x \mapsto <u,v>'_x$ is Borel for any fixed $u,v\in \mb H$;
\item[(v)] For all $p \in {\mb H}_x$,
$$\frac{\sqrt 3}{3}|p|\le |p|'_x\le K(x)|p|$$
for some Borel function $K$ with
$$\lim_{n\to\pm\infty}\frac1n \log K(f^nx)=0.$$
\end{itemize}
\end{lemma}

With the point-wisely defined inner product, one can define a family of point-wise coordinates induced by maps
$\Phi_x$ for $\mu$-almost every $x$, where $\Phi_x$ is an affine map taking a
neighborhood of $0$ in $\mb H$ to a neighborhood of $x$ in $\mb H$:\\

Noting that, there exist a countable $f$-invariant measurable partition $\{\G_i\}_{i=1,2,\ldots}$ of $\G$, on each $\G_i$ the dimensions of $E^u$ and $E^c$ and the codimension of
$E^s$ are constant. For $x\in\G_i$, we fix orthogonal subspaces
$\tilde E_i^u, \tilde E_i^c$ and $\tilde E_i^s$ of $\mb H$  such that
$\dim\tilde E_i^u = \dim E^u(x), \ \dim \tilde E_i^c(x) = \dim E^c(x)$
and codim $\tilde E_i^s = $ codim $E^s(x)$, and define
$L_x : \mb H_x \to \mb H$ being such that

\medskip
(i) $L_x(E^\tau(x))= \tilde E_i^\tau, \ \tau=u,c,s$; and

(ii) $<L_x u, L_xv> \ = \ <u,v>_x'$ for all $u, v \in \mb H_x$.

\medskip
\noindent
Such a linear map exists and can be chosen to vary
measurably in $x$ (see e.g. \cite{C}).\\

For $r>0$, let $\tilde
B_i(0,r) =
\tilde B_i^u(0,r) \times \tilde B_i^c(0,r) \times \tilde B_i^s(0,r)$ where
$\tilde B_i^\tau(0,r)$ is the ball of radius $r$ centered at $0$ in
$\tilde E_i^\tau$. The coordinate patches $\{\Phi_x\}$ are then
given by
$$
\Phi_x : \tilde B_i(0, \d(x) l(x)^{-1}) \to \mb H \ , \  \Phi_x(u) =
{\rm Exp}_x(L_x^{-1}(u)),\ \text{ for } x\in\G_i,
$$
where ${\rm Exp}_x : {\mb H}_x \to \mb H$ is the exponential map
(the usual identification of the tangent space ${\mb H}_x$ at $x$
with $\{x\} + \mb H$), $\d(x)$ is the $f$-invariant measurable function at the beginning of this
subsection, and $l$ is a function to be determined.
Maps connecting charts along orbits are denoted by
$$
\tilde f_x = \Phi_{fx}^{-1} \circ f \circ \Phi_x \ .
$$
Since $\Phi_{fx}^{-1}$ is extendible to an affine map on all of $\mb H$,
we sometimes view $\tilde f_x$ as $\tilde f_x: \tilde B_i(0, \d l(x)^{-1}) \to \mb H$ for $x\in\G_i$.

The properties of $\Phi_x$ and $\tilde f_x$ are
summarized in Theorem \ref{T:LyapunovChart} below, which follows from Proposition 4 in \cite{LY}. We inherit the notations of \cite{LY} here:
$D(\tilde f_x)_0$ means the derivative of $\tilde f_x$ evaluated at
the point $0$ in the chart, and so on.
To control the nonlinearity
in  charts, we will need the following bound which follows from
the compactness of $\L$ and $C^{2}$-ness of $f$.

\begin{lemma}\label{L:C2Bound}
There exist $M_2>0$ and $r_0>0$ such that $\|D^2f_x\|<M_2$
for all $x \in \mb H$ with dist$(x, A)<r_0$.
\end{lemma}

\begin{theorem}\label{T:LyapunovChart}
{\it Given a measurable $f$-invariant function $\d:\G\to\mb R^+$, there is a measurable function $l:\G\to [1,+\infty)$ with
$e^{-\d(x)}l(x)\le l(f(x))\le e^{\d(x)} l(x)$ such that the following hold at
$\mu$-a.e. $x$:
\begin{itemize}
\item[(a)] For all $y,y'\in B(0,\d(x) l(x)^{-1})$,
$$
l(x)^{-1} |y-y'|\le |\Phi_x(y)-\Phi_x(y')| \le \sqrt 3|y-y'|;
$$
\item[(b)] For $x\in\G_i$, $D(\tilde f_x)_0$ maps each $\tilde E_i^\tau, \ \tau=u,c,s$, into itself,
with
$$|D(\tilde f_x)_0u|\ge e^{\l(x)}|u|, \quad e^{-2\d_0(x)}|w|\le|D(\tilde f_x)_0w|\le e^{2\d_0(x)}|w|
$$
$${\rm and} \quad |D(\tilde f_x)_0v|\le e^{-\l(x)}|v|$$
for $u\in \tilde E_i^u$;
$w\in \tilde E_i^c$ and $v\in \tilde E_i^s$.
\item[(c)] The following hold on $B(0,\d(x) l(x)^{-1})$:

\noindent (i) Lip$(\tilde f_x-D(\tilde f_x)_0)<\d(x)$;

\noindent (ii) Lip$(D\tilde f_x)\le l(x)$.
\end{itemize}}
\end{theorem}


\section{Symbolic Dynamics and Thermodynamic Formalism}\label{S:AppSymDyn}
In this section, we collect some results of symbolic dynamics and thermodynamic formalism which are used in Appendix \ref{S:SymDyn} and \ref{SS:Equilibrium}.  These results are borrowed from \cite{Bowen} without proof and symbols used in \cite{Bowen} are also inherited. \\

Let $\{1,\ldots,n\}$ be the set of possible state of a physical system. A configuration is a point
$$\underline{x}=\{x_i\}_{i=-\infty}^\infty\in \prod_{\mb Z}\{1,\ldots,n\}=\Sigma_n.$$
Give the discrete topology on $\{1,\ldots,n\}$ and the product topology on $\Sigma_n$. Then,  $\Sigma_n$ is a compact metrizable space and the metric can be defined in the following way:
$$d_\b(\underline{x},\underline y)=\sum_{j\in\mb Z}\b^{|j|}(1-\chi_{\{0\}}(|x_j-y_j|)),$$
where $\b\in(0,1)$ and $\chi_{\{0\}}$ is the characteristic function of $\{0\}\subset \mb R$.
 For any continuous function $\phi:\Sigma_n\to \mb R$, define
$$var_k\phi=\sup\left\{|\phi(\underline{x})-\phi(\underline{y})\big|\ x_i=y_i\text{ for all }{i}\le k\right\}.$$
\begin{theorem}\label{T:AppGibMea}
Let $\phi:\Sigma_n\to \mb R$ satisfy that there exist $c>0$ and $\a\in(0,1)$ such that $var_k\phi\le c\a^k$ for all $k\in\mb N$. Then,  there is a unique $\mu\in \mc M_\sigma(\Sigma_n)$, for which there are constants $c_2\ge c_2>0$ and $P$ such that
$$c_1\le \frac{\mu\{\underline y|\ y_i=x_i \text{ for all }i\in[0,m]\}}{\exp{\left(-Pm+\sum_{k=0}^{m-1}\phi(\sigma^k\underline{x})\right)}}\le c_2$$
for all $\underline x\in \Sigma_n$ and $m\ge 0$.
\end{theorem}
Here $\mc M_\sigma(\Sigma_n)$ is the collection of all $\s$-invariant probability measure on $\Sigma_n$, which is a compact convex metrizable space since $\Sigma_n$ is compact metrizable space. The measure $\mu$ in Theorem \ref{T:AppGibMea} is also written $\mu_\phi$ and called {\bf Gibbs measure} of $\phi$.

For an $n\times n$ matrix of $0$'s and $1$'s, A, set
$$\Sigma_A=\{\underline x\in \Sigma_n|\ A_{x_ix_{i+1}}=1\text{ for all }i\in\mb Z\}.$$
It is easy to see that $\Sigma_A$ is closed and $\s\Sigma_A=\Sigma_A$.
\begin{definition}\label{D:FA}
Let $\mc F_A$ be the collection of all continuous functions $\phi:\Sigma_A\to \mb R$ satisfying $var_k\phi\le b\a^k\text{ for all } k\ge 0$ for some  $b>0$ and $\a\in(0,1)$.
\end{definition}
\begin{theorem}\label{T:AppGibMea1}%
Suppose $(\Sigma_A,\s)$ is topologically mixing and $\phi \in \mc F_A$. Then,  there is a unique $\s$-invariant mxing Borel probability measure $\mu$ on $\Sigma_A$ satisfying that there exist $c_2\ge c_1>0$ and $P$ such that
$$c_1\le \frac{\mu\{\underline y|\ y_i=x_i \text{ for all }i\in[0,m]\}}{\exp{\left(-Pm+\sum_{k=0}^{m-1}\phi(\sigma^k\underline{x})\right)}}\le c_2$$
for all $\underline x\in \Sigma_A$ and $m\ge 0$.
\end{theorem}

Denote $\mc C(X)$ the family of continuous functions on a compact metric space.
\begin{theorem}\label{T:VarPrin}
For any $\phi\in \mc C(\Sigma_A)$ and $\mu\in\mc M_\s(\Sigma_A)$, one has that
$$h_\mu(\s)+\int \phi d\mu\le P_\s(\phi),$$
where $h_\mu(\s)$ is the metric entropy of $\mu$ and $P_\s(\phi)$ is the topological pressure of $\phi$.
\end{theorem}

\begin{theorem}\label{T:GibUnique}
Suppose that $\phi\in \mc F_A$, $(\Sigma_A,\s)$ is topologically mixing and $\mu_\phi$ is the Gibbs measure of $\phi$. Then,  $\mu_\phi$ is the unique $\mu\in \mc M_\s(\Sigma_A)$ such that
$$h_\mu(\s)+\int \phi d\mu=P_\s(\phi).$$
\end{theorem}

\begin{theorem}\label{T:VarPrin1}
Suppose that $T:X\to X$ is a continuous map on a compact metric space and $\phi\in \mc C(X)$. Then,  for any $\mu\in \mc M_T(X)$, one has that
$$h_\mu(T)+\int \phi d\mu\le P_T(\phi).$$
\end{theorem}

\begin{theorem}\label{T:Conjugacy}
Suppose that $T_1:X_1\to X_1$ and $T_2:X_2\to X_2$ are continuous maps on compact metric spaces $X_{1}$ and $X_{2}$ respectively, and $\pi:X_1\to X_2$ is continuous and onto and also satisfies that $\pi\circ T_1=T_2\circ \pi$. Then,  for all $\phi\in \mc C(X_2)$, one has that
$$P_{T_2}(\phi)\le P_{T_1}(\phi\circ \pi).$$
 \end{theorem}

\begin{theorem}\label{T:VarPrin2}
Suppose that $T:X\to X$ is a continuous map on a compact metric space and $\phi \in \mc C(X)$. Then,
$$P_T(\phi)=\sup_{\mu\in \mc M_T(X)}\left(h_\mu(T)+\int \phi d\mu\right).$$
\end{theorem}

For $\phi\in \mc C(X)$ and a continuous map $T:X\to X$ on a compact metric space $X$, if $\mu\in \mc M_T(X)$ satisfies that $h_\mu(T)+\int \phi d \mu=P_T(\phi)$, then $\mu$ is called an {\bf equilibrium state} for $\phi$ with respect to $T$. \\

\section{Equilibrium States of Uniform Hyperbolic Systems in a Hilbert Space}\label{S:Markov}
In this appendix, we assume that Conditions C1)-C3) hold and $f$ is uniformly hyperbolic on $\Lambda$. We show that  for every H\"{o}lder continuous function  $\phi:\L\to \mb R$, there exists a unique equilibrium state $\mu_{\phi}\in \mc M_f(\L)$ with respect to $f$ when $f|_{\L}$ is topologically mixing. The proof of this result is based on the Markov partition. The proof is exactly same as Bowen's \cite{Bowen}. In the construction of Markov partition, only the invertibility of $f|_{\L}$ is used even though $f$ is not invertible in $\mathbb{H}$.   For the sake of  convenience for readers, we include the proofs with details, which is almost identical to Bowen's.

\subsection{Construction of a Markov Partition}

In this section, we construct Markov Partitions of $\L$ for $f$ based on the shadowing lemma and symbolic dynamics.

We first note from Lemma \ref{L:CanonicalCoor} that  for any small $\e\in(0,\e_0)$ there is a $\d>0$ such that $W^s_\e(x)\cap W^u_\e(y)$ consists of a single point $[x,y]$ whenever $x,y\in \L$ and $|x-y|<\d$. Then, we define the following:

We call a subset $R\subset \L$  a {\it rectangle} if its diameter is small and
$[x,y]\in R\text{ for any }x,y\in R.$
$R$ is said to be {\it proper} if $R$ is closed and $R=\overline{int(R)}$ where $int(R)$ denotes the interior of $R$ as a subset of $\L$. For $x\in R$, denote
\[W^s(x,R)=W^s_\e\cap R\text{ and }W^u(x,R)=W^u_\e(x)\cap R
\]
where $\e\in(0,\e_0)$ is small and fixed (see Lemma \ref{L:SUManifolds}) and the diameter of $R$ is small compared to $\e$. Let
\begin{align*}
\partial^sR&=\{x\in R:\ x\notin int(W^u(x,R))\}; \\
\partial^uR&=\{x\in R:\ x\notin int(W^s(x,R))\}.
\end{align*} Here the interiors $int(W^u(x,R))$ and $int(W^s(x,R))$ are subsets of $W^u_\e(x)\cap \L$ and $W^s_\e(x)\cap \L$ respectively.

\begin{lemma}\label{L:RBoundary}
Let $R$ be a closed rectangle. $\partial^sR$ and $\partial^uR$ are closed subsets of $\L$ and $\partial R=\partial^sR\cup \partial^uR$.
\end{lemma}
\begin{proof}
Let $x\in int(R)$. Then, there exists a  neighbourhood $U^u$ of $x$ in $W^u_\e(x)$ such that  $W^u(x,R)=R\cap U^u$. Hence, $x\in int(W^u(x,R))$. Similarly, for $x\in int(R)$, we have $x\in int(W^s(x,R))$. To show $\partial R=\partial^sR\cup \partial^uR$, it is enough to show $\partial R\subset \partial^sR\cup \partial^uR$.

Suppose not, i.e., $x\in\partial R -(\partial^sR\cup \partial^uR)$. Then  $x\in int(W^u(x,R))\cap int(W^s(x,R))$. For each $y\in \L$ near $x$, we have
\[[x,y]\in W^s_\e(x)\cap \L \text{ and } [y,x]\in W^u_\e(x)
\]
and both $[x,y]$ and $[y,x]$ are continuous in $y$. Hence, for each $y\in \L$ close enough to $x$, we have $[x,y],  [y,x]\in R$. Therefore,
$$y'=[[y,x],[x,y]]\in R\cap W^s_\e(y)\cap W^u_\e(y)=\{y\}.$$
So,  $x\in int(R)$, a contradiction.

 Next, we show that $\partial^sR$ and $\partial^uR$ are closed. Since $R$ is closed, it suffices to show that these sets are closed in $R$. Suppose $x\in R-\partial^u R$, then, by definition, $x\in int(W^s(x,R))$, which is a neighborhood of $x$ in $W^s_\e(x)\cap \L$. For any $y\in R$, if $|y-x|$ is small enough, then
$$[x,y]\in int(W^s(x,R)),$$
and $[x,y]$ is continuous on $y$. Thus, for any $y'\in W^s_\e(y)\cap \L$ with $|y'-y|$ small enough, $$[x,y']\in int(W^s(x,R))$$
which implies that
$$y'=[y,[x,y']]\in R.$$
So $y\in int(W^s(y,R))$, we have proved that $\partial^u R$ is closed. Similar argument will show that $\partial^s R$ is closed.

\end{proof}

\begin{definition}\label{D:Markov}
A Markov partition of $\L$ is a finite covering $\mc R=\{R_1,\ldots, R_m\}$ of $\L$ by proper rectangles with
\begin{itemize}
\item[(a)] $int(R_i)\cap int(R_j)=\emptyset$ if $j\neq i$;
\item[(b)] $fW^u(x,R_i)\supset W^u(fx,R_j)$ and $fW^s(x,R_i)\subset W^s(fx,R_j)$ when $x\in int(R_i)$ and $fx\in int(R_j)$.
\end{itemize}
\end{definition}

\begin{lemma}\label{L:MarkovExist}
 Assume that $f|_\L$ is uniformly hyperbolic and topologically mixing. Then,  $\L$ has Markov partitions $\mc R$ of arbitrarily small diameter.
\end{lemma}

\begin{proof} Let $\b>0$ be sufficiently small (comparing to a fixed small $\e\in(0,\e_0]$ as in Lemma \ref{L:SUManifolds}, $\d$  as in Lemma \ref{L:CanonicalCoor}, and $\e$ as in Lemma \ref{L:Expansive}).  Choose $\a\in(0,\b)$ small as in Lemma \ref{L:Shadow} so that  every $\a$-pseudo-orbit in $\L$ is $\b$-shadowed in $\L$. Since $f$ is $C^2$ and $\L$ is compact, there is an open neighborhood $V$ of $\L$ in $\mb H$, on which $f$ is uniformly continuous. Thus,  there is $\g\in(0,\frac\a2)$ such that
\begin{equation}\label{E:UniControlOnV}
|fx-fy|<\frac\a2,\text{ for any }  x,y\in V \text{ with} |x-y|<\g.
\end{equation}
Let $P=\{p_1,\ldots,p_r\}$ be a finite $\g$-dense subset of $\L$, namely,  for each $x\in\L$ there is $p_s\in P$ such that $|x-p_s|<\g$. Let
$$\Sigma(P)=\left\{\underline q\in \prod_{-\infty}^\infty P:\ |fq_j-q_{j+1}|<\a \text{ for all } j\in \mb Z\right\}.$$
By using Lemma \ref{L:Shadow} and \ref{L:Expansive}, we have that for each $\underline q\in\Sigma(P)$ there exists a unique $\t(\underline q)\in\L$ which $\b$-shadows $\underline q$. Since $P$ is  a $\a$-dense subset of $\L$, for each $x\in\L$ there is a $\underline q$ with $x=\t(\underline q)$, and furthermore, one can make $|x-q_0|<\g$.

 Let $\underline q,\underline q'\in \Sigma(P)$ with $q_0=q'_0$. We define $\underline q^*=[\underline q,\underline q']\in \Sigma(P)$ by setting
$$q^*_j=\begin{cases}
q_j&\text{ for } j\ge 0\\
q'_j&\text{ for } j\le 0.
\end{cases}$$
Then,  $|f^j\t(\underline q)-f^j\t(\underline q^*)|\le 2\b$ for $j\ge 0$ and $|f^j\t(\underline q')-f^j\t(\underline q^*)|\le 2\b$ for $j\le 0$. Hence,
$$\t(\underline q^*)\in W^s_{2\b}(\t(\underline q))\cap W^u_{2\b}(\t(\underline q')),$$
which can also be written as
$$\t[\underline q,\underline q']=[\t(\underline q),\t(\underline q')].$$
Let
\[T_s=\{\t(\underline q)|\ \underline q\in \Sigma(P),\ q_0=p_s\}.
\]
For $x,y\in T_s$, write $x=\t(\underline q)$, $y=\t(\underline q')$ with $q_0=q'_0=p_s$. Then,
$$[x,y]=\t[\underline q,\underline q']\in T_s,$$
which implies that $T_s$ is a rectangle.

Next, we show that if $x=\t(\underline q)$ with $q_0=p_s$ and $q_1=p_t$, then
\begin{equation}\label{E:FInc}
fW^s(x,T_s)\subset W^s(fx,T_t),
\end{equation}
and
\begin{equation}\label{E:BInc}
fW^u(x,T_s)\supset W^u(fx,T_t).
\end{equation}

For $y\in W^s(x,T_s)$, suppose $y=\t(\underline q')$ with $q'_0=p_s$. Then,
$$y=[x,y]=\t[\underline q,\underline q']\text{ and }fy=\t(\sigma[\underline q,\underline q'])\in T_t,$$
where $\sigma$ is the shift operator and $\sigma[\underline q,\underline q']$ has $p_t$ in its zeroth position. Since $fy\in W^s_{\e}(fx)$,   $fy\in W^s(fx,T_t)$. So, (\ref{E:FInc}) holds.

(\ref{E:BInc}) can be proved in the same fashion. We consider $y\in W^u(fx,T_t)$ and assume $y=\t(\underline q')$ with $q'_0=p_t$. Then,
$$y=[y,fx]=\t[\underline q',\sigma(\underline q)]\text{ and }f^{-1}y=\t(\sigma^{-1}[\underline q',\sigma(\underline q)])\in T_s,$$
 as $\sigma^{-1}[\underline q',\sigma(\underline q)]$ has $p_s$ in its zeroth position. Since $f^{-1}y\in W^u_\e(x)$, $f^{-1}y\in W^u_\e(x,T_s)$. Hence, (\ref{E:BInc}) holds.

 The following lemma can be proved by using the exact same argument used in \cite{Bowen}.
\begin{lemma}\label{L:Closeness}
$\t:\Sigma(P)\to \L$ is continuous.
\end{lemma}
A immediate consequence of the above lemma is that each $T_s$ is closed.\\

Note that $\mc T=\{T_1,\ldots, T_r\}$ is a covering by rectangles. However the  $T_j$'s may not be proper or interior of them may not be  disjoint.

For each $x\in\L$, let
$$\mc T(x)=\{T_j\in\mc T|\ x\in T_j\}\text{ and } \mc T^*(x)=\{T_k\in\mc T|\ T_k\cap T_j\neq \emptyset\text{ for some } T_j\in \mc T(x)\}.$$
  Since $\mc T$ is a closed cover of $\L$ and $\partial T_j$'s are nowhere dense sets,   $Z=\L-\left(\cup_j\partial T_j\right)$ is an open dense subset of $\L$. Letting
$Z^*=\{x\in\L|\ W^s_{\frac13\e}(x)\cap \partial^sT_k=\emptyset \text{ and }W^u_{\frac13\e}(x)\cap \partial^uT_k=\emptyset \text{ for all }T_k\in\mc T^*(x)\},$ we have the following lemma.
\begin{lemma}\label{L:Z*Open}
$Z^*$ is an open and dense subset of $\L$.
\end{lemma}
\begin{proof}
We first prove  that $Z^*$ is open. For a given $x\in Z^*$, by definition, we have that
\[
x\in\left(\cap_{T_k\in\mc T(x)}int(T_k)-\cup_{T_k\notin\mc T(x)}T_k\right),
 \]which is open. For any $x,y\in \L$, if  $|y-x|$ is small enough, then $\mc T(y)=\mc T(x)$.
  Using Lemma \ref{L:RBoundary}, we have that $\cup_{T_k\in\mc T^*(x)}\partial^\tau T_k$ is closed in $\L$ for $\tau=u,s$. Since, for $\tau=u,s$, $W^\tau_\e(x)$ is also closed, and the space is normal, $W^\tau_\e(x)$ and $\cup_{T_k\in\mc T^*(x)}\partial^\tau T_k$ are separated with a positive distance. Then, by Lemma \ref{L:SUManifolds}, for $y\in \L$ if $|y-x|$ is small enough, then $W^\tau_\e(y)$ is close enough to $W^\tau_\e(x)$ so that  $W^\tau_\e(y)$ is disjoint to $\cup_{T_k\in\mc T^*(x)}\partial^\tau T_k$ for $\tau=u,s$. Therefore, $y\in Z^*$. Thus,  $Z^*$ is open.

 Define
 \begin{equation}\label{E:T_k^*}
 T_k^*=\left(\cup_{x\in \partial^s T_k}(W^s_{\frac13\e}(x)\cap \L)\right)\cup \left(\cup_{x\in \partial^u T_k}(W^u_{\frac13\e}(x)\cap \L)\right).
 \end{equation}
 Clearly, $Z^*\supset\L-\cup_{T_k\in \mc T}T_k^*$. So it is enough to  show that each $T_k^*$ is a nowhere dense set in $\L$.

 We first note that $T_k^*$ is closed. In fact,  for any $x\notin T_k^*$, by definition of $T_k^*$, one has that
 \[
 W^s_{\frac13\e}(x)\cap \partial^s T_k=\emptyset\text{ and } W^u_{\frac13\e}(x)\cap \partial^u T_k=\emptyset.
 \]
  Since $W^s_{\frac13\e}(x)\cap \L$, $\partial^s T_k$, $W^u_{\frac13\e}(x)\cap\L$, and $\partial^u T_k$ are all closed subsets $\L$, and $\L$ is a normal metric space, we have  for $\tau=u,s$, $W^\tau_{\frac13\e}(x)$ and $\partial^\tau T_k$ has a positive separation. By the continuous dependence of $W^\tau_{\e}(x)$ on $x\in \L$, we have that for any $y\in\L$ closed enough to $x$, $y\notin T_k^*$, which yields that $T_k^*$ is closed.

 In the following, we prove that $int(T_k^*)=\emptyset$. Suppose not, i.e., there is $x\in int(T_k^*)$. We will prove that $x\in int(T_k)$ which gives a contradiction. Suppose this is true. Since for $\tau=u,s$, $W^\tau(x,T_k)$ is a neighborhood of $x$ in $W^\tau_\e\cap \L$,  if there is a $y\in \partial^u T_k\cap W^u_{\frac13\e}(x)$, then for any $y'\in W^s_\e(y)\cap \L$ with $|y'-y|$ small enough
 \[[x,y']\in int(W^s(x,T_k))\text{ thus }y'=[y,[x,y']]\in W^s(y,T_k),
 \]
 which implies that $y\in int(W^s(y,T_k))$ which contradicts $y\in \partial^u T_k$. Thus,
  \[\partial^u T_k\cap W^u_{\frac13\e}(x)=\emptyset.
  \]
 The same argument gives
 \[\partial^s T_k\cap W^u_{\frac13\e}(x)=\emptyset.
 \]
 Hence, $x\notin T_k^*$, a contradiction.

  We now prove that $x\in int(T_k)$ provided $x\in int(T_k^*)$. Suppose that there exists  $y_1\in \partial^u T_k$ such that $x\in W^u_{\frac13\e}(y_1)$. Since $x\in int (T^*_k)$ and $y\in \partial^u T_k$, there is  $y_1\in (W^s_{\frac13\e}(y)\cap\L-T_k)$ such that $[x,y_1]\in int (T^*_k)$. We note that $\partial^u T_k \cap W^u_{\frac13\e}([x,y_1])=\emptyset$. If not, then there is a $y_1'\in \partial^u T_k \cap W^u_{\frac13\e}([x,y_1])=\emptyset$. Thus, $[x,y_1]=[y,y_1']\in T_k$.  Then $y_{1}=[y,[x,y_{1}]]\in T_{k}$, a contradiction. Therefore,  there is a $y_2\in \partial^s T_k$ such that $[x,y_1]\in W^s_\e(y_2)$, thus $x\in W^s_{\e}(y_2)$. Hence,
 $$x=[y_2,y_1]\in T_k.$$
 Since  $x$ is an arbitrary  point, we have that $int(T_k^*)\subset T_k$.  Hence, $x\in int(T_k)$.
 By switching the super index "u" and "s" above, using exactly same argument, we obtain  the same conclusion. The proof is complete.
    \end{proof}

For $T_j\cap T_k\neq \emptyset$, let
\begin{align*}
&T_{j,k}^1=\{x\in T_j:W^u(x,T_j)\cap T_k\neq \emptyset,\ W^s(x,T_j)\cap T_k\neq \emptyset\}=T_j\cap T_k,\\
&T_{j,k}^2=\{x\in T_j:W^u(x,T_j)\cap T_k\neq \emptyset,\ W^s(x,T_j)\cap T_k= \emptyset\},\\
&T_{j,k}^3=\{x\in T_j:W^u(x,T_j)\cap T_k= \emptyset,\ W^s(x,T_j)\cap T_k\neq \emptyset\},\\
&T_{j,k}^4=\{x\in T_j:W^u(x,T_j)\cap T_k= \emptyset,\ W^s(x,T_j)\cap T_k= \emptyset\}.
\end{align*}
  We note that for $x,y\in T_j$, $W^s([x,y],T_j)=W^s(x,T_j)$ and $W^u([x,y],T_j)=W^u(y,T_j)$. Thus,   each $T_{j,k}^n$ is a rectangle and each $x\in T_j\cap Z^*$ lies in $int(T_{j,k}^n)$ for some $n\in\{1,2,3,4\}$. By definition, $T_{j,k}^4$ is open in $\L$ and $T_{j,k}^1$ is close in $\L$.

By definition of $Z^{*}$, for each $x\in Z^*$ one can define a nonempty set
\[R(x)=\cap\left\{int(T_{j,k}^n)\big|\ x\in T_j, T_k\cap T_j\neq \emptyset \text{ and } x\in T_{j,k}^n\right\}.
\]
 For $x\in Z^*$, $R(x)$ is an open rectangle in $\L$. Let $y\in R(x)\cap Z^*$. Since $R(x)\subset T_s$ for any $T_s\in \mc T(x)$ and $R(x)\cap T_j=\emptyset$ for any $T_j\notin \mc T(x)$, we have that $\mc T(y)=\mc T(x)$. For $T_j\in \mc T(x)=\mc T(y)$ and $T_k\cap T_j\neq\emptyset$, $y\in T^n_{j,k}$  since $R(x)\subset T^n_{j,k}$. Hence, $R(y)=R(x)$. Suppose $R(x)\cap R(x')\neq\emptyset$ for some $x,x'\in Z^*$. Then, there is a $y\in R(x)\cap R(x')\cap Z^*$ since  $Z^*$ is open and dense. So,  $R(x)=R(y)=R(x')$. Since there are only finitely many $T_{j,k}^n$'s, there are only finitely many distinct $R(x)$'s. Let
\[\mc R=\left\{\overline{R(x)}|\ x\in Z^*\right\}=\{R_1,\ldots,R_m\}.
\]
Since for $x,x'\in Z^*$, either $R(x')=R(x)$ or $R(x)\cap R(x')=\emptyset$, $(\overline{R(x)}-R(x))\cap Z^*=\emptyset$. Since $Z^*$ is dense in $\L$, we have that $\overline{R(x)}-R(x)$ has no interior in $\L$. Using $R(x)\subset int(\overline{R(x)})$, we obtain
$$\overline{int(\overline{R(x)})}=\overline{R(x)}\ i.e.\ \overline{R(x)}\text{ is proper}.$$
For $R(x)\neq R(x')$ with $x,x'\in Z^{*}$,  let $R(x,x')=int\left(\overline{R(x)}\right)\cap int\left(\overline{R(x'})\right)$ which is an open subset of $\L$. Since $R(x)$ and $R(x')$ are open and dense in $int\left(\overline{R(x)}\right)$ and $int\left(\overline{R(x')}\right)$ respectively and $R(x)\cap R(x')=\emptyset$, we have that $R(x,x')\cap R(x)$ and $R(x,x')\cap R(x')$ are disjoint open and dense subsets of $R(x,x')$. This can not happen if $R(x,x')\neq \emptyset$, thus $$R(x,x')=int\left(\overline{R(x)}\right)\cap int\left(\overline{R(x'})\right)=\emptyset.$$

To prove that $\mc R$ is a Markov partition, we need to show condition (b) in Definition \ref{D:Markov} holds.

\noindent Next we prove the following technical lemma.
\begin{lemma}\label{L:R(fx)}
 For any $x,y\in Z^*\cap f^{-1}Z^*$, if $R(x)=R(y)$ and $y\in W^s_\e(x)$, then $R(fx)=R(fy)$.
 \end{lemma}
 \begin{proof}
  We first prove that $\mc T(fx)=\mc T(fy)$. Suppose $fx\in T_j$.  Let $fx=\t(\sigma\underline{q})$ with $q_1=p_j$ and $q_0=p_s$. Then,  $x=\t(\underline{q})\in T_s$. That $R(x)=R(y)$ implies $y\in T_{s}$. Then,  by (\ref{E:FInc}), we have that
\[
fy\in fW^s(x,T_s)\subset W^s(fx,T_j)
\]
which implies that $\mc T(fx)\subset \mc T(fy)$.
Note that $y\in W^{s}_{\e}(x)$ if and only if $x\in W^{s}_{\e}(y)$ here. Similarly,  $\mc T(fy)\subset \mc T(fx)$. Hence,  $\mc T(fx)=\mc T(fy)$.

Now suppose $fx,fy\in T_j$ and $T_k\cap T_j\neq\emptyset$. We prove that $fx, fy \in T_{j,k}^n$. Since $fy\in W^s_\e(fx)$, we have that $W^s(fy,T_j)=W^s(fx,T_j)$. Suppose
$$W^u(fy,T_j)\cap T_k=\emptyset \text{ while } W^u(fx,T_j)\cap T_k\neq \emptyset,$$
from which we will derive a contradiction.

Let $z\in \L$ such that $fz\in W^u(fx,T_j)\cap T_k$. Since $fx=\t(\sigma(\underline{q})), q_1=p_j, q_0=p_s$,  from (\ref{E:BInc}), we have that
\[
fz\in W^u(fx,T_j)\subset fW^u(x,T_s)\text{ or in another word } z\in W^u(x,T_s).
\]
Let $fz=\t(\sigma\underline{q}')$ with $q_1'=p_k$ and $q_0=p_t$. Then,  $z\in T_t$, and $fW^s(z,T_t)\subset W^s(fz,T_k)$ by (\ref{E:FInc}). Thus,  we have that $T_s\in \mc T(x)=\mc T(y)$ and $z\in T_t\cap T_s\neq\emptyset$. Since $x, y\in T^n_{s,t}$ and $z\in W^u(x,T_s)$, there exists a $z'\in W^u(y,T_s)\cap T_t$. Therefore,
\[
z''=[z,y]=[z,z']\in W^s(z,T_t)\cap W^u(y,T_s).\]
Since $fz,fy\in T_j$, we have
$$fz''=[fz,fy]\in W^s(fz,T_k)\cap W^u(fy,T_j),\text{ thus }W^{u}(fy,T_{j})\cap T_{k}\neq \emptyset,$$
which gives a contradiction. Therefore,  $R(fx)=R(fy)$. This completes the proof of the lemma.
\end{proof}

Let $Y^*=\cup_{T_k\in \mc T}T_k^*$, where $T_k^*$ is defined in (\ref{E:T_k^*}). In the proof of Lemma \ref{L:Z*Open}, we have shown that each $T_k^*$ is closed and nowhere dense. Thus,  $Y^*$ is closed and nowhere dense since there are only finitely many $T_k^*$'s. Furthermore, if $x\notin Y^*\cup f^{-1}Y^*$, then $x\in Z^*\cap f^{-1}Z^*$. We claim that the set $W^s(x,R(x))\cap (Z^*\cap f^{-1}Z^*)$ is open and dense in $W^s(x,\overline{R(x)}) \subset W^s_\e(x)\cap \L$. This claim is an immediate consequence of the  following lemma since $W^{s}(x,R(x))\cap (Z^*\cap f^{-1}Z^*)$ contains a subset
$$
 \left(W^{s}(x,R(x))-\cup_{T_{k}\subset \mc T^{*}(x)}T_{k}^{*}\right)\cap (f|_{\L})^{-1}\left(W^{s}(fx,R(fx))-\cup_{T_{k}\subset \mc T^{*}(fx)}T_{k}^{*}\right).$$

\begin{lemma}\label{L:T_k^*}
Given any $x\in Z^*$ and each $T_k^*$ with $T_{k}\in \mc T^{*}(x)$, the set $(W^\tau_{\frac13\e}(x)\cap \L)\cap T_k^*$ is closed and nowhere dense  in $W^\tau_{\frac13\e}(x)\cap \L$ for $\tau=u,s$.
\end{lemma}
\begin{proof}
Otherwise, suppose there is a $z\in int((W^s_{\frac13\e}(x)\cap \L)\cap T_k^*)$ (as a subset of $W^s_{\frac13\e}(x)\cap \L$) for some $T_k^*\in \mc T^{*}(x)$ (noting that $T_{k}^{*}$ is a closed subset of $\L$). By definition of $Z^*$ and $T_k^*$, we have that $(W^s_{\frac13\e}(x)\cap \L)\cap \partial^{s} T_{k}=\emptyset$, thus there is $y\in \partial^u T_k$ so that $z=[x,y]$. By Lemma \ref{L:SUManifolds}, for any $y_1\in W^s_\e(y)\cap \L$ with $|y_1-y|$ small enough, one has that
$$[x,y_1]\in int ((W^s_\e(x)\cap \L)\cap T_k^*) (\text{as a subset of } W^s_{\frac13\e}(x)\cap \L).$$
Also note that for any $z'\in int((W^s_{\frac13\e}(x)\cap \L)\cap T_k^*)$, there exists $y'\in \partial^u T_k$ so that $z'=[x,y']$, then we have
$$y_1=[y,y_1]=[y,[x,y_1]]=[y,y']\text{ for some }y'\in \partial^u T_k,$$
which implies that $y_1\in T_k$. Thus,  $y\in int(W^s(y,T_k))$ which contradicts to that $y\in \partial^{u}T_{k}$. \\
The same argument works for the part of $W^u$, and the proof is omitted here. This completes the proof of the lemma.
\end{proof}

 Since, for any $x\in R(x)$, $W^{s}(x,R(x))\cap (Z^*\cap f^{-1}Z^*)$ is open and dense as a subset of $W^s_\e(x)\cap \L$. From Lemma \ref{L:R(fx)}, we have that $R(fy)=R(fx)$ for any $y\in W^{s}(x,R(x))\cap Z^{*}\cap f^{-1}Z^{*}$. By the  continuity of $f|_{\L}$, we have that
$$fW^s(x,\overline{R(x)})\subset \overline{R(fx)}.$$
 So $fW^s(x,\overline{R(x)})\subset W^s(fx,\overline{R(fx)})$ as  $fW^s(x,\overline{R(x)})\subset W^s_\e(fx)$.\\

This completes the proof of half of condition (b) in Definition \ref{D:Markov}. The other half can be proved in the exactly same way by working on $\left(f|_{\L}\right)^{-1}$ and the proof is omitted. Note that, since $\L$ is compact and $f|_{\L}$ is homeomorphism, then $\left(f|_{\L}\right)^{-1}$ is also a homeomorphism. And, also noting that, in the above proof, we only used $W^s\cap \L$ and $W^u\cap \L$ on which $f$ is invertible. So one can apply the above argument to $\left(f|_{\L}\right)^{-1}$ with $(W^u\cap\L)_{\left(f|_{\L}\right)^{-1}}= (W^s\cap \L)_{f|_\L}$ and $(W^s\cap\L)_{\left(f|_{\L}\right)^{-1}}= (W^u\cap \L)_{f|_\L}$.
\end{proof}


\subsection{Symbolic Dynamics Induced by a Markov Partition}\label{S:SymDyn}
In this section, we will build the connection between $f|_\L$ and a symbolic dynamical system generated by a Markov partition.

Throughout this section, we fix a Markov partition of $\L$ with diameter $<\b$ for some small $\b>0$, and denote which by $\mc R=\{R_1,\ldots, R_m\}$. Then,  one can define the {\it transition matrix} $A=A(\mc R)$ by
$$A_{ij}=\begin{cases}
1 &\text{ if } int(R_i)\cap f^{-1} int (R_j)\neq \emptyset\\
0 &\text{ otherwise}.
\end{cases}
$$
\begin{lemma}\label{L:FBInc}
Suppose $x\in R_i$, $fx\in R_j$ and $A_{ij}=1$. Then,  $fW^s(x,R_i)\subset W^s(fx,R_j)$ and $fW^u(x,R_i)\supset W^u(fx,R_j)$.
\end{lemma}
\begin{proof}
If $int(R_i)\cap f^{-1}int(R_j)\neq\emptyset$, then there is an $x'\in (int(R_i)\cap f^{-1}int(R_j))\cap (Z^*\cap f^{-1}Z^*)$. For any $x\in R_i\cap f^{-1} R_j$, we have that $W^s(x,R_i)=\{[x,y]|\ y\in W^s(x',R_i)\}$. Then, by definition of Markov partition, we have that
\begin{align*}
fW^s(x,R_i)&=\{[fx,fy]|\ y\in W^s(x',R_i)\}\subset\{[fx,z]|\ z\in W^s(fx',R_j)\}= W^s(fx,R_j).
\end{align*}
The second part follows if one considers the map $(f|_\L)^{-1}$ (see the paragraph in the end of the proof of Lemma \ref{L:MarkovExist}).
\end{proof}

Define that
$$\partial^s\mc R=\cup_j\partial^s R_j\text{ and } \partial^u\mc R=\cup_j\partial^u R_j.$$
\begin{lemma}\label{L:RBInc}
$f(\partial^s \mc R)\subset \partial^s\mc R$ and $f^{-1}(\partial^u \mc R)\subset \partial^u\mc R$.
\end{lemma}
\begin{proof}
We will show that for any $x\in R_i$, if $x\notin \partial^u\mc R$ then $fx\notin \partial^u \mc R$. Note that the set
$$\cup_j\left(int(R_i)\cap f^{-1}int(R_j)\right)$$ is open and dense in $R_i$. Then,  there are a $j$ and $x_n\in int(R_i)\cap f^{-1}int(R_j)$ with $\lim_{n\to\infty}x_n=x$. Therefore, $A_{ij}=1$, $x\in R_i$ and $fx\in R_j$. So, by Lemma \ref{L:FBInc}, $fW^s(x,R_i)\subset W^s(fx,R_j)$. Since $x\notin\partial^u \mc R$,  $W^s(x,R_i)$ is a neighborhood of $x$ in $W^s_\e(x)\cap \L$. So, by the continuity of $(f|_\L)^{-1}$, $W^s(fx,R_j)$ is a neighborhood of $fx$ in $W^s_\e(fx)\cap \L$, thus $fx\notin \partial^uR_j$. We have proved that $f^{-1}(\partial^u \mc R)\subset \partial^u\mc R$. The other part follows from the  similar argument which is omitted here.
\end{proof}

\begin{lemma}\label{L:ProRec}
For small $\d>0$, let $D^s\subset W^s_\d(x)\cap \L$ and $D^u\subset W^u_\d(x)\cap \L$. Then,  $[D^u,D^s]:=\{[y_{1},y_{2}]|\ y_{1}\in D^u \text{ and } y_{2}\in D^s\}$ is a proper rectangle if and only if $D^\tau=\overline{int(D^\tau)}$ as subsets of $W^\tau_\d(x)\cap \L$ for $\tau=u,s$.
\end{lemma}
\begin{proof}
It is easy to see that $[D^u,D^s]$ is a rectangle, and by Lemma  \ref{L:CanonicalCoor} and the compactness of $\L$, we have that
$$\overline{[int (D^u), int( D^s)]}=[\overline{int (D^u)}, \overline{int (D^s)}].$$
Then,  this lemma follows from the following fact that
$$int([D^u,D^s])=[int(D^u),int(D^s)].$$
Now we prove the above identity.

Let $y\in int([D^u,D^s])$. Since $int([D^u,D^s])$ contains an open neighbourhood of $y$ as a subset in $\L$, which is denoted by $V$. Then,  $W^\tau_\d(y)\cap V$ is an open neighbourhood of $y$ as a subset in $W^\tau_\d(y)\cap \L$, for $\tau=u,s$. For any $y'\in W^s_\d(y)\cap V$, since $y'\in [D^u,D^s]$, and also note that
$$y'=[[y',x],[x,y']] \text{ with }[y',x]=[y,x]\in W^u_\d(x) \text{ and }[x,y']\in W^s_\d(x) \text{ being uniquely defined},$$
then $[x,y']\in D^s$ is uniquely defined for any $y'\in W^s_\d(y)\cap V$. By Lemma  \ref{L:CanonicalCoor}, for any $y''\in W^s_\d(x)$ with $|y''-[x,y]|$ small enough, $[y,y'']\in W^s_\d(y)\cap V$. Then,  $y''=[x,[y,y'']]\in D^s$ which implies that $[x,y]\in int(D^s)$. The same argument shows that $[y,x]\in int(D^u)$. So $y\in [int(D^u),int(D^s)]$. Hence, $int([D^u,D^s])\subset[int(D^u),int(D^s)]$.

For the other part, let $y=[y_1,y_2]$ with $y_1\in int(D^u)$ and $y_2\in int(D^s)$. Then, by Lemma  \ref{L:CanonicalCoor}, for $y'\in \L$ with $|y'-y|$ small enough, one has that
$$[y',y_1]\in int(D^u)\text{ and }[y_2,y']\in int(D^s),$$
which implies that $y'\in[D^u,D^s]$. Thus,  $y\in int([D^u,D^s])$. Hence,  $int([D^u,D^s])\supset[int(D^u),int(D^s)]$. The proof is complete.
\end{proof}

\begin{definition}\label{D:USubRec}
Let $R,R'$ be two rectangles. $R'$ is called a {\it u-subrectangle} of $R$ if the following are satisfied:
\begin{itemize}
\item[(a)] $R'\neq\emptyset$, $R'\subset R$ and $R'$ is proper;
\item[(b)] $W^u(x,R')=W^u(x,R)$ for all $x\in R'$.
\end{itemize}
\end{definition}

\begin{lemma}\label{L:USubRec}
Let $S$ be a u-subrectangle of $R_i$ and $A_{ij}=1$. Then,  $f(S)\cap R_j$ is a u-subrectangle of $R_j$.
\end{lemma}
\begin{proof}
For any $x\in R_i\cap f^{-1}R_j$, set $D=W^s(x,R_i)\cap S$. Since $S$ is a u-subrectangle, we have that
$$S=\cup_{y\in D}W^u(y,R_i)=[W^u(x,R_i),D].$$
Since $S$ is proper and nonempty, by Lemma \ref{L:ProRec}, we have that $D\neq \emptyset$ and $D=\overline{int(D)}$. Therefore, we have that
$$f(S)\cap R_j=\cup_{y\in D}(fW^u(y,R_i)\cap R_j).$$
By Lemma \ref{L:FBInc}, we have that for all $y\in D$, $f(y)\in R_j$ and $fW^u(y,R_i)\cap R_j=W^u(fy,R_j)$. Then,
\begin{equation}\label{E:ExpFS}
f(S)\cap R_j=\cup_{y'\in f(D)}W^u(y',R_j)=[W^u(fx,R_j),f(D)].
\end{equation}
Noting that $R_j=[W^u(fx,R_j),W^s(fx,R_j)]$ is a proper rectangle, then, by Lemma \ref{L:ProRec}, we have that
\begin{equation}\label{E:ProperB1}
W^u(fx,R_j)=\overline{int(W^u(fx,R_j))}
 \end{equation}
 as a subset of $W^u_\e(fx)\cap \L$. As $f|_{W^s_\e(x)\cap \L}$ is a hemeomorphism onto a neighborhood of $fx$ in $W^s_\e(fx)\cap \L$, we have that
\begin{equation}\label{E:ProperB2}
f(D)=\overline{int(f(D))} \text{ since } D=\overline{int(D)},
\end{equation}
which together with (\ref{E:ProperB1}) imply that $f(S)\cap R_j$ is proper by Lemma \ref{L:ProRec}.

Obviously, $f(S)\cap R_j\neq \emptyset$ since $f(D)\neq \emptyset$. For any $y''\in f(S)\cap R_j$, by (\ref{E:ExpFS}), there exists $y'\in f(D)$ such that $y''\in W^u(y',R_j)$, then
$$W^u(y'',R_j)=W^u(y',R_j)\subset f(S)\cap R_j \text{ thus } W^{u}(y'',f(S)\cap R_{j})=W^{u}(y'',R_{j}).$$
Now we have proved that $f(S)\cap R_j$ is a u-subrectangle of $R_j$.
\end{proof}

To build the connection between symbolic dynamics and $f|_\L$,  for each $\underline a\in\Sigma_A$, consider the set
$$\Pi(\underline{a})=\cap_{j\in\mb Z}f^{-j}R_{a_j},$$
which will be shown consisting only one point thus induces a map from $\Sigma_A$ to $\L$.

\begin{lemma}\label{L:SymDynByMP}
For each $\underline{a}\in \Sigma_A$, $\Pi(\underline{a})$ contains  a single point which is denoted by $\pi(\underline{a})$. The map
$\pi:\Sigma_A\to \L \text{ is a continuous surjection satisfying  that } \pi\circ \sigma=f\circ \pi.$
Furthermore, $\pi$ is one-to-one over the set $Y=\L-\left(\cup_{i\in \mb Z}f^j(\partial^s\mc R\cup \partial^u\mc R)\right)$.
\end{lemma}
\begin{proof}
For any word $a_1a_2\cdots a_n$ with $A_{a_ja_{j+1}}=1, j=1,\ldots,n-1$, note that
$$\cap_{j=1}^nf^{n-j}R_{a_j}=R_{a_n}\cap f\left(\cap_{j=1}^{n-1}f^{n-1-j}R_{a_j}\right).$$
By Lemma \ref{L:USubRec}, we have that $R_{a_2}\cap f(R_{a_1})$ is a u-subrectangle of $R_{a_2}$. Then, by induction, we have that
$\cap_{j=1}^nf^{n-j}R_{a_j}$ is a u-subrectangle of $R_{a_n}$. Consider the set
$$K_n(\underline{a}):=\cap_{j=-n}^nf^{-j}R_{a_j}.$$
Noting that
$$K_n(\underline{a})=f^{-n}\left(\cap_{j=-n}^nf^{n-j}R_{a_j}\right),$$
and since $(f|_{\L})^{-1}$ is continuous, one has that
$$K_n(\underline{a}):=\cap_{j=-n}^nf^{-j}R_{a_j}=\overline{int(K_n(\underline{a}))}\neq\emptyset.$$
By definition, it is obvious that $K_n(\underline{a})\supset K_{n+1}(\underline{a})\supset\cdots$, thus
$$K(\underline{a}):=\cap_{j=-\infty}^\infty f^{-j}R_{a_j}=\cap_{n=1}^\infty K_n(\underline{a})\neq \emptyset.$$
Suppose $x,y\in  K(\underline{a})$, then $f^jx,f^j y\in R_{a_j}$ for all $j\in \mb Z$, and thus $x=y$ by Lemma \ref{L:Expansive}.
Noting that
$$K(\sigma\underline{a})=\cap_{j\in\mb Z}f^{-j}R_{a_{j+1}}=f\left(\cap_{j\in\mb Z}f^{-j}R_{a_j}\right)=fK(\underline{a}),$$
then we have that
$$\pi\circ \sigma=f\circ \pi.$$
Now we show that $\pi$ is continuous. Otherwise there is a $\g>0$ such that for every $N\in \mb N$ there are $\underline{a}_N, \underline{a}'_N\in \Sigma_A$ with $a_{N,j}=a_{N,j}'$ for all $j\in[-N,N]$ and $|\pi\underline{a}_N-\pi\underline{a}_N'|>\g$. Note that
$$|f^j(\pi\underline{a}_N)-f^j(\pi\underline{a}_N')|<\b\ (\text{the diameter of the partition}) \text{ for all } j\in[-N,N],$$
since $f^j(\pi\underline{a}_N),f^j(\pi\underline{a}_N')\in R_{a_{N,j}}$ for all $j\in[-N,N]$. By taking a subsequence, we may assume that $\pi\underline{a}_N\to x$ and $\pi\underline{a}_N'\to x'$ as $N\to \infty$. Then,
$$|f^jx-f^jx'|\le \b \text{ for all } j\in\mb Z\text{ and } |x-x'|\ge \g,$$
which contradicts Lemma \ref{L:Expansive}. Thus,  $\pi$ is continuous.

Given $x\in Y$ and suppose $f^jx\in R_{a_j}$ for all $j\in \mb Z$. Since $x\in Y$, for all $j\in \mb Z$, we have that $f^jx\in int(R_{a_j})$,  and then $A_{a_ja_{j+1}}=1$. Thus,  $x=\pi(\underline{a})$ with $\underline{a}=\{a_j\}_{j\in\mb Z}\in \Sigma_A$. So $Y\subset \pi(\Sigma_A)$. Suppose $x=\pi(\underline{b})$, then $f^jx\in R_{b_j}$ and thus $b_j=a_j$ since $f^jx\neq \partial^u\mc R\cup\partial^s\mc R$. So $\pi$ is one-to-one over $Y$. Since $Y$ is a dense set of $\L$ and $\pi(\Sigma_A)$ is a closed (thus compact) subset of $\L$ by the continuity of $\pi$, one has that $\pi(\Sigma_A)=\L$. The proof is complete.
\end{proof}

\begin{lemma}\label{L:ShiftMixing}
If $f|_\L$ is topologically mixing so is $\sigma:\Sigma_A\to \Sigma_A$.
\end{lemma}
\begin{proof}
Let $U,V$ be two nonempty open subsets of $\Sigma_A$. Then,  for some $\underline{a}, \underline{b}\in\Sigma_A$ and $N\in \mb N$
\begin{align*}
&U\supset U'=\{\underline{x}\in \Sigma_A|\ x_j=a_j\text{ for all } j\in [-N,N]\},\\
&V\supset V'=\{\underline{x}\in \Sigma_A|\ x_j=b_j\text{ for all } j\in [-N,N]\}.
\end{align*}
Therefore, by recalling the definition of $K_N$ in the proof of Lemma \ref{L:SymDynByMP}, one has that
\begin{align*}
&U'':=int(K_N(\underline{a}))=\cap_{j=-N}^Nf^{-j}int(R_{a_j})\neq\emptyset,\\
&V'':=int(K_N(\underline{b}))=\cap_{j=-N}^Nf^{-j}int(R_{b_j})\neq\emptyset.
\end{align*}
Also note that for any $\underline{x}\in\Sigma_A$ such that $x=\pi(\underline{x})\in U''$, one has $f^jx\in R_{x_j}$ for all $j\in \mb Z$ and $f^jx\in int(R_{a_j})$ for all $j\in[-N,N]$, which implies that $x_j=a_j$ for all $j\in[-N,N]$. Thus,  $\pi^{-1}(U'')\subset U'$. Similarly, we have that $\pi^{-1}(V'')\subset V'$. Since $f|_\L$ is topologically mixing, there exists $N_0>0$ such that $f^nU''\cap V''\neq \emptyset$ for all $n>N_0$. Then,
$$\s^nU\cap V\supset\pi^{-1}(f^nU'')\cap \pi^{-1}(V'')=\pi^{-1}(f^nU''\cap V'')\neq\emptyset\text{ for all }n>N_0.$$
The proof is complete.
\end{proof}

\subsection{Equilibrium State of an H\"older Continuous Function}\label{SS:Equilibrium}
The main aim of this section is to prove the following Proposition.
\begin{prop}\label{P:GibState}
Assume that $f|_{\L}$ is topologically mixing. For every H\"{o}lder continuous function  $\phi:\L\to \mb R$, there exists a unique equilibrium state $\mu_{\phi}\in \mc M_f(\L)$ with respect to $f$.
\end{prop}

To prove Proposition \ref{P:GibState}, we need the following lemma.
\begin{lemma}\label{L:QExpansive}
There are $\d>0$ and $\a\in(0,1)$ such that if $x,y\in \L$ satisfying $|f^kx-f^ky|<\d$ for all $k\in[-N,N]$ then $|x-y|<\a^N$.
\end{lemma}
\begin{proof}
For a given $\l\in (0,\l_0)$ fix an $\e\in(0,\e_0)$ as in Lemma \ref{L:SUManifolds} and then let $\d$ be smaller than the ``$\d$''  in Lemma \ref{L:CanonicalCoor} for such $\e$. Suppose $|f^kx-f^ky|<\d$ for all $k\in[-N,N]$, then it is straightforward to obtain the following estimates by applying Lemma \ref{L:SUManifolds}:
\begin{align*}
&\left|[f^kx,f^ky]-f^kx\right|\ge e^{-k\l}|x-[x,y]| \text{ for all } k\in [-N,0],\\
&\left|[f^kx,f^ky]-f^ky\right|\ge e^{k\l}|y-[x,y]|\text{ for all } k\in[0,N].
\end{align*}
Therefore, $$|x-y|\le |x-[x,y]|+|y-[x,y]|\le e^{-N\l}\left(\left|[f^{-N}x,f^{-N}y]-f^{-N}x\right|+\left|[f^Nx,f^Ny]-f^Ny\right|\right).$$
Since $[\cdot,\cdot]$ is uniformly continuous by Lemma \ref{L:CanonicalCoor} and the compactness of $\L$, there exists $\d>0$ such that for $x,y\in \L$ with $|x-y|<\d$, the inequalities $|y-[x,y]|,|x-[x,y]|<\frac12$ always hold. For such $\d$, taking $\a=e^{-\l}$ is enough for the lemma.
\end{proof}

\noindent{\bf Proof of Proposition \ref{P:GibState}.}  Choose a Markov partition $\mc R$ for $\L$ of diameter $<\d$ as in Lemma \ref{L:QExpansive}. Then,  a transition matrix $A$ for $\mc R$ and $\pi:\Sigma_A\to \L$ as in subsection \ref{S:SymDyn} are well-defined. Set $\phi^*=\phi\circ \pi$. If $\underline{x},\underline{y}\in \Sigma_A$ satisfying $x_k=y_k$ for $k\in[-N,N]$, then one has
$$f^k\pi(\underline{x}),f^k\pi(\underline{y})\in R_{x_k}=R_{y_k}\text{ for all } k\in[-N,N].$$
Thus, by Lemma \ref{L:QExpansive}, we have that
$$|\pi(\underline{x})-\pi(\underline{y})|<\a^N, |\phi^*(\underline{x})-\phi^*(\underline{y})|\le a(\a^\t)^N\text{ and }\phi^*\in \mc F_A,$$
for some $a>0$ and $\t\in (0,1]$ which depend on $\phi$ only. For the definition of $\mc F_A$ we refer to Definition \ref{D:FA} in Appendix A.\\

Since $f|_\L$ is topologically mixing, by Lemma \ref{L:ShiftMixing}, so is $\s|_{\Sigma_A}$. Applying Theorem \ref{T:AppGibMea1}, there is a Gibbs measure $\mu_{\phi^*}$. Let $D_s=\pi^{-1}(\partial^s\mc R)$ and $D_u=\pi^{-1}(\partial^u\mc R)$. Note that $D_s$ and $D_u$ are closed subsets of $\Sigma_A$ and each of them is not equal to $\Sigma_A$. By Lemma \ref{L:RBInc}, we also have
$$\s D_s\subset D_s \text{ and }\s^{-1}D_u\subset D_u.$$
Since $\mu_{\phi^*}$ is $\s$-invariant, one has that $\mu_{\phi^*}(\s^n D_s)=\mu_{\phi^*}(D_s)$. Thus, also noting that $\s^{n+1}D_s\subset \s^n D_s$,
$$\mu_{\phi^*}(D_s)=\lim_{n\to\infty}\mu_{\phi^*}(\s^nD_s)=\mu_{\phi^*}\left(\cap_{n\ge 0}\s^nD_s\right).$$
Note that $\cap_{n\ge 0}\s^nD_s$ is a $\s$-invariant  closed subset of $\Sigma_A$. Since $\mu_{\phi^*}$ is mixing thus ergodic, and $\Sigma_A-\cap_{n\ge 0}\s^nD_s$, a nonempty open subset of $\Sigma_A$ has positive $\mu_{\phi^*}$-measure (this follows from the definition of Gibbs measure for which we refer to the paragraph after Theorem \ref{T:AppGibMea}). Thus,  one has that $\mu^*(D_s)=0$. The same argument shows that $\mu^*(D_u)=0$.

Set $$\mu_\phi=\pi^*\mu_{\phi^*} \text{ i.e. }\mu_\phi(E)=\mu_{\phi^*}(\pi^{-1}E)\text{ for all Borel set }E\subset\L.$$
By Lemma \ref{L:SymDynByMP}, we know that $\pi$ is one-to-one on the $\mu_{\phi^*}$-full measure set $\Sigma_A-\cup_{n\in\mb Z}\s^n(D_s\cup D_u)$. Then,  $\mu_\phi$ is $f$-invariant and the two automorphisms $\s:(\Sigma_A,\mu_{\phi^*})\circlearrowleft$ and $f:(\L,\mu)\circlearrowleft$ are conjugate. Therefore,  $$h_{\mu_\phi}(f)=h_{\mu_{\phi^*}}(\s)\text{ and }h_{\mu_\phi}(f)+\int_\L \phi d\mu_\phi=h_{\mu_{\phi^*}}(\s)+\int_{\Sigma_A} \phi^* d\mu_{\phi^*}.$$
By Theorem \ref{T:GibUnique} and \ref{T:Conjugacy}, we have that
$$h_{\mu_\phi}(f)+\int_\L \phi d\mu_\phi=h_{\mu_{\phi^*}}(\s)+\int_{\Sigma_A} \phi^* d\mu_{\phi^*}=P_\s(\phi^*)\ge P_f(\phi),$$
which together with Theorem \ref{T:VarPrin1} imply that $P_\s(\phi^*)=P_f(\phi)$. Thus,  $\mu_\phi$ is an equilibrium state of $\phi$.\\

To prove the uniqueness, we need the following lemma which is a well known fact. So we omit the proof for which we refer to Lemma 4.3 in \cite{Bowen}.
\begin{lemma}\label{L:PreImaOfPi}
For any $\mu\in \mc M_f(\L)$, there is a $\nu\in \mc M_\s(\Sigma_A)$ with $\pi^*\nu=\mu$.
\end{lemma}

Let $\mu\in \mc M_f(\L)$ be an equilibrium state of $\phi$ and choose one $\nu\in \mc M_\s(\phi^*)$ satisfying $\pi^*\nu=\mu$. Then, by definition of metric entropy, one has $h_\nu(\s)\ge h_\mu(f)$. Therefore,
$$h_\nu(\s)+\int_{\Sigma_A}\phi^* d\nu\ge h_\mu(f)+\int_\L \phi d\mu=P_f(\phi)=P_\s(\phi^*).$$
So $\nu$ is an equilibrium state of $\phi^*$, hence $\nu=\mu_{\phi^*}$ by Theorem \ref{T:GibUnique}. So one has that
$$\mu=\pi^*\mu_{\phi^*}=\mu_\phi.$$
The proof of Proposition \ref{P:GibState} is completed. \\

\section{Ruelle Inequality and Entropy Formula.}\label{S:AppRueEntr}
Ruelle's inequality builds the connection between the metric entropy and the positive Lyapunov exponents. The following result can be derived from \cite{T} straightforwardly. For results in finite dimensions, we refer to \cite{R78}.
\begin{theorem}\label{T:RuelleIne}
Let $f:X\to X$ be a $C^1$ map on a Banach space, $\L$ be an $f$-invariant compact subset of $X$, then for any $\mu\in \mc M_f(\L)$, one has that
$$h_\mu(f)\le \int_\L \sum \l_i^+m_i d\mu,$$
where $\l_i$ are Lyapunov exponents, $m_i$ is the multiplicity of $\l_i$ and $\l_i^+=\max\{\l_i,0\}$.
\end{theorem}

The following result is due to \cite{LSh}.

\begin{theorem}\label{T:EntropyFormula}
Let $f:X\to X$ be a $C^2$ map on a Hilbert space and $\L$ be an $f$-invariant compact subset of $X$, for any $\mu\in \mc M_f(\L)$,  if $\mu$ is an SRB measure, then
$$h_\mu(f)= \int_\L \sum \l_i^+m_i d\mu,$$
where $\l_i$ are Lyapunov exponents, $m_i$ is the multiplicity of $\l_i$ and $\l_i^+=\max\{\l_i,0\}$.
\end{theorem}

 \addcontentsline{toc}{section}{References}

\bibliographystyle{plain}

\begin{thebibliography}{99}


\bibitem{ABV} J. F. Alves, C. Bonatti and M. Viana.  SRB measures for partially hyperbolic systems whose central direction is mostly expanding. {\it Invent. Math.} {\bf 140} (2000), 351-398.

\bibitem{BP} L. Barreira  and Y. Pesin.
\newblock Nonuniform Hyperbolicity: Dynamics of Systems with Nonzero Lyapunov Exponents,
\newblock{\em Cambridge University Press}, 2007.

\bibitem{BY} M. Benedicks and L.-S. Young, Absolutely continuous invariant measures and random
perturbations for certain one-dimensional maps, {\em Ergod. Th. \& Dynam. Sys.,}
{\bf 12}(1992), 13每37.

\bibitem{Brin} M. Brin,
\newblock H\"older continuity of invariant distributions,
\newblock{\em Smooth Ergodic Theory and Its Applications, Proc. Symp. Pure Math.} {\bf 69}(2001), 99-101.

\bibitem{BV} C. Bonatti and M. Viana. SRB measures for partially hyperbolic systems whose central direction is mostly contracting. {\it Israel J. Math.} {\bf 115} (2000), 157-193.

\bibitem{Bowen74} R. Bowen,
\newblock Periodic points and measures for Axiom A diffeomorphisms,
\newblock{\em Trans. of AMS}, {\bf 154}(1974), 377-397.


\bibitem{Bowen} R. Bowen,
\newblock Equilibrium States and the Ergodic Theory of Anosov Diffeomorphisms,
\newblock{\em Springer lecture notes in mathematics}, {\bf 470}(1975).

\bibitem{BR} R. Bowen and D. Ruelle, The ergodic theory of Axiom A
ows, {\em Invent. Math.}, {\bf 29}(1975), 181-202.

\bibitem{ChLP} S-N. Chow,  X-B. Lin and K. Palmer,
\newblock A shadowing lemma with applications to semilinear parabolic equations,
\newblock {\em SIAM J. Math. Anal.}, {\bf 3}(1989), 547--557.

\bibitem{C} C. Castaing and M. Valadier,
\newblock Convex analysis and measurable multifunctions,
\newblock{\em Springer lecture notes in mathematics, } {\bf 580}, 1977.

\bibitem{CY} B. Cowieson and L.-S. Young. SRB measures as zero-noise limits.
{\it Ergod. Th. Dynam. Syst. } {\bf 25} (2005), 1115-1138.

\bibitem{ER} J.-P. Eckmann and D. Ruelle, Ergodic theory of chaos and strange attractors,
Rev. Mod. Phys. {\bf 57}(1985), 617-656.

\bibitem{GY}, J. Guckenheimer, M.  Wechselberger, and L-S.  Young,  Chaotic attractors of relaxation oscillators.
{\em Nonlinearity}, {\bf 19} (2006),  701每720.

\bibitem{Hale} J. K. Hale, Attractors and dynamics in partial differential equations. From finite to infinite dimensional dynamical systems (Cambridge, 1995), 85每112, NATO Sci. Ser. II Math. Phys. Chem., 19, Kluwer Acad. Publ., Dordrecht, 2001.

\bibitem{H} D. Henry,
\newblock Geometric Theory of Semilinear Parabolic Equations,
{\em Springer New York}, 1981.

\bibitem{HuangLu} W. Huang and K. Lu, Entropy, Chaos and  weak Horseshoe for Infinite Dimensional Random Dynamical Systems, submitted.

\bibitem{J} M. Jakobson, Absolutely continuous invariant measures for one-parameter families
of one-dimensional maps, {\em Comm. Math. Phys.}, {\bf 81} (1981), 39每88.

\bibitem{K} A. Katok,
\newblock Lyapunov exponents, entropy and periodic orbits for diffeomorphisms,
\newblock {\em Inst. Hautes $\acute E$tudes Sci. Publ. Math.}, {\bf 51}(1980),  137--173.

\bibitem{KS} A. Katok and J.-M. Strelcyn, Invariant manifolds, entropy and billiards, smooth
maps with singularities, {\em Springer Lecture Notes in Math.}, 1222 (1986).


\bibitem{La} S. Lang,
\newblock Real and functional analysis,
\newblock{\em Springer-Verlag}, 1993.

\bibitem{LY1} F. Ledrappier and L-S. Young,
\newblock The metric entropy of diffeomorphisms,
\newblock {\em  Ann. Math.}, {\bf 122}(1985), 509--574.



\bibitem{LL} Z. Lian and K. Lu,
\newblock Lyapunov exponents and invariant manifolds for random dynamical
systems in a Banach space,
\newblock {\em  Memoirs of AMS.}, {\bf 206}(2010), no.967

\bibitem{LSh} Z.  Li and L. Shu,  The metric entropy of random dynamical systems in a Hilbert space: characterization of invariant measures satisfying Pesin's entropy formula, {\it Discrete Contin. Dyn. Syst.}, {\bf 33} (2013), 4123每4155.

\bibitem{LY} Z. Lian and L-S. Young,
\newblock Lyapunov Exponents, Periodic Orbits and Horseshoes for Mappings of
Hilbert Spaces,
\newblock {\em Annales Henri Poincar\'{e}}, {\bf 12} (2011)1081-1108.


\bibitem{LiuL} P. Liu and K. Lu, A note on partially hyperbolic attractors: Entropy conjecture and SRB measures, {\em  Discrete and Continuous Dynamical Systems - Series A,} {\bf 35}(2015), 341 - 352.

\bibitem{LWY} K. Lu, Q. Wang, and L-S.  Young, Strange attractors for periodically forced parabolic equations, {\em Mem. Amer. Math. Soc.},  {\bf 224} (2013), no. 1054, 85 pages.


\bibitem{M} R. Ma$\tilde n\acute e$,
\newblock Lyapunov exponents and stable manifolds for compact
transformations,
\newblock {\em Springer lecture notes in mathematics},  {\bf 1007}(1983), 522-577.

\bibitem{O}  V. I. Oseledets,
\newblock A multiplicative ergodic theorem. Lyapunov characteristic
numbers for dynamical systems,
\newblock{\em Trans. Moscow Math. Soc.}, {\bf 19}(1968), 197-231.


\bibitem{P} P. Pesin,
\newblock Characteristic Lyapunov exponents, and smooth ergodic
theory,
\newblock {\em Russian Math. Surveys} {\bf 32}(1977), 55-144.

\bibitem{PS} Ya. B. Pesin and Ya. G. Sinai,  Gibbs measures for partially hyperbolic attractors, {\em Ergodic Theory Dynam. Systems,}  {\bf 2}(1982),  417每438.

\bibitem{Peterson} K. Peterson, Ergodic theory, {\em Cambridge Univ. Press}, Cambridge, 1983

\bibitem{Pu} C. Pugh and M. Shub,
\newblock Ergodic Attractor
\newblock {\em Transactions of the American Mathematical Society},
{\bf 312}(1989),   1-54.


\bibitem{Ruelle} D. Ruelle,  A measure associated with Axiom A attractors, {\em  Amer. J. Math.},  {\bf 98}(1976), 619每654.

\bibitem{R78} D. Ruelle,
\newblock An inequality of the entropy
of differentiable maps,
\newblock {\em Bol. Sc. Bra. Mat.}, {\bf 9}(1978),  83-87.

\bibitem{R79} D. Ruelle,
\newblock  Ergodic theory of differentiable dynamical systems,
\newblock{\em Publ.
Math., Inst. Hautes \' Etud. Sci.}, {\bf 50}(1979),  27-58.




\bibitem{R} D. Ruelle,
\newblock Characteristic exponents and invariant manifolds in Hilbert space,
\newblock {\em Ann. Math.}, \textbf{115}(1982), 243--290.

\bibitem{S} G. Sell G and Y. You Y,
\newblock Dynamics of evolutionary equations, {\em Springer New York}, 2010.

\bibitem{Sinai} Ya. G. Sinai, Gibbs measure in ergodic theory,  {\em Russ. Math. Surv.},  {\bf 27}(1972), 21每69.

\bibitem{Smale} S. Smale,
\newblock Differential Dynamical Systems,
\newblock{\em Bull. of AMS}, {\bf 73}(1963), 747-817.


\bibitem{Te} R. Temam,
\newblock Infinite Dimensional Dynamical Systems in Mechanics and Physics,
\newblock {\em Applied Math. Sc.}, Springer-Verlag {\bf 68} (1997).

\bibitem{T} P. Thieullen,
\newblock Asymptotically compact dynamic bundles, Lyapunov exponents, entropy,
dimension
\newblock{\em Ann. Inst. H. Poincar$\acute e$, Anal. Non lin$\acute
e$aire,} {\bf 4}(1987) no.1 49-97

\bibitem{WO} Q. Wang and W. Ott, Dissipative homoclinic loops of two-dimensional maps and strange attractors
with one direction of instability, {\em Comm. Pure Appl. Math.}, {\bf 64} (2011),  1439每1496.

\bibitem{Y} L-S. Young,
\newblock Ergodic theory of differentiable dynamical systems, "Real and Complex
Dynamics", {\em Ed. Branner and Hjorth, NATO ASI series, Kluwer Academic
Publishers} 1995,  293-336

\bibitem{Young2} L-S. Young, What are SRB measures, and which dynamical systems have them? Dedicated to David Ruelle and Yasha Sinai on the occasion of their 65th birthdays, {\em J. Statist. Phys.} {\bf 108}(2002),  733每754.

\bibitem{Young} L-S. Young,
\newblock Chaotic phenomena in three settings: large, noisy and out of equilibrium,
{\em Nonlinearity,} {\bf 21} (2008),  245每252.

\end{thebibliography}

 \end{document}